\documentclass[reqno,a4paper]{amsart}

\usepackage{array}
\usepackage[T1]{fontenc}
\usepackage{enumerate, amsmath, amsfonts, amssymb, amsthm, mathrsfs, mathabx}
\usepackage{bbm}
\usepackage[alphabetic, initials]{amsrefs}
\usepackage[lmargin=3.5cm,rmargin=3.5cm,top=2.9cm,bottom=2.9cm]{geometry}

\usepackage[all]{xy}
\usepackage{tikz}
\usetikzlibrary{arrows, decorations, shapes}
\usetikzlibrary{intersections}
\usetikzlibrary{calc}
\usepackage{tikz-cd}
\usepackage{colonequals}
\usetikzlibrary{arrows}
\usetikzlibrary{decorations.pathmorphing}
\usepackage[colorlinks, citecolor=blue]{hyperref}
\usepackage[noabbrev,capitalise]{cleveref}
\usepackage{multirow}
\usepackage{ulem}
\usepackage{enumitem}
\setlist[enumerate]{format=\normalfont}
\usepackage{caption}
\usepackage{cancel}

\usepackage{graphicx}
\usepackage{caption}
\usepackage{subcaption}

\usepackage{setspace}
\newcommand{\marginparstretch}{0.6}
\let\oldmarginpar\marginpar
\renewcommand\marginpar[1]{\-\oldmarginpar[\framebox{\setstretch{\marginparstretch}\begin{minipage}{\marginparwidth}{\raggedleft\tiny #1}\end{minipage}}]{\framebox{\setstretch{\marginparstretch}\begin{minipage}{\marginparwidth}{\raggedright\tiny #1}\end{minipage}}}}

 \newcommand{\arrow}[2][20]
 {
  \hspace{-5pt}
  \begin{tikzpicture}
   \node (A) at (0,0) {};
   \node (B) at (#1pt,0) {};
   \draw [#2] (A) -- (B);
  \end{tikzpicture}
  \hspace{-5pt}
 }

\newtheorem{theorem}{Theorem}[section]
\newtheorem{corollary}[theorem]{Corollary}
\newtheorem{proposition}[theorem]{Proposition}
\newtheorem{conjecture}[theorem]{Conjecture}
\newtheorem{lemma}[theorem]{Lemma}
\newtheorem*{theorem*}{Theorem}%[section]

\theoremstyle{definition}
\newtheorem{definition}[theorem]{Definition}
\newtheorem{example}[theorem]{Example}
\newtheorem{remark}[theorem]{Remark}

\newcommand{\tc}{\mathcal{T}}

\newcommand{\mc}{\mathcal{M}}
\newcommand{\uc}{\mathcal{U}}
\renewcommand{\sc}{\mathcal{S}}
\newcommand{\cc}{\mathcal{C}}
\newcommand{\dc}{\mathcal{D}}
\newcommand{\fc}{\mathcal{F}}
\newcommand{\gc}{\mathcal{G}}
\newcommand{\pc}{\mathcal{P}}
\newcommand{\Pf}{\mathfrak{P}}

\newcommand{\hf}{\mathfrak{h}}

\newcommand{\Cf}{\mathfrak{C}}
\newcommand{\ic}{\mathcal{I}}
\newcommand{\xc}{\mathcal{X}}
\newcommand{\yc}{\mathcal{Y}}
\newcommand{\hc}{\mathcal{H}}

\newcommand{\df}{\mathfrak{d}}
\newcommand{\ef}{\mathfrak{e}}
\newcommand{\Df}{\mathfrak{D}}
\newcommand{\codim}{\operatorname{codim}}

\newcommand{\modA}{\operatorname{mod}A}
\renewcommand{\mod}{\operatorname{mod}}

\newcommand{\Fac}{\operatorname{Fac}}
\newcommand{\Sub}{\operatorname{Sub}}
\newcommand{\Filt}{\operatorname{Filt}}

\newcommand{\End}{\operatorname{End}}
\newcommand{\Image}{\operatorname{Im}}

\newcommand{\Hom}{\operatorname{Hom}}
\newcommand{\Hasse}{\operatorname{\bf Hasse}}

\newcommand{\add}{\operatorname{add}}

\newcommand{\proj}{\operatorname{proj}}
\renewcommand{\star}{\operatorname{star}}
\newcommand{\tors}{\operatorname{ \bf tors}}
\newcommand{\ftors}{\operatorname{\bf tfree}}

\newcommand{\rep}[1]{%
  {%Make sure some parameters are set to the default
    \renewcommand{\arraystretch}{1}
    \setlength{\extrarowheight}{0pt}
    \tiny%
    \begin{matrix}%
      #1%
    \end{matrix}%
  }%
}

\title{Scattering diagrams for Artin algebras}
\author[Treffinger]{Hipolito Treffinger}
\address{Universidad de Buenos Aires. Facultad de Ciencias Exactas y Naturales. Departamento de Matemática. Buenos Aires, Argentina.
CONICET - Universidad de Buenos Aires. Instituto de Investigaciones Matemáticas “Luis A. Santaló”  (IMAS). Buenos Aires, Argentina.}
\email{htreffinger@dm.uba.ar}
\usepackage{url}
\begin{document}

\begin{abstract}
    For an arbitrary Artin algebra $A$, we construct a minimal and consistent scattering diagram by approximating its module category $\mathrm{mod}\,A$ using the subcategories $(\mathrm{mod}\,A)_\ell$ of modules of length at most $\ell \in \mathbb{N}$.
    We prove that each subcategory $(\operatorname{mod}A)_\ell$ possesses a well-behaved lattice of torsion classes, a finite wall-and-chamber structure $\mathfrak{D}_\ell(A)$ and an associated picture group $G_\ell(A)$ with a natural categorical interpretation.
    Using these properties, we build a finite, minimal, consistent scattering diagram for each $\ell \in \mathbb{N}$.
    Passing to the inverse limit, we establish the existence of a minimal consistent scattering diagram for $A$.
    In particular, when $A$ is a finite-dimensional algebra over $\mathbb{C}$, our inverse limit construction is canonically isomorphic to Bridgeland's stability scattering diagram.
\end{abstract}

\maketitle
\setcounter{tocdepth}{1}
\tableofcontents

%%%%%%%%%%%
\section{Introduction}
%%%%%%%%%%%
\subsection*{Motivation and main result}
Scattering diagrams were introduced by Kontsevich and Soibelman \cite{KontsevichSoibelman} in homological mirror symmetry as a combinatorial tool to construct the mirror of a log-Calabi-Yau variety.  
Later, Gross, Hacking, Keel, and Kontsevich \cite{GHKK} introduced cluster scattering diagrams, providing a foundational geometric framework for understanding the cluster algebras introduced by Fomin and Zelevinsky \cite{FZ}.
Bridgeland \cite{BridgelandScat} subsequently demonstrated that for skew-symmetric cluster algebras with an acyclic initial seed, the scattering diagram can be constructed via the representation theory of the quiver path algebra using stability conditions and motivic Hall algebras.
Specifically, Bridgeland proved the existence of a unique  minimal consistent scattering diagram, often called stability scattering diagram, for every finite-dimensional $\mathbb{C}$-algebra, showing it is isomorphic to the cluster scattering diagramin the hereditary case. 
This result was later generalised to Jacobian algebras with a non-degenerate potential by Mou \cite{Mou}.
Since then, scattering diagrams have attracted significant attention at the interface of representation theory, cluster theory, and geometry (see, e.g., \cite{BFMN, BurcroffLeeMou, LabardiniMou, Mou2, PlamondonYurikusa, Reading, ReadingStella}).

Despite these developments, the construction of stability scattering diagrams has relied heavily on motivic techniques and the geometric properties of finite-dimensional $\mathbb{C}$-algebras, leaving the general case of Artin algebras open. 
The main result of this paper resolves this question in full generality:

\begin{theorem}\label{thm:mainIntro}[\cref{thm:scatteringdiagrams,thm:isomorphism}]
For any Artin algebra $A$ there is a minimal consistent scattering diagram $(\Df(A), \Phi : \Df^1(A) \to G(A))$. 
Moreover, this scattering diagram is isomorphic to the stability scattering diagram if $A=\mathbb{C}Q/I$, where $Q$ is a finite quiver and $I$ is an admissible ideal of $\mathbb{C}Q$.
\end{theorem}

The main obstacle in the construction is that the wall-and-chamber structure of an arbitrary Artin algebra may contain infinitely many walls, so the wall-crossing product associated with a generic path is not defined a priori, see below.
Our method relies on the bounded-length subcategories $(\modA)_\ell$ of $\modA$, to which we associate a finite minimal consistent scattering diagram for each $\ell$ using lattices of torsion classes. 
By taking the inverse limit over $\ell$, we obtain a scattering diagram for arbitrary $A$.
The torsion-theoretic framework we adopt avoids motivic Hall algebra techniques, offering a categorical and combinatorial construction that is both transparent and widely applicable.

\medskip
The significance of extending scattering diagrams to arbitrary Artin algebras stems from the profound impact of cluster theory on modern representation theory. 
From its inception, cluster theory has driven the development of fundamental categorical structures \cite{CCS, BMRRT, IngallsThomas, Am, GLS1, KellerMGS, JKS}. 
Beyond quiver representations, this influence opened major new directions in representation theory, including the K\"onig--Yang correspondences \cite{KonigYang}, higher Auslander--Reiten theory \cite{Iya}, and $\tau$-tilting theory \cite{AIR, DF}.

In this setting, the appearance of scattering diagrams in cluster theory \cite{GHKK, BridgelandScat} quickly found applications in representation theory, giving a new perspective on the study of stability conditions \cite{King}. 
Even before the introduction of scattering diagrams, the study of quiver moduli has long been, and continues to be, a powerful tool to understand identities arising in algebraic geometry, such as quantum dilogarithm identities \cite{Reineke, KellerMGS}. 
The introduction of scattering diagrams also shifted attention toward the global geometry of the space of stability conditions, extending the earlier work of \cite{ITW, BHIT}  for hereditary algebras that was based on the notion of semi-invariants of quivers \cite{Schofield}.

For an Artin algebra $A$, the wall-and-chamber structure $\Df(A)$ encodes a remarkable amount of representation-theoretic data. 
In particular, $\Df(A)$ recovers important aspects of the $\tau$-tilting theory of $A$ via the bijection between the chambers of $\Df(A)$ and the support $\tau$-tilting modules \cite{BST, Asai}, where wall-crossing corresponds to mutations of support $\tau$-tilting along bricks \cite{Treffinger_c-vectors, AsaiSemibricks}. 
The translation between the homological language of $\tau$-tilting theory and stability conditions was given by $g$-vectors \cite{AIR, DIJ} and $c$-vectors \cite{Fu}.
In many cases, the wall-and-chamber structure $\Df(A)$ also provides a natural realisation of the duality between the Grothendieck group $K_0(A)$ of the module category $\modA$ and the Grothendieck group $K_0(K^{[-1,0]}(\proj A))$ of the extriangulated category $K^{[-1,0]}(\proj A)$ of two-term complexes of projectives up to homotopy \cite{NP, GNP}, see also \cite{AIR, DF, IOTW2, DIJ, PlamondonYurikusa, Garcia}.

Our construction focuses exclusively on the category $\modA$ and its bounded-length subcategories. 
The identification of the categorical structures that interact with $(\mod A)\ell$ through the wall-and-chamber structure $\Df\ell(A)$ is left for future research.
\medskip

%%%%%%%%%%%
\subsection*{Basics on scattering diagrams and strategy of the proof}
In this work, we give a constructive proof that every Artin algebra determines a minimal consistent scattering diagram. 
We now briefly recall the basic notions needed for the construction.
Formal definitions for scattering diagrams appear in \cref{sec:torsion scattering diagram}; see \cite{Nakanishi} for a comprehensive description.

A scattering diagram is defined from the following data:
A lattice $N$ of finite rank $n$, its dual lattice $M$ and a skew-symmetric form $\{-,-\}: N\times N \to \mathbb{Q}$. 
Using these data, one constructs a group $G$, whose precise definition depends on the context. 

A scattering diagram is a pair $(\Df, \Phi:\Df^1\to G)$, where $\Df$ is a cone complex in $M\otimes\mathbb{R} \cong \mathbb{R}^n$.
The map $\Phi: \Df^1 \to G$ assigns to every wall $\df$ an element $g_\df \in G$.
The scattering diagram $(\Df, \Phi:\Df^1\to G)$ is said to be minimal if $\Phi(\df)\neq 1_G$ for every wall $\df$.

A $\Df$-generic path is a smooth curve $\gamma: [0,1] \to \mathbb{R}^n$ starting and ending in chambers (that is, codimension-zero cones), avoiding cones of $\Df$ of codimension $2$ or higher, and crossing walls transversely. 
For every $\Df$-generic path $\gamma$, one should be able to define an element $g_\gamma$ which is calculated as the ordered product of the elements $g_\df$ associated to the walls $\df$ crossed by $\gamma$.
Then the scattering diagram $(\Df, \Phi:\Df^1\to G)$ is consistent if the element $g_\gamma$ is completely determined by $\gamma(0)$ and $\gamma(1)$.

The representation theory of Artin algebras naturally provides candidates for some of the ingredients needed to construct scattering diagrams:
\begin{itemize}
\item The lattice $N$  is the Grothendieck group $K_0(A)$ of $A$;
\item the realisation $M \otimes \mathbb{R} \cong \mathbb{R}^n$ of the dual lattice $M$ is the set of all stability conditions defined over $\modA$; and
\item the cone complex $\Df$ is the wall-and-chamber structure of $A$.
\end{itemize}

The main idea of the construction stems from an observation originated in \cite{BKT}.
Every stability condition $v\in \mathbb{R}^n$ induces two torsion pairs $(\tc_v, \overline{\fc}_v)$ and $(\overline{\tc}_v, \fc_v)$ in $\modA$. 
Moreover, $\tc_v = \overline{\tc}_v$ if $v$ belongs to a chamber.
Suppose that one can construct a group $G(A)$ whose elements correspond to intervals of torsion classes.
Then one can define a scattering diagram $(\Df(A), \Phi: \Df(A) \to G(A))$ in such a way that the element associated to a generic path depends only on the torsion classes determined by its endpoints.
Consequently, the scattering diagram is consistent.

Although the theory of torsion classes in $\modA$ allows for the proof of consistency for certain generic paths $\gamma$ even when it crosses infinitely many walls (see \cref{sec:groupoids}), a new approach was needed in order to consider arbitrary generic paths.
The principal obstruction arises when $\Df$ is an infinite cone complex: a given path $\gamma$ might cross infinitely many walls and the ordered product $g_\gamma$ might not be well-defined.
The classical strategy to avoid this problem is to show that the scattering diagram $(\Df, \Phi:\Df\to G)$ is the inverse limit of a family 
\(
\{(\Df_\ell, \Phi_\ell:\Df_\ell\to G_\ell) \mid \ell \in \mathbb{N}\}
\)
of consistent scattering diagrams such that $\Df_\ell$ is finite for all $\ell \in \mathbb{N}$.

A key feature of the category $\modA$ is that it is a length category.
By the Jordan-H\"older theorem, every module admits a finite filtration whose factors are simple modules.
If $(\modA)_\ell$ denotes the full subcategory consisting of modules of length at most $\ell$, then
\[
\modA = \bigcup_{\ell\in\mathbb{N}} (\modA)_\ell.
\]
This allows us to construct finite scattering diagrams whose inverse limit recovers the desired scattering diagram of the algebra.

%\begin{theorem}[\cref{thm:boundedscatteringdiagrams}]
%Let $A$ be an Artin algebra.
%Then there is a minimal consistent scattering diagram $(\Df_\ell(A),\Phi^\ell_A: \Df_\ell^1(A) \to G_\ell(A))$ associated to $(\modA)_\ell$ which is finite for every $\ell \in \mathbb{N}$.
%\end{theorem}
%We then obtain the scattering diagram of Theorem~\ref{thm:mainIntro} as the inverse limit of the scattering diagrams 
%\[
%\{(\Df_\ell(A),\Phi^\ell_A: \Df_\ell^1(A) \to G_\ell(A))\mid \ell \in \mathbb{N}\}.
%\]

To be able to construct the scattering diagrams for $(\modA)_\ell$, we make a detailed study of torsion classes, stability conditions, and picture groups in this category. 
Several of the results obtained along the way are of independent interest, as we explain in the following subsections.

%%%%%%%%%%%
\subsection*{The category of modules of bounded length and torsion classes}
Let $\tors A$ be the lattice of torsion classes in $\modA$.
For every $\ell \in \mathbb N$, there is a natural equivalence relation $\sim_\ell$ defined on $\tors A$.
Given $\tc, \tc' \in \tors A$, we write $\tc \sim_\ell \tc'$ if and only if $\tc\cap (\modA)_\ell = \tc'\cap(\modA)_\ell$. 
The set of equivalence classes of $\sim_\ell$ is denoted by $\tors_\ell A$ and the equivalence class of $\tc$ is denoted by $\tc_\ell$.
By abuse of notation, we also denote by $\tc_\ell$ the corresponding subcategory $\tc_\ell := \tc\cap (\modA)_\ell$.
Similarly, one can define $\fc_\ell$ for a torsion free class $\fc$ in $\modA$.

Even if $(\modA)_\ell$ is not an additive category, the subcategories $\tc_\ell$ retain the essential properties of torsion classes in $\modA$, see \cref{thm: torsion pairs of bounded length}.
%\begin{theorem}[\cref{thm: torsion pairs of bounded length}]
%    Let $(\tc, \fc)$ be a torsion pair in $\modA$. 
%    Then the following holds. 
%    \begin{enumerate}
%        \item \( \fc_\ell= \left\{M \in (\modA)_\ell \mid \Hom_A(T,M)=0 \, \forall \, T\in \tc_\ell\right\}. \)
%        \item \(        \tc_\ell= \left\{M \in (\modA)_\ell \mid \Hom_A(M,F)=0 \, \forall \, F\in \fc_\ell\right\}.  \)
%        \item For every $M \in (\modA)_\ell$ there is a short exact sequence 
%        \[
%        0 \to tM \to M \to fM \to 0
%        \]
%        where $tM \in \tc_\ell$ and $fM\in\fc_\ell$.
%    \end{enumerate}
%\end{theorem}
%
Moreover, the set $\tors_\ell A$ also shares many of the structural properties of $\tors A$.

\begin{theorem}[\cref{prop:completelattice,prop: semidistributive ell,prop:minimal extending}]
The set $\tors_\ell A$ is a complete semidistributive lattice and it admits a labelling of the edges of its Hasse quiver whose labels corresponds to the set $Br_\ell(A)$ of all bricks of length at most $\ell$.
\end{theorem}

We conclude this subsection with two remarks.
First, $\tors_\ell A$ is always a quotient of $\tors A$ as lattices and   
\[
     \underleftarrow\lim \tors_\ell A =\tors A.
    \]
Moreover, it follows from \cite{ST} that $\tors A\cong\tors_\ell A$ for some $\ell \in \mathbb{N}$ if and only if $A$ is $\tau$-tilting finite in the sense of \cite{DIJ}.

Second, the lattices $\tors A$ and $\tors_\ell A$ exhibit strikingly different behaviour whenever $\tors A$ is infinite.
In this situation, it was shown in \cite{DIJ} that the Hasse quiver of $\tors A$ is not connected.
By contrast, every example we have computed suggests that $\tors_\ell A$ is connected, and we conjecture that this always holds.
This naturally raises the question whether semibricks \cite{AsaiSemibricks} in $(\modA)_\ell$ determine all torsion classes $\tc_\ell$ in $(\modA)_\ell$, see \cref{cor:semibricks}.
It is worth noting that most of the tools used to study $\tors A$, such as $\tau$-tilting theory \cite{AIR, DF}, cannot be applied directly in this setting, since they rely heavily on the additive structure of $\modA$.

%%%%%%%%%%%%%%%
\subsection*{Wall-and-chamber structures for $(\modA)_\ell$}
The next ingredient is to interpret the lattice $\tors_\ell A$ geometrically via stability conditions.
Let $K_0(A)$ be the Grothendieck group of $\modA$.
Recall that a stability condition, in the sense of King \cite{King}, is simply a linear map $\theta : K_0(A) \to \mathbb{R}$. 
Identifying $\Hom_\mathbb{Z}(K_0(A), \mathbb{R}) \cong \mathbb{R}^n$, we regard a stability condition $\theta$ as the vector $v\in\mathbb R^n$ satisfying $\theta([M]) = \langle v, [M] \rangle$ for every $M \in \modA$, where \mbox{\(\langle -,- \rangle: \mathbb{R}^n\times\mathbb{R}^n \to \mathbb{R}\)} is the canonical inner product in $\mathbb{R}^n$.
We say that a module $M \in \modA$ is $v$-semistable if $\langle v, [M]\rangle = 0 $ and $\langle v, [L]\rangle \leq 0 $ for every submodule $L$ of $M$.
The subcategory of $\modA$ of all $v$-semistable modules is denoted by $\mod^{ss}_vA$. 
Similarly, the subcategory of $(\modA)_\ell$ of all $v$-semistable modules is denoted by $(\mod^{ss}_vA)_\ell$. 

Given a module $M\in \modA$, the stability locus $\dc(M)$ is defined as 
$$\dc(M) = \{v\in \mathbb{R}^n \mid M \text{ is $v$-semistable}\}.$$
The set $\dc(M)$ is said to be a wall if it is of codimension $1$ in $\mathbb{R}^n$.
Following \cite{BridgelandScat, BST}, we define the wall-and-chamber $\Df_\ell(A)$ to consists of all walls $\dc(M)$ with $M\in (\modA)_\ell$ together with the connected components of the set
\[
\mathbb{R}^n \setminus \bigcup_{0 \neq M \in (\modA)_\ell} \dc(M),
\]
which we call chambers.
One of the key features of $\Df_\ell(A)$ is that it is always finite, regardless the algebra $A$ and the integer $\ell$.

Our next objective is to understand the geometry of $\Df_\ell(A)$ in terms of the bounded-length torsion classes introduced in the previous subsection.

Following \cite{Asai}, we say that two stability conditions are $\ell$-TF-equivalent if \mbox{$(\tc_v)_\ell = (\tc_{v'})_\ell$} and $(\overline{\tc}_v)_\ell = (\overline{\tc}_{v'})_\ell$. 
The $\ell$-TF-equivalence class of a stability condition $v\in \mathbb{R}^n$ is denoted by $\df^\ell_v$.
This induces a partition of $\Df_\ell(A)$ into finitely many $\ell$-TF-equivalence classes:
\[
\Df_\ell(A) = \bigcup_{v \in \mathbb{R}^n} \df_v^\ell.
\]
We describe properties of $\ell$-TF-equivalence classes in $\mathbb{R}^n$ as follows, compare with \cite[Theorem~2.17]{Asai}.\footnote{We note that recently Asai and Iyama used the same techniques to build a wall-and-chamber structure and its dual polytope for a single module $M\in \modA$ in \cite{AsaiIyama}.}

\begin{theorem}[\cref{thm:bounded l-TF}, \cref{cor:finite cone complex for l}]
Let $v, v' \in \mathbb{R}^n$ be two different stability conditions. 
Then $v$ and $v'$ are $\ell$-TF-equivalent if and only if there is no brick $B \in (\modA)_\ell$ such that the intersection between the segment $[v,v']$ and $\dc(B)$ consists of a single point. 
Moreover, every $\ell$-TF-equivalence class $\df_v^\ell$ is a relatively open cone in $\overline{\df_v^\ell}$. 
\end{theorem}

As a consequence of this result we can show that the wall-and-chamber structure $\Df_\ell(A)$ is a geometric realisation for the poset $\tors_\ell^{ss}A$, see \cref{prop: tors to chambers,cor: maximal inclusions tors^ss}.
%%%%%%%%%%%
\subsection*{Picture groups for Artin algebras}
Having constructed the support $\Df_\ell(A)$ of the scattering diagram $(\Df_\ell(A), \Phi_\ell: \Df_\ell(A) \to G_\ell(A))$, we now turn our attention to the group $G_\ell(A)$, which we chose to be the picture group associated to $\Df_\ell(A)$. 

The notion of group pictures was introduced by Igusa in his PhD thesis \cite{IgusaPhD} to study the $K$-theory of group algebras.
Roughly speaking, given a group $G$, each picture encodes the generators and relations of deformations of $G$. 
In particular, every picture has an associated picture group. 
The story of picture groups in representation theory starts with \cite{ITW, IT}, where the authors use the semi-invariant picture associated to every hereditary algebra of Dynkin type $H$ to produce a corresponding picture group $G(H)$. 
The authors also define a cluster morphism category $\Cf(H)$ associated to $H$ and they show that $\Cf(H)$ is a \mbox{CAT$(0)$-category} with fundamental group  $G(H)$, which implies that $G(H)$ is $K(\pi, 1)$.
For a detailed account of the history of picture groups and related concepts, see \cite{IT2}.

The development of $\tau$-tilting theory \cite{AIR} led to the definition of a picture group $G(A)$ for every $\tau$-tilting finite algebra $A$ \cite{HansonIgusa} by Hanson and Igusa using the wall-and-chamber structure $\Df(A)$ of $A$.  

The definition of a $\tau$-cluster morphism category $\Cf(A)$ was introduced in the work of Buan, Hanson and Marsh \cite{BuanHanson, BuanMarsh}, which relies heavily on the homological properties of $\modA$. 
Later on, Schroll, Tattar, Treffinger and Williams \cite{STTW} showed that the $\tau$-cluster morphism category $\Cf(A)$ can also be defined using the combinatorial properties of the wall-and-chamber structure $\Df(A)$ of $A$. 
Kaipel later extended this approach in \cite{Kaipel}, showing that there is a category $\Cf(\Sigma, \Pf)$ associated to every finite fan $\Sigma$ admitting an admissible partition $\Pf$ of its cones. 
Moreover, Kaipel's construction associates a picture group $G(\Sigma,\Pf)$ to every partitioned fan $(\Sigma,\Pf)$ with a poset structure.

To state our next result, we first observe that $\Df_\ell(A)$ can be seen as a fan 
%\[
%\Sigma_\ell(A) := \{\sigma_v:=\overline{\df_v^\ell} \mid v\in \mathbb{R}^n\}.
%\]
$\Sigma_\ell(A)$, which is not always a simplicial, see \cref{ex: no li chamber}.
This fan can be partitioned in a natural way, by considering the equivalence relation $\sim_{ss}$ defined by $\sigma_v \sim_{ss}\sigma_{v'}$ if and only if $(\mod^{ss}_vA)_\ell = (\mod^{ss}_{v'}A)_\ell$.
This induces a nontrivial partition $\Pf_{\ell}(A)$.

\begin{theorem}
Let $A$ be an Artin algebra, $\ell \in \mathbb{N}$.
Then the pair $(\Sigma_\ell(A), \Pf_\ell(A))$ is a partitioned fan in the sense of Kaipel. 
In particular there is a $\tau$-cluster morphism category $\Cf_\ell(A)$ and a picture group $G_\ell(A)$ for every $A$ and $\ell$. 
\end{theorem}

An important consequence of the previous result is that we can define the picture group $G(A)$ for every Artin algebra as 
\[
G(A) := \underleftarrow{\lim}\, G_\ell(A).
\]
Note that this definition recovers the classical definition of picture groups for $\tau$-tilting finite algebras. 

The main drawback of this definition for the picture group $G(A)$ is that, although geometrically very explicit, it is difficult to understand and control the product of elements in the group. 
In order to avoid this problem, we study in \cref{sec:groupoids} a groupoid $\gc_\ell^{ss}(A)$ associated to the lattice $\tors^{ss}_{\ell}A$. 
Using this groupoid, we obtain the following presentation of $G_\ell(A)$, where $\Df^0_\ell(A)$ and $\Df^1_\ell(A)$ correspond to the cones of $\Df_\ell(A)$ of codimension $0$ and $1$, respectively.

\begin{proposition}[\cref{prop:categorical interpretation}]
The picture group $G_\ell(A)$ is isomorphic to the group of subcategories of $(\modA)_\ell$ generated by 
\[
\{(\tc_{\df})_\ell \mid \df \in \Df^{0}_\ell(A)\} \text{ and } \{(\mod_{\df'}^{ss}A)_\ell \mid \df' \in \Df^1_\ell(A)\}  
\]
with the $*$-product and with identity element corresponding to the subcategory $\{0\}\subset (\modA)_\ell$.
\end{proposition}

Combining our categorical description of $G_\ell(A)$ with the results of \cite{BKH}, we show that for every $\tau$-tilting finite algebra $A$ there is a faithful functor $\Gamma_A : \Cf(A) \to G$ where $G$ is viewed as a category with a single object, see \cref{thm:faithfulfunctor}. 
This provides an important step towards the proof of the conjecture that the picture group $G(A)$ of any $\tau$-tilting finite algebra $A$ is a $K(\pi, 1)$-group. 

%%%%%%%%%%%
\subsection*{Scattering diagrams}
We are now ready to define, for every Artin algebra $A$ and every $\ell \in \mathbb{N}$ the scattering diagram 
$$(\Df_\ell(A), \Phi_\ell: \Df^1_\ell(A) \to G_\ell(A))$$ 
where $\Phi_\ell: \Df^{1}_\ell(A) \to G_\ell(A)$ assigns to wach wall $\df\in \Df^1_\ell(A)$ the element $(\mod_{\df}^{ss}A)_\ell\in G_\ell(A)$.
As it turns out, the categorical description of the picture group $G_\ell(A)$ obtained above allows us to compute explicitly the element $g_\gamma$ associated to every generic path $\gamma: [0,1] \to \Df_\ell(A)$. 
Consequently, we obtain the following theorem.

\begin{theorem}[\cref{thm:boundedscatteringdiagrams}]
Let $A$ be an Artin algebra and $\ell$ be a positive integer. 
Then the pair $(\Df_\ell(A),\Phi^\ell_A: \Df_\ell^1(A) \to G_\ell(A))$ is a minimal consistent scattering diagram.
\end{theorem}

Combining the finite scattering diagrams through inverse limits yields the scattering diagram of an arbitrary Artin algebra.

Finally, we compare the scattering diagram $(\Df_\ell(A),\Phi^\ell_A: \Df_\ell^1(A) \to G_\ell(A))$, which we call the torsion scattering diagram, with the scattering diagram built by Bridgeland in \cite{BridgelandScat} for finite dimensional algebras of the form $A = \mathbb{C}Q/I$, usually known as the stability scattering diagram. 
%
%The stability scattering diagram is the pair $(\Df(A),\hat{\Phi}_A: \Df^1(A) \to \hat{H}(A))$
%whose support is the same wall-and-chamber structure $\Df(A)$.
%The map $\hat{\Phi}_A: \Df^1(A) \to \hat{H}(A)$ assigns to each wall $\df$ the element $\hat{\Phi}_{A}(\df) = [\mod_\df^{ss}A \to \mc]$ which is an element to the group of units of the motivic Hall algebra $\hat{H}(A)$.
%Our last result is the following. 
%
%\begin{theorem}[\cref{thm:isomorphism}]
%Let $A=\mathbb{C}Q/I$ be a finite-dimensional algebra. 
%Then the scattering diagram $(\Df_\ell(A),\Phi^\ell_A: \Df_\ell^1(A) \to G_\ell(A))$ is isomorphic to $(\Df(A),\hat{\Phi}_A: \Df^1(A) \to \hat{H}(A))$.
%\end{theorem}
More concretely, we prove in \cref{thm:isomorphism} that our construction recovers Bridgeland's scattering diagram whenever the latter is defined.

%%%%%%%%%%%
\subsection*{Structure of the paper.}
In \cref{sec:background}, we establish the framework of the paper, introduce the basic notation, and recall the necessary preliminary results.
\Cref{sec:l-torsion classes} is devoted to the study of the interplay between torsion classes and the category $(\modA)_\ell$ of modules of length at most $\ell$.
In \cref{sec:conecomplex} we introduce and study the wall-and-chamber structure $\Df_\ell(A)$ associated to the category $(\modA)_\ell$.
In \cref{sec:groupoids} we consider a groupoid $\gc(\pc)$ associated to a poset $\pc$ and  give a categorical interpretation of the composition in this groupoid for the lattices $\tors A$, $\tors_\ell A$, $\tors^{ss}A$ and $\tors_\ell^{ss}A$.
In \cref{sec:picture group}, we return to wall-and-chamber structures and show that they naturally give rise to a $\tau$-cluster morphism category $\Cf_\ell(A)$ and a picture group $G_\ell(A)$.
Building on the previous sections, in \cref{sec:torsion scattering diagram} contains the construction of the scattering diagrams $(\Df_\ell(A),\Phi^\ell_A: \Df_\ell^1(A) \to G_\ell(A))$ and $(\Df(A),\Phi_A: \Df^1(A) \to G(A))$.
Finally, in \cref{sec:isomorphism}, we prove that our scattering diagram is isomorphic to the stability scattering diagram introduced by Bridgeland in \cite{BridgelandScat}.

%%%%%%%%%%%%%
\subsection*{Acknowledgements}
This work is the product of countless discussions over several years with many of my closest colleagues, who have patiently listened to my often vague and unclear ideas.
They are the ones who pointed me in the right direction whenever I got stuck.
In particular, I thank Sibylle Schroll, Baptiste Rognerud, Aran Tattar, Alfredo Nájera Chávez, Johanne Haugland, and many others.
I am especially grateful to Maximilian Kaipel, who generously took the time to read an earlier version of this manuscript and provide valuable feedback.
I also thank Vicente Treffinger Ghergo for his companionship and encouragement during the final stages of the writing this paper.

%%%%%%%%%%%%%
\subsection*{Funding}
This work was completed without any funding and in spite of Argentina’s current science policies, which severely hinder basic research in the country.

%%%%%%%%%%%%%
\subsection*{AI disclosure}
The use of artificial intelligence in the preparation of this paper was limited to improving grammar and readability.

%%%%%%%%%%%%%%%%%%
\section{Setting and background.}\label{sec:background}

\subsection{Framework and notation}
In this paper $A$ refers to an Artin algebra. 
We work on the category $\modA$ of finitely generated (right) $A$-modules. It is well known that $\modA$ is Krull--Schimidt and idempotent complete.
Since we are interested in the study of the module category $\modA$ of $A$, we assume without loss of generality that $A$ is basic.

The number of isomorphism classes of simple $A$-modules is finite and we denote it by $n$. 
We fix a set $\{S(1), \dots, S(n)\}$ of representatives of the isomorphism classes of simple $A$-modules.
The Jordan--Hölder theorem states that for every object $M \in \modA$, there exists a composition series
\[
0= M_0 \subset M_1 \subset \dots\subset M_{\ell -1} \subset M_\ell = M
\]
such that each subquotient $M_i/M_{i-1}$ is simple for all $1 \leq i \leq \ell$. 
Moreover, any two composition series have the same length and the same multiset $\{M_1/M_0, \dots, M_\ell/M_{\ell-1}\}$ of simple modules. 
We denote by $\lg(M)$ the length of any such composition series.

Let $\ell$ be a nonnegative integer.
We denote by $(\modA)_\ell$ the full subcategory of $\modA$ consisting of all $A$-modules of length at most $\ell$:
\[
(\modA)_\ell = \{M \in \modA \mid \lg(M)\leq \ell\}.
\]
From this definition, it follows immediately that
\[
\modA = \bigcup_{\ell\in \mathbb{N}} (\modA)_\ell.
\]
Every subcategory $\xc \subset \modA$ is assumed to be full and closed under isomorphisms.
Given a subcategory $\xc$ of $\modA$, we say that an object $M$ is filtered by $\xc$ if there exists a chain of subobjects 
\[
0 = M_0 \subseteq M_1 \subseteq \dots \subseteq M_{r-1} \subseteq M_r=M
\]
such that $M_i/M_{i-1} \in \xc$ for every $1 \leq i \leq r$. 
We denote by $\Filt(\xc)$ the subcategory of $\modA$ consisting of all the objects in $\modA$ that are filtered by $\xc$. 
The category $\add \xc$ consists of all the direct summands of the finite direct sums of objects in $\xc$.
We denote by $\Fac \xc$ the category of objects $M$ in $\modA$ for which there is an object $X \in \xc$ and an epimorphism $p: X \to M$. 
Similarly, the category $\Sub\xc$ consists of all the objects $M\in \modA$ for which there is an object $X\in \xc$ and a monomorphism $i: M \to X$.
If $\xc = \add(X)$ for some object $X\in \modA$, then by $\Fac X$ and $\Sub X$ we mean $\Fac (\add (X))$ and $\Sub (\add(X))$, respectively.

An object $B\in \modA$ is called a brick if its endomorphism algebra $\End_A(B)$ is a division ring. 
The set of all bricks in $\modA$ is denoted by Br$(A)$.
A semibrick $\sc$ is a set of objects $\sc=\{B_i \in \modA \mid i \in I\}$ in $\modA$ indexed by an index set $I$ such that $\Hom_A(B_i, B_j)$ is either $0$ if $i \neq j$ or a division ring if $i=j$.
It is clear from the definition that every object in a semibrick is necessarily a brick. 

%%%%%%%%%
\subsection{Notions in lattice theory}
A lattice $(\pc, \leq)$ with meet $\wedge$ and join $\vee$ is said to be complete if arbitrary subsets of $\pc$ have joins and meets.
The lattice $(\pc, \leq)$ is join semidistributive if whenever $x\vee y = x\vee z$, then it satisfies  $x\vee(y\wedge z)= x\vee y$.
A complete lattice is said to be completely join semidistributive if for every $x\in\pc$ and every set $S\subseteq \pc$ such that $x\vee y = x\vee y'$ for every pair $y,y'\in S$, then $x\vee \left(\bigwedge_{z\in S} z\right)=x\vee y$ for some $y\in S$. 
Dually, a lattice is said to be meet semidistributive if whenever $x\wedge y = x\wedge z$, then it satisfies  $x\wedge(y\vee z)= x\wedge y$.
A complete lattice is said to be completely meet semidistributive if for every $x\in\pc$ and every set $S\subseteq \pc$ such that $x\wedge y = x\wedge y'$ for every pair $y,y'\in S$, then $x\wedge \left(\bigvee_{z\in S} z\right)=x\wedge y$ for some $y\in S$. 
A lattice is said to be completely semidistributive if
if it is both completely join semidistributive and completely meet semidistributive.

Given a lattice $(\pc, \leq)$, an element $x\in \pc$ is said to be completely join irreducible if $x\neq \bigvee_{y<x} y$.
It has been shown that every semidistributive lattice admits a labelling of its edges by join-irreducible elements, see \cite[Proposition 9.1]{Thomas}. 

%%%%%%%%%
\subsection{Torsion classes in module categories}
A torsion pair in $\modA$ is a pair of subcategories $(\tc, \fc)$ such that $\Hom_A(T,F)=0$ for every $T \in \tc, F \in \fc$ and for every object $X \in \modA$ there is a short exact sequence 
        \[0 \to tX \to X \to fX \to 0\]
where $tX \in \tc$ and $fX \in \fc$.
If $(\tc, \fc)$ is a torsion pair we say that $\tc$ is a torsion class and that $\fc$ is a torsion-free class. 
It follows from the definition of torsion pairs that $\tc$ completely determines $\fc$:
  \[\fc = \{M\in \modA \mid \Hom_A(T,M) =0 \;\forall\, T\in \tc\}.\]
Similarly $\fc$ determines $\tc$ as \(\tc = \{M\in \modA \mid \Hom_A(M,F) =0 \;\forall\, F\in \fc\}\).
In particular, if $(\tc, \fc)$ and $(\tc', \fc')$ are two torsion pairs we have that $\tc \subset \tc'$ if and only if $\fc' \subset \fc$.

Unless stated otherwise, we denote by $\tors A$ and $\ftors A$ the set of all torsion classes and all torsion-free classes in $\modA$, respectively.
However, in some places we denote by $\tors A$ the set of all torsion pairs $(\tc, \fc)$ in $\modA$. 
When we do the latter, we will explicitly state it.

A well-known characterization of torsion classes in $\modA$ states that a subcategory $\tc$ of $\modA$ is a torsion class if and only if $\tc$ is closed under quotients and extensions. Dually, a subcategory $\fc$ of $\modA$ is a torsion-free class if and only if $\fc$ is closed under subobjects and extensions. 
An important consequence of this characterization is that the poset $(\tors A, \subseteq)$ is a complete lattice with the intersection as its meet \cite{IRTT}.

While the meet of two torsion classes $\tc, \tc' \in \tors A$ is just their intersection, their join is less easy to describe.
We now explain different ways to find it. 
First, one can consider the intersection 
\[
T(\tc \cup \tc') = \bigcap_{(\tc \cup \tc' )\subset \widetilde{\tc} \in \tors A} \widetilde{\tc}
\]
of all the torsion classes $\widetilde{\tc}$ containing $\tc \cup \tc'$.
This intersection always gives back a torsion class since $\tors A$ is a complete lattice. 
Also, note that this set is never empty since $\modA$ is a torsion class. 
In fact, the minimal torsion class $T(\xc)$ containing a subcategory $\xc \subset \modA$ is given by
\[
T(\xc) = \bigcap_{\xc \subset \tc \in \tors A} \tc.
\]

Another way to find the minimal torsion class containing a given subcategory $\xc$ is as follows. Given $\xc \subset \modA$, it follows from the definitions that the category 
\[
\xc^{\perp} = \{M \in \modA \mid \Hom_A(X, M) = 0 \text{ for every $X\in \xc$}\}
\]
is a torsion free class in $\modA$. Dually, 
\[
^\perp\xc= \{M \in \modA \mid \Hom_A(M, X) = 0 \text{ for every $X\in \xc$}\}
\]
is a torsion class in $\modA$. 
Moreover, for every $\xc \subset \modA$ the pair $({}^\perp(\xc^\perp), \xc^\perp)$ is a torsion pair in $\modA$ and ${}^\perp(\xc^\perp)$ is the minimal torsion class in $\modA$ containing $\xc$.

Alternatively, one can use that $\ftors A$ is also a complete lattice with intersection as its meet.
Moreover, $\ftors A$ is the dual lattice of $\tors A$. 
Since there is a bijection between torsion classes and torsion-free classes, it is frequently convenient to consider the join of two given torsion classes $\tc, \tc' \in \tors A$ as the torsion class corresponding to $\fc\cap\fc' \in \ftors(A)$ where $(\tc, \fc)$ and $(\tc',\fc')$ are torsion pairs in $\modA$.
Finally, one can show that $T(\xc)=\Filt(\Fac \xc)$.
Note that, in general, $\Fac(\Filt \xc)$ is not closed under extensions and therefore is not a torsion class.

Given two distinct torsion classes $\tc$ and $\tc'$ in $\tors A$ such that $\tc\subset\tc'$, we say that $\tc$ is maximally contained in $\tc'$ if a torsion class $\tc''$ satisfying $\tc \subset \tc'' \subset \tc'$ is either $\tc''=\tc$ or $\tc''=\tc'$. 
We denote by $\Hasse(A)$ the Hasse quiver of $\tors A$. 
That is, the quiver whose vertices are the elements of $\tors A$ and in which there is an arrow from $\tc \to \tc'$ if $\tc$ is maximally contained in $\tc'$.

It was shown in \cite{DIRRT} that $\tors A$ is a completely semidistributive lattice, generalizing the result of \cite{GM} where it is shown that $\tors A$ is a semidistributive lattice. 
An important result, first proved in \cite{BCZ}, states that there is a bijection between the (completely) join-irreducible elements of $\tors A$ and the bricks in $\modA$. 
This then induces a labeling of the edges of the Hasse diagram of $\tors A$, whose set of labels is contained in Br$(A)$.

%%%%%%%%%
\subsection{Grothendieck groups, cones and fans}\label{subsec: cones and fans}

The Grothendieck group $K_0(A)$ of $A$ is defined as the abelian group generated by the isomorphism classes $[M]$ of objects \mbox{$M\in\modA$}, modulo the relations generated by short exact sequences in $\modA$.
\[
K_0(A) := {\langle[M] \mid M\in \modA\rangle\over\langle [M]-[L]-[N]\mid 0\to L \to M\to N\to 0  \rangle}
\]
The Jordan--Hölder theorem induces a natural isomorphism 
$\Phi: K_0(A) \to \mathbb{Z}^n$ between $K_0(A)$ and $\mathbb{Z}^n$ which sends the class $[S(i)]$ of the simple module $S(i)$ to $\Phi([S(i)])=e_i$, the $i$-th element of the canonical basis of $\mathbb{Z}^n$. 
By abuse of notation, we identify the class $[M]\in K_0(A)$ with the vector $\Phi([M])$ and thus regard $[M]$ as an element of $\mathbb R^n$.
In this case we are thinking of $\mathbb{R}^n$ as $\mathbb{R}^n = \mathbb{Z}^n\otimes\mathbb{R}$.
Given two vectors $v, w \in \mathbb{R}^n$, we denote by $\langle v, w\rangle$ their canonical inner product in $\mathbb{R}^n$, that is $\langle v, w\rangle = \sum_{i=1}^n v_iw_i$.

A subset $\df$ of $\mathbb{R}^n$ is said to be an (open) convex cone if for every finite set $\{v_1, \dots, v_t\} \subset \df$ the set  
    $$\langle v_1, \dots, v_t \rangle_+:=\left\{v \in \mathbb{R}^n \mid v=\sum_{i=1}^t \lambda_i v_i \text{ with $\lambda_i > 0$ for every $1\leq i \leq t$}\right\}\subseteq \df.$$
    A cone $\df$ is polyhedral if there is a finite set $\{v_1, \dots, v_t\} \subset \df$ such that \mbox{$\df \subseteq \langle v_1, \dots, v_t \rangle_+$}.
  A polyhedral cone $\df$ is said to be simplicial if it can be generated by a linearly independent set.
The codimension $d$ of a cone $\df \subset \mathbb{R}^n$ is equal to $n$ minus the dimension $\dim_{\mathbb{R}} \langle \df\rangle$ of the subspace $\langle \df\rangle$ generated by $\df$.
We denote the codimension of $\df$ as $\codim \df$.

A set of open cones $\Df = \{\df \subset \mathbb{R}^n\}$ is said to be a cone complex if, for any two cones $\df_1, \df_2 \in \Df$ we have that $\df_1\cap\df_2 = \varnothing$. 
A set of closed cones $\Sigma= \{\sigma \subset \mathbb{R}^n\}$ is said to be a fan if, given $\sigma_1, \sigma_2 \in \Sigma$ we have that $\sigma_1\cap\sigma_2 = \sigma$ for some cone $\sigma \in \Sigma$ and every face $\sigma'$ of a cone $\sigma \in \Sigma$ satisfies $\sigma' \in \Sigma$.
In this paper, we will work with cone complexes
$\Df=\{\df\subset\mathbb R^n\}$ such that
\[
\Sigma_\Df=\{\overline{\df}\subset\mathbb R^n\mid \df\in\Df\}
\]
is a fan.

Given a cone complex $\Df$, we denote by $\Df^i$ the set 
$\Df^i := \{\df \in \Df \mid \codim \df = i\}$. 
Clearly \[ \Df = \bigcup_{i=0}^n \Df^i.\]
If $\df \in \Df$ is a cone of codimension $d$, we denote by $\pi_\df$ the canonical projection $\pi_\df : \mathbb{R}^n \to \df^\perp \cong \mathbb{R}^d$.
Following \cite{Amini, Adiprasito}, we denote by $\star(\df)$ the cone complex in $\mathbb R^d$ obtained by projecting via $\pi_\df$ all cones $\df'$ whose closures contain $\df$.
\[\star(\df) := \left\{\pi_{\df}(\df') \subset \mathbb{R}^d \mid \df'\in \Df \text{ and } \df \subset \overline{\df'}\right\}\]
Note that different cones $\df_1, \df_2 \in \Df$ might satisfy $\star(\df_1) = \star(\df_2)$.

%%%%%%%%%
\subsection{Filtrations induced by chains of torsion classes}\label{subsec: star product}

Given two subcategories $\xc$ and $\yc$ we denote by $\xc * \yc$ the full subcategory of $\modA$ whose objects are all $Z \in \modA$ such that there exists a short exact sequence $0\to X \to Z \to Y \to 0$ where $X \in \xc$ and $Y \in \yc$. 
More generally, if $I$ is a (possibly infinite) totally ordered set and 
$\{\xc_t \subseteq \modA \mid t \in I\}$
is a set of subcategories of $\modA$ indexed by $I$, we denote by $\bigast_{t \in I}^{\leftarrow} \xc_t$ the subcategory of objects $M \in \modA$ having a filtration 
$$0=M_0 \subset M_1 \subset \dots \subset M_{r-1} \subset M_r=M$$
where $M_i/M_{i-1} \in \xc_{s_i}$ for all $1\leq i \leq r$ and $s_1 > s_2 >\dots > s_r$.
Note that the category $\xc * \yc$ is not necessarily closed under extensions, even when this is the case for $\xc$ and $\yc$. 
The following observation is well-known.

\begin{proposition}\label{prop:extending torsion classes}
    Let $(\tc, \fc)$ and $(\tc', \fc')$ be two torsion pairs such that $\tc \subset \tc'$ or, equivalently, such that $\fc' \subset\fc$.
    Let $\xc_{[\tc,\tc']}$ be the subcategory $\tc' \cap \fc$. 
    Then $\tc' = \tc*\xc_{[\tc,\tc']}$ and $\fc = \xc_{[\tc,\tc']}* \fc'$. 
\end{proposition}

Fix a totally ordered set $I$ with a minimal element $0$ and a maximal element $1$.
Then a chain of torsion pairs $\eta$ is a set 
$$\eta:= \{ (\tc_\eta(i), \fc_\eta(i)) \in \tors A \mid 
\tc_\eta(i) \subset \tc_\eta(j) \text{ if $j \leq i$ in $I$}\}.$$
Given any chain of torsion pairs $\eta$ and any $t\in I$ one can show \cite[Proposition~2.10]{Treffinger-HN-filt} that 
$$
\left(\bigcup_{j>t} \tc_\eta(j),\bigcap_{j>t} \fc_\eta(j)\right) \text{ and }
\left(\bigcap_{i<t} \tc_\eta(i),\bigcup_{i<t} \fc_\eta(i)\right)
$$
are torsion pairs in $\modA$. 
Moreover, an easy verification shows that the torsion class $\bigcup_{j>t} \tc_\eta(j)$ is contained in $\bigcap_{i<t} \tc_\eta(i)$.
Then, associated to every chain of torsion pairs $\eta$ we define the set of subcategories 
\begin{multline*}
    \xc_\eta := \left\{\xc_\eta(i) \subset \modA \mid \xc_\eta(0) = \tc_\eta(0) \cap \bigcap_{i>0} \fc_\eta(i),\; \xc_\eta(1) = \fc_\eta(1)\cap \bigcap_{i<1} \tc_\eta(i) \;\text{ and }\right. \\
\left. \xc_\eta(k) = \bigcap_{i<k} \tc_\eta(i) \cap \bigcap_{j>k} \fc_\eta(j) \text{ for $0<k< 1$}\right\}.
\end{multline*}
It follows from the properties of torsion pairs that chains of torsion pairs always induce a distinguished filtration.

\begin{theorem}\cite[Theorem~1.2]{Treffinger-HN-filt}\label{thm:HarderNarasimhanfiltrations}
Let $\eta$ be a chain of torsion pairs as above. 
Then every module $M \in \modA$ admits a Harder-Narasimhan filtration, that is a filtration
$$0=M_0 \subset M_1 \subset \dots \subset M_{r-1} \subset M_r=M$$
where $M_i/M_{i-1} \in \xc_\eta (t_i)$ for all $1\leq i \leq r$ and $t_1 > t_2 >\dots > t_r$. 
Moreover, this filtration is unique up to isomorphism.
\end{theorem}

As a corollary of this result, one can extend \cref{prop:extending torsion classes} to relate the torsion classes $\tc_\eta(0)$ and $\tc_\eta(1)$ and the torsion free classes $\fc_\eta(0)$ and $\fc_\eta(1)$  for any chain of torsion pairs $\eta$ using the $*$-product defined above.

\begin{corollary}\label{cor:infinite products modA}
    For every chain of torsion classes $\eta$ the following holds. 
    $$\tc_\eta (0) = \tc_\eta(1) * \left(  \bigast^{\leftarrow}_{t\in I} \xc_\eta(t)\right)\quad
    \text{  } \quad\fc_\eta (1) = \left(  \bigast^{\leftarrow}_{t\in I} \xc_\eta(t)\right) *\fc_\eta(0).$$
\end{corollary}

\begin{proof}
    Let 
    $$\eta:= \{ (\tc_\eta(i), \fc_\eta(i)) \in \tors A \mid 
\tc_\eta(i) \subseteq \tc_\eta(j) \text{ if $j \leq i$ in $I$}\}$$
be a chain of torsion pairs. 
We only show that 
$$\tc_\eta (0) = \tc_\eta(1) * \left(  \bigast^{\leftarrow}_{t\in I} \xc_\eta(t)\right),$$
since the other case is similar.
This expression implies that an object $T_0\in \modA$ satisfies $T_0\in \tc_\eta(0)$ if and only if there is a short exact sequence $0 \to T_1 \to T_0 \to M \to 0$ where $T_1 \in \tc_\eta(1)$ and $M$ admits a filtration 
$$0=M_0 \subset M_1 \subset \dots \subset M_{r-1} \subset M_r=M$$
where $M_i/M_{i-1} \in \xc_\eta (t_i)$ for all $1\leq i \leq r$ and $t_1 > t_2 >\dots > t_r$. 
The claim then follows from \cref{thm:HarderNarasimhanfiltrations}. 
\end{proof}

%%%%%%%%%%%%%%%%%%
\section{Torsion classes and modules with bounded length}\label{sec:l-torsion classes}

For any positive $\ell \in \mathbb{N}$, we define a relation $\sim_\ell$ on $\tors A$ by setting $\tc \sim_\ell \tc'$ if and only if \mbox{$\tc\cap (\modA)_\ell = \tc'\cap (\modA)_\ell$}.
This is clearly an equivalence relation.
Let us denote by $\tors_\ell A$ the set of equivalence classes of torsion classes with respect to this equivalence relation. 
In this section we study the set $\tors_\ell A$.
In \cref{subsec:torsionclassesin modA_l} we study the basic properties of the elements of $\tors_\ell A$.
Later, in \cref{subsec:lattice tors_l A} we study its lattice-theoretic properties. 
Finally, \cref{subsec:bricklabelling-l} we show that $\tors_\ell A$ has a brick labeling analogous to that of $\tors A$.

%%%%%%%%%
\subsection{Torsion classes in $\modA_\ell$}\label{subsec:torsionclassesin modA_l}
We start by showing some properties of the equivalence classes in $\tors_\ell A$. 
Our first result concerning these equivalence classes is that they behave like torsion classes in $(\modA)_\ell$, despite $(\modA)_\ell$ not being an abelian category. 

\begin{theorem}\label{thm: torsion pairs of bounded length}
    Let $(\tc, \fc)$ be a torsion pair in $\modA$. Then the following holds. 
    \begin{enumerate}
        \item \(
        \fc\cap (\modA)_\ell= \left\{M \in (\modA)_\ell \mid \Hom_A(T,M)=0 \, \forall \, T\in \tc\cap(\modA)_\ell\right\}.
        \)
        \item \(
        \tc\cap (\modA)_\ell= \left\{M \in (\modA)_\ell \mid \Hom_A(M,F)=0 \, \forall \, F\in \fc\cap(\modA)_\ell\right\}.
        \)
        \item For every $M \in (\modA)_\ell$ there is a short exact sequence 
        \[
        0 \to tM \to M \to fM \to 0
        \]
        where $tM \in \tc\cap(\modA)_\ell$ and $fM\in\fc\cap(\modA)_\ell$.
    \end{enumerate}
\end{theorem}

\begin{proof}
    \textit{(1)}  We have that \(\Hom_A(T,M)=0 \) for every \( T\in \tc\) and for every $M\in \fc\cap(\modA)_\ell$. 
    In particular, we have that \(\Hom_A(T,M)=0 \) for every \( T\in \tc\cap(\modA)_\ell\) and for every $M\in \fc\cap(\modA)_\ell$. 
    This shows that 
    \[
    \fc\cap (\modA)_\ell \subseteq \left\{M \in (\modA)_\ell \mid \Hom_A(T,M)=0 \, \forall \, T\in \tc\cap(\modA)_\ell\right\}.
    \]

    We now show the reverse inclusion. 
    Let $M \in (\modA)_\ell$ such that \(\Hom_A(T,M)=0 \) for every \( T\in \tc\cap(\modA)_\ell\) and let $T'\in \tc$. 
    We need to show that $\Hom_A(T', M)=0$. 
    Let $f\in \Hom_A(T', M)$. 
    Then the morphism $f$ factors through its image $\Image f$ because $\modA$ is abelian. 
    Since $\Image f$ is a quotient of $T'$ we have that $\Image f$ belongs to $\tc$. 
    Also, since $\Image f$ is a subobject of $M$ we have that $\lg(\Image f)\leq \lg(M)\leq \ell$. 
    Then $\Image f \in \tc\cap (\mod A)_\ell$; hence $\Image f =0$ since $\Hom_A(\Image f, M)=0$ by hypothesis. 
    Therefore
    \(M \in \tc^\perp=\fc\) and \(M \in \fc\cap (\mod A)_\ell\).\\
    \textit{(2)} It is dual to (1).\\
    \noindent
    \textit{(3)} Since $(\tc, \fc)$ is a torsion pair in $\modA$ for every $M \in \modA$ we have a short exact sequence
    \[
    0 \to tM \to M \to fM \to 0
    \]
    where $tM \in \tc$ and $fM\in\fc$.
    In particular, we have such a sequence for every \mbox{$M \in (\modA)_\ell$}. 
    Since \mbox{$\lg(tM)\leq \lg(M)$} and $\lg(fM) \leq \lg(M)$, the result follows.
\end{proof}

From now on, given $\tc\in\tors A$, let $\tc_\ell\in\tors_\ell A$ denote its equivalence class.
By abuse of notation, we also use $\tc_\ell$ to denote the subcategory $\tc \cap (\modA)_\ell\subset (\modA)_\ell$.
Similarly for $\fc_\ell$.
We now give a couple of corollaries to the previous theorem.  

\begin{corollary}
    Let $(\tc, \fc)$ be a torsion pair in $\modA$ and let $\ell \geq 0$. 
    Then 
    \[
    \tc_\ell =\{tM \mid M \in (\modA)_\ell \} \quad \text{ and }\quad
    \fc_\ell =\{fM \mid M \in (\modA)_\ell \}.
    \]
\end{corollary}

% \begin{proof}
%     Let $tM$ be the torsion subobject of $M \in(\modA)_\ell$.
%     Then $tM$ is in $\tc$ by definition and $tM\in (\modA)_\ell$ because the length of $tM$ cannot be greater than $\ell$, the length of $M$. 
%     So $\{tM \mid M \in (\modA)_\ell \} \subset \tc\cap(\modA)_\ell$. 
%     The converse inclusion follows from the fact that $tM =M$ for every object $M\in \tc$.

%     The other equality is shown in a similar fashion.
% \end{proof}

%We define the equivalence relation $\sim_\ell$ using torsion classes. A similar construction can be done using torsion-free classes. 

\begin{corollary}\label{cor:sametorsion-free}
    Let $(\tc,\fc)$ and $(\tc', \fc')$ be two torsion pairs in $\modA$. Then $\tc_\ell = \tc'_\ell$ if and only if $\fc_\ell = \fc'_\ell$.
\end{corollary}
%\begin{proof}
    % We only show that $\tc\cap (\modA)_\ell = \tc'\cap (\modA)_\ell$ implies $\fc\cap (\modA)_\ell = \fc'\cap (\modA)_\ell$ since the other implication is shown by dual arguments.
    % Let $M \in \fc\cap (\modA)_\ell$ and $T'\in \tc'$. 
    % Then any morphism $f\in \Hom_A(T', M)$ factors through its image $\Image f$ because $\modA$ is an abelian category. 
    % Since $\Image f$ is a quotient of $T'$ we have that $\Image f$ belongs to $\tc'$. 
    % Also, since $\Image f$ is a subobject of $M$ we find that the length of $\Image f$ is at most $\ell$. Then we can conclude that $\Image f \in \tc'\cap (\mod A)_\ell$. 
    % By hypothesis $\tc'\cap (\mod A)_\ell = \tc\cap (\mod A)_\ell$ so we get that $\Image f \in \tc$ and $\Hom_A(\Image f, M)=0$. 
    % Then $\Image f =0$ and we can conclude that $M \in \fc'\cap (\modA)_\ell$.
    % Since the reverse inclusion can be shown using the same argument, the result follows. 
%\end{proof}

Suppose that $(\tc, \fc)$ and $(\tc', \fc')$ are two torsion pairs such that  $\tc_\ell \neq \tc'_\ell$. 
Then, without loss of generality, we may assume the existence of an object $M \in(\modA)_\ell$ such that $M \in \tc \setminus\tc'$. 
We now show that such an object can be chosen to be a brick in $\fc'_\ell$. 
Compare with \cite[Lemma~2.18]{Asai}.

\begin{proposition}\label{prop:objectintheintersection}
    Let $(\tc,\fc)$ and $(\tc', \fc')$ be two torsion pairs in $\modA$ and suppose that $\tc_\ell \neq \tc'_\ell$.
    Then there is a $M\in (\modA)_\ell$ such that  either $ M \in \tc'_\ell\cap \fc_\ell$ or $M \in \tc_\ell\cap \fc'_\ell$.
    Moreover, such an object $M$ can be chosen to be a brick. 
\end{proposition}

\begin{proof}
   By hypothesis, either $(\tc\setminus\tc')\cap (\modA)_\ell$ or $(\tc'\setminus\tc)\cap (\modA)_\ell$ contains a nonzero object $M$.
   Without loss of generality, suppose that there is an object \mbox{$M \in (\tc'\setminus\tc)\cap (\modA)_\ell$} that is nonzero.
    Take the short exact sequence 
    \[
    0\to tM \to M \to fM \to 0
    \]
    with respect to the torsion pair $(\tc, \fc)$.
    By hypothesis, $M$ is not an object of $\tc$ and it is nonzero. This implies that $fM$ is nonzero and belongs to $\fc$. 
    Also, $fM\in \tc'$ since $\tc'$ is closed under quotients. 
    Finally, $fM \in (\modA)_\ell$ since $\lg(fM) \leq \lg(M) \leq \ell$. 
    Therefore, $fM$ is non-zero and belongs to $\tc'_\ell\cap \fc_\ell$.
    %We can show with similar arguments that the existence of a nonzero object in $(\tc\setminus\tc')\cap (\modA)_\ell$ implies the existence of a nonzero object in $\tc_\ell\cap \fc'_\ell$. 

To prove the final statement, let $B$ be an object of minimal length in $\tc_\ell\cap \fc'_\ell$.
Then $B$ is an object of minimal length in $\tc\cap \fc'$.
Our claim then follows from \cite[Lemma~3.8]{DIRRT}.
\end{proof}

%%%%%%%%%
\subsection{The lattice $\tors_\ell A$}\label{subsec:lattice tors_l A}
In this subsection we consider the lattice-theoretic properties of $\tors_\ell A$. 

\begin{proposition}\label{prop:completelattice}
    The poset $(\tors_\ell A, \subseteq)$ is a complete lattice with the intersection as its meet. 
    Dually, the poset $(\ftors_\ell A, \subseteq)$ is a complete lattice with the intersection as its meet.
    Moreover, the pairing $(\tc_\ell, \fc_\ell)$ induces an order-reversing lattice isomorphism. 
\end{proposition}

\begin{proof}
    Let $\{\tc_i \mid i\in I\}$ be a family of torsion classes indexed by a set $I$. 
    Then
    \[
    \bigcap_{i\in I} (\tc_i)_\ell= \left(\bigcap_{i\in I} \tc_i\right)\cap (\modA)_\ell.
    \]
    Since $\bigcap_{i\in I} \tc_i$ is a torsion class, we have that  $\bigcap_{i\in I} (\tc_i)_\ell \in \tors_\ell A$. 
    Then $\tors_\ell A$ is a complete lattice by \cite[Chapter I, Lemma 34]{Gr}. 
    The proof for $(\ftors_\ell A, \subseteq)$ is analogous.
    The moreover part of the statement follows directly from \cref{thm: torsion pairs of bounded length}.
\end{proof}

\begin{proposition}\label{thm:epi of lattices}
    The map $\varphi_{\infty,\ell}: \tors A \to \tors_\ell A$ defined by $\varphi_{\infty, \ell}(\tc)=\tc_\ell$ is an epimorphism of lattices. 
    Moreover, the map $\varphi^{-1}_{\infty, \ell}(\tc_\ell)$ is an interval in $\tors A$.
\end{proposition}

\begin{proof}
    Clearly $\varphi_{\infty,\ell}$ is a surjective map. 
    Moreover, the map $\varphi_{\infty,\ell}$ preserves intersections and hence it is a morphism of lattices since both the meet and join in $\tors A$ and $\tors_\ell A$ can be defined via the intersection.

    Consider the set 
    \[   \varphi^{-1}_{\infty, \ell}(\tc_\ell) = \{\tc'\in\tors A \mid \tc'_\ell = \tc_\ell\}.\]
    In order to prove that $\varphi^{-1}_{\infty, l}(\tc_\ell)$ is an interval in $\tors A$ we need to show that it is convex and has a unique minimal and a unique maximal element. 
    
    To show that $\varphi^{-1}_{\infty, \ell}(\tc_\ell)$ is convex, let $\tc_1, \tc_2, \tc_3$ be three torsion classes verifying $\tc_1\subset\tc_2\subset\tc_3$ and $\tc_1, \tc_3 \in  \varphi^{-1}_{\infty, \ell}(\tc_\ell)$.
    Since \(
    (\tc_1)_\ell\subset (\tc_2)_\ell \subset (\tc_3)_\ell\) and \((\tc_1)_\ell=(\tc_3)_\ell=\tc_\ell\),  we have that $\tc_2 \in \varphi^{-1}_{\infty, \ell}(\tc_\ell)$ and $\varphi^{-1}_{\infty, \ell}(\tc_\ell)$ is convex.
   
   Let $T(\tc_\ell) \in \tors A$ be the unique minimal torsion class containing $\tc_\ell$ and let $\tc \in \varphi^{-1}_{\infty, l}(\tc_\ell)$.
    Then $\tc_\ell \subset T(\tc_\ell)_\ell = T(\tc_\ell)\cap(\modA)_\ell$. 
    Also, $T(\tc_\ell) \subset \tc$ because $\tc_\ell \subset \tc$. 
    This implies $T(\tc_\ell)_\ell \subset \tc_\ell$.
    Therefore, $T(\tc_\ell)\in \varphi^{-1}_{\infty, \ell}(\tc_\ell)$ and it is the unique minimal element in $\varphi^{-1}_{\infty, l}(\tc_\ell)$.

    Using similar arguments, one can show the existence of a unique minimal torsion-free class $F(\fc_\ell)$ satisfying $F(\fc_\ell)_\ell= \fc_\ell$.
    If $\widetilde\tc$ is the torsion class such that $(\widetilde\tc, F(\fc_\ell))$ is a torsion pair in $\modA$ then it follows from \cref{cor:sametorsion-free} that  $\widetilde\tc \in \varphi^{-1}_{\infty, \ell}(\tc_\ell)$ and that is the unique maximal element in $\varphi^{-1}_{\infty, l}(\tc_\ell)$.
\end{proof}

The following result can be proven using arguments similar to those in the proof of \cref{thm:epi of lattices}. 
We leave the details to the reader. 

\begin{proposition}\label{prop:morphismstodiagram}
    Let $\ell<m$ be two nonnegative integers. 
    The map $$\varphi_{m,\ell}: \tors_m A \to \tors_\ell A$$ defined by $\varphi_{m, \ell}(\tc_m)=\tc_\ell$ is then an epimorphism of lattices. 
    Moreover, $\varphi^{-1}_{m, \ell}(\tc_\ell)$ is an interval in $\tors_m A$ with a unique minimal and a unique maximal element.
\end{proposition}

The epimorphisms of lattices defined in the previous proposition behave well under composition, as we show in the next result.

\begin{proposition}\label{prop:composition}
    Let $\ell_1, \ell_2, \ell_3$ be three nonnegative integers such that $\ell_1\leq \ell_2 \leq \ell_3$. 
    Then $\varphi_{\ell_3,\ell_2}\circ\varphi_{\ell_2,\ell_1}=\varphi_{\ell_3,\ell_1}$.
\end{proposition}

\begin{proof}
    \begin{align*}
        \left(\varphi_{\ell_3,\ell_2}\circ\varphi_{\ell_2,\ell_1}\right)(\tc_{\ell_3})&= ((\tc_{\ell_3})\cap (\modA)_{\ell_2})\cap (\modA)_{\ell_1}\\
        & = (\tc_{\ell_3})\cap( (\modA)_{\ell_2}\cap (\modA)_{\ell_1})\\
        & = (\tc_{\ell_3})\cap(\modA)_{\ell_1}\\
        & = \varphi_{\ell_3,\ell_1}(\tc_{\ell_3})
    \end{align*}
\end{proof}
 
As a consequence of \cref{prop:composition} we can build the diagram $F : \mathbb{N} \to Sets$ such that $F(n) = \tors_n A$ and $F(n\leq m) = \varphi_{m,n}$. 
Furthermore, \cref{prop:morphismstodiagram} implies that $\tors A$ is a cone for this diagram. 

\begin{corollary}\label{cor:limit1}
    \[
     \underleftarrow\lim \tors_\ell A =\tors A.
    \]
\end{corollary}

\begin{figure}
    \centering
% https://q.uiver.app/#q=WzAsNixbOCwwLCIoXFx0b3JzIEEpXzAiXSxbNiwwLCIoXFx0b3JzIEEpXzEiXSxbNCwwLCJcXGNkb3RzIl0sWzIsMCwiIChcXHRvcnMgQSlfbCJdLFsyLDIsIlxcdG9ycyBBIl0sWzAsMCwiXFxjZG90cyJdLFszLDIsIlxcdmFycGhpX3tsLGwtMX0iLDAseyJzdHlsZSI6eyJoZWFkIjp7Im5hbWUiOiJlcGkifX19XSxbMiwxLCJcXHZhcnBoaV97MiwxfSIsMCx7InN0eWxlIjp7ImhlYWQiOnsibmFtZSI6ImVwaSJ9fX1dLFsxLDAsIlxcdmFycGhpX3sxLDB9IiwwLHsic3R5bGUiOnsiaGVhZCI6eyJuYW1lIjoiZXBpIn19fV0sWzQsMywiXFx2YXJwaGlfe1xcaW5mdHksbH0iLDAseyJzdHlsZSI6eyJoZWFkIjp7Im5hbWUiOiJlcGkifX19XSxbNCwxLCJcXHZhcnBoaV97XFxpbmZ0eSwxfSIsMCx7InN0eWxlIjp7ImhlYWQiOnsibmFtZSI6ImVwaSJ9fX1dLFs0LDAsIlxcdmFycGhpX3tcXGluZnR5LDB9IiwyLHsic3R5bGUiOnsiaGVhZCI6eyJuYW1lIjoiZXBpIn19fV0sWzUsMywiXFx2YXJwaGlfe2wrMSxsfSJdXQ==
\[\begin{tikzcd}
	\cdots && { \tors_\ell A} && \cdots && {\tors_2 A} && {\tors_1 A} \\
	\\
	&& {\tors A}
	\arrow["{\varphi_{\ell+1,\ell}}", from=1-1, to=1-3]
	\arrow["{\varphi_{\ell,\ell-1}}", two heads, from=1-3, to=1-5]
	\arrow["{\varphi_{3,2}}", two heads, from=1-5, to=1-7]
	\arrow["{\varphi_{2,1}}", two heads, from=1-7, to=1-9]
	\arrow["{\varphi_{\infty,\ell}}", two heads, from=3-3, to=1-3]
	\arrow["{\varphi_{\infty,2}}", two heads, from=3-3, to=1-7]
	\arrow["{\varphi_{\infty,1}}"', two heads, from=3-3, to=1-9]
\end{tikzcd}\]
    \caption{$\tors A$ as the inverse limit of the lattices $\tors_\ell A$ for $\ell\in\mathbb{N}$.}
        \label{fig:inverse limits}
\end{figure}
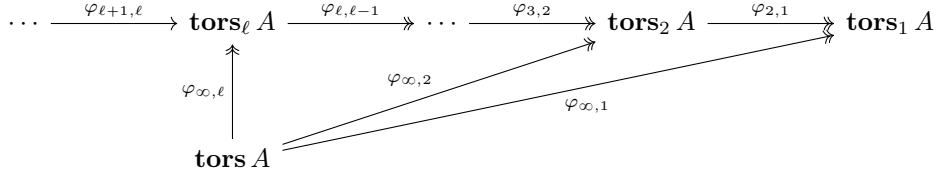

\begin{proposition}\label{prop: semidistributive ell}
    For every algebra $A$ and every $\ell \in \mathbb{N}$, the lattice $\tors_\ell A$ is completely semidistributive.
\end{proposition}

\begin{proof}
    We claim that $\tors_\ell A$ is a completely join semidistributive lattice.
    Let $\tc_\ell \in\tors_\ell A$ and $S=\{(\tc_s)_\ell\}\subset \tors_\ell A$ be such that for every $(\tc_s)_\ell \in S$ we have 
    \[(\tc_\ell)\cap(\tc_s)_\ell = \tc'_\ell\] for some fixed $\tc'_\ell \in \tors_\ell A$.
    Let us denote $\overline{\tc}_\ell = \bigvee\limits_{S} (\tc_s)_\ell$. 
    We need to show that \( \tc_\ell \cap \overline{\tc}_\ell=\tc'_\ell.\)

    Clearly, $(\tc'_\ell)\subset \tc_\ell \cap \overline{\tc}_\ell$. 
    Suppose there is an object $M$ in $\tc_\ell \cap \overline{\tc}_\ell$ such that $M \not \in \tc'_\ell$.
    Without loss of generality, we can assume that such $M$ is of minimal length.
    Note that our assumption implies that there is no $(\tc_s )_\ell \in S$ such that $M\in (\tc_s)_\ell$.
    Then $M$ is filtered by objects in $\bigcup\limits_S (\tc_s)_\ell$. 
    In particular, there is a proper nonzero submodule $M_1$ of $M$ such that $M_1\in (\tc_s)_\ell\in S$. 
    Moreover, $M_1 \in \tc_\ell$.
    Then $M/M_1$ is also an object of $\tc_\ell\cap\overline{\tc}_\ell$ such that $M/M_1 \not \in \tc'_\ell$ and $0 \neq \lg(M/M_1) < \lg(M)$, contradicting the minimal length of $M$. 
    This leads to a contradiction and, therefore $\tors_\ell A$ is a completely join semidistributive lattice.
    
    Similarly, we can prove that $\ftors_\ell A$ is completely join semidistributive. 
Then the complete meet semidistributivity of $\tors_\ell A$ follows from the complete join semidistributivity of $\ftors_\ell A$ together with \cref{prop:completelattice}.
\end{proof}

%%%%%%%%%%%%
\subsection{Bricks and labelling of $\tors_\ell A$}\label{subsec:bricklabelling-l}
The lattice $\tors_\ell A$ inherits many important properties from $\tors A$ as we now show. 
In the remainder of the section we follow closely the exposition given in \cite{Thomas}.

Observe that every object $M \in \modA$ can only be filtered by objects $L$ such that $\lg(L)\leq \lg(M)$.
In the following proposition, we show that a torsion class $\tc_\ell \in \tors_\ell A$ is the minimal torsion class containing the set $Br_\ell(\tc)$ of all its bricks in $\tc_\ell$.
In lattice-theoretic terms, we show that $\tc_\ell$ is the join of the minimal torsion class $T(B)_\ell$ containing each brick $B$ in $Br_\ell(\tc)$.

\begin{proposition}\label{prop:allbricks}
Let $\tc_\ell\in \tors_lA$ and let $Br_\ell(\tc)$ be the set of all the bricks $B\in \tc_\ell$.
Then 
\[
\tc_\ell = \bigvee_{B\in Br_\ell(\tc)} T(B)_\ell.
\]
\end{proposition}

\begin{proof}
Let \( \uc_\ell = \bigvee_{B\in Br_\ell(\tc)} T(B)_\ell \). 
Clearly $\uc_\ell \subset \tc_\ell$. 
Let $M \in \tc_\ell$.
We want to show that $M \in \uc_\ell$.
We show this by induction on $\lg(M)$.
If $\lg(M)=1$ then $M$ is a brick because it is simple and there is nothing to prove. 

Now suppose that $\lg(M) >1$.
If $M$ is a brick, there is nothing to prove.
Suppose that $M$ is not a brick.
Then there is a map $f \in \End_A(M)$ which is neither zero nor an isomorphism inducing the following short exact sequence. 
\[
0 \to f(M) \to M \to M/f(M) \to 0
\]
By the choice of $f$ we get that $\lg (f(M)) < \lg (M)$ and $\lg (M/f(M)) < \lg (M)$ and then, by inductive hypothesis, both $f(M), M/f(M) \in \uc_\ell$.
The result then follows because $\uc_\ell$ is closed under extensions in $(\modA)_\ell$.
\end{proof}

Recall that an element $x$ in a lattice $\pc$ is said to be completely join-irreducible if there is no subset $Y=\{y \in \pc \mid y < x\} \subset \pc$ such that $x= \bigvee_{y\in Y} y$.
We now give a characterisation of the join-irreducible elements of $\tors_\ell A$. 
This is a generalisation of \cite[Theorem 1.5]{BCZ}, whose proof we simply adapt to our context.

\begin{theorem}
The map $T(-)_\ell : (\modA)_\ell \to \tors_\ell A$ induces a bijection between the bricks in $(\modA)_\ell$ and the complete join irreducible in $\tors_\ell A$.
\end{theorem}

\begin{proof}
Let $B$ be a brick in $(\modA)_\ell$ and take $T(B)_\ell \in \tors_\ell A$. 
We need to show that there is a unique maximal element of $\tors_\ell A$, denoted by $\tc_\ell$, such that such that $\tc_\ell \subsetneq T(B)_\ell$.
We claim that this element is of the form $(T(B)\cap {}^\perp F(B))_\ell$, where $F(B)$ is the minimal torsion-free class containing $B$.
% and ${}^\perp F(B)$ is the torsion class consisting of all the objects in $\modA$ receiving no nonzero maps for objects in $F(B)$.

Since $B\not\in {}^\perp F(B)$,  $(T(B)\cap {}^\perp F(B))_\ell$ is strictly contained in $T(B)_\ell$.
Let $\tc_\ell $ be an element of $\tors_\ell A$ strictly contained in $T(B)_\ell$.
Since $\tc_\ell$ is closed under quotients there is no object $X$ in $\tc_\ell$ such that there is an epimorphism $p: X \to B$.
%Then \cite[Lemma 1.7.(1)]{AsaiSemibricks} implies that $\tc\cap\tors_\ell A \subset T(B)_\ell$.
Then \cite[Lemma 1.7.(1)]{AsaiSemibricks} implies that $\tc_\ell  \subset {}^\perp B = {}^\perp F(B)$. 
Hence we can conclude that $\tc_\ell \subset (T(B)\cap {}^\perp F(B))_\ell$. 
This shows that $T(B)_\ell$ is a join-irreducible element in $\tors_\ell A$ for every brick $B\in (\modA)_\ell$. 
The fact that every join-irreducible is of the form $T(B)_\ell$ for a brick $B \in Br_\ell(A)$ follows directly from \cref{prop:allbricks}.

Now, suppose that $\tc_\ell$ is a join irreducible element in $\tors_\ell A$ such that $\tc_\ell =T(B)_\ell=T(B')_\ell$ for a pair of bricks $B, B'$ in $(\modA)_\ell$.
Then \cite[Lemma 1.7.(1)]{AsaiSemibricks} implies the existence of epimorphisms $p: B\to B'$ and $p': B' \to B$. 
Therefore $B$ is isomorphic to $B'$ and we can conclude that the map $T(-)_\ell : (\modA)_\ell \to \tors_\ell A$ is injective on bricks, finishing the proof.
\end{proof}

%We denote by $\tc \strsubset \tc'$ whenever $\tc$ is maximally contained in $\tc'$. 
Given two elements $\tc_\ell$ and $\tc'_\ell$ in $\tors_\ell A$ we say that $\tc_\ell$ is maximally contained in $\tc'_\ell$ if $\tc_\ell \subset \tc'_\ell$ and for every torsion class $\tc''$ verifying $\tc_\ell \subset \tc''_\ell\subset\tc'_\ell$ 
either $\tc_\ell = \tc''_\ell$ or $\tc''_\ell =\tc'_\ell$. 
As a corollary of the previous result and \cite[Proposition 9.1]{Thomas} we obtain the following.

\begin{corollary}\label{cor:bricklabelling_l}
    Suppose that  $\tc_\ell, \tc'_\ell \in \tors_\ell A$ are such that $\tc_\ell$ is maximally contained in $\tc'_\ell$. 
    Then, up to isomorphism, there exists a unique brick $B\in (\modA)_\ell$ such that 
    \[
    \tc'_\ell = T(B)_\ell \vee \tc_\ell.
     \]
\end{corollary}

This result implies, in particular that we can assign a brick to every arrow in the Hasse quiver $\hc_\ell(A)$ of $\tors_\ell A$, which is usually known as the brick labelling of $\hc_\ell(A)$. 

Although the previous result determines completely the brick labelling of $\tors_\ell A$, it may be useful to determine the brick $B$ labelling a given arrow $\tc_\ell \xrightarrow{B} \tc'_\ell$ in $\hc_\ell(A)$ directly from $\tc_\ell$ and $\tc'_\ell$. 
This problem was solved in \cite{BCZ} for $\tors A$ using the notion of minimal extending modules that we adapt for $\tors_\ell A$.

\begin{definition}\label{def:minext}
Let $\tc_\ell\in \tors_\ell A$ and $M\in (\modA)_\ell \setminus \tc_\ell$. 
We say that $M$ is a minimal extending module for $\tc_\ell$ if the following conditions are satisfied:
\begin{enumerate}
\item Every proper quotient $N$ of $M$ belongs to $\tc_\ell$.
\item If $0 \to M \to X \to T \to 0$ is a non-split exact sequence with $T\in \tc_\ell$ then $X\in \tc_\ell$. 
\item $\Hom_A (T, M)=0$ for every $T \in \tc_\ell$.
\end{enumerate}
\end{definition}

The first thing to notice is that the third condition in \cref{def:minext} together with \cref{thm: torsion pairs of bounded length} implies that a minimal extending module $M$ for $\tc_\ell$ satisfies $M \in \fc_\ell$.
The following is the corresponding extension of \cite[Theorem~2.8]{BCZ}.% and their proof works for $\tors_\ell A$, which we include for the sake of completeness. 

\begin{proposition}\label{prop:minimal extending}
Let $\tc_\ell, \tc_\ell' \in \tors_\ell A$ be such that $\tc_\ell$ is maximally contained in $\tc'_\ell$. 
Then there is a unique minimal extending module $B\in (\modA)_\ell \setminus \tc_\ell$ such that $\tc'_\ell = T(\{B\} \cup \tc_\ell)$.
Moreover, $B$ is a brick. 
\end{proposition}

\begin{proof}
Let $M\in \tc'_\ell \setminus \tc_\ell$.
If every proper quotient of $M$ is in $\tc_\ell$, we set $M=B$. 
Otherwise, let $L$ be a maximal subobject of $M$ such that $B:=M/L$ verifies that every proper quotient of $B$ is in $\tc_\ell$.
First, note that $B$ is nonzero since $\tc_\ell$ is properly contained in $\tc'_\ell$. 
By construction $B$ verifies \cref{def:minext}.(1).

Secondly, note that $B$ is a brick. 
Indeed, if $f\in \End_A(B)$ then both $\Image(f)$ and $B/\Image(f)$ are quotients of $B$. 
Hence, if $f$ is not zero nor an isomorphism,  $\Image(f)$ and $B/\Image(f)$ both belong to $\tc_\ell$ by construction of $B$. 
This implies that $B\in \tc_\ell$ since it is closed under extensions in $(\modA)_\ell$. 
Therefore $f$ is either zero or an isomorphism and $B$ is a brick.
Moreover, $B$ is unique by \cite[Lemma~2.5]{BCZ}.

%Using the same argument we conclude that $\Hom_A(T, B)=0$ for every $T \in \tc_\ell$.
Let $T\in \tc_\ell$ and let $f\in \Hom_A(T, B)$.
Then, $\Image(f) \in \tc_\ell$ since $\Image(f)$ is a quotient of $T$ and $\tc_\ell$ is closed under quotients. 
If $f$ is nonzero, $B / Im f \in \tc_\ell$ by hypothesis. 
Then $B\in \tc_\ell$ because $\tc_\ell$ is closed under extensions, a contradiction with our assumption.
Then $\Hom_A(T, B)=0$ for every $T \in \tc_\ell$.
This shows that $B$ satifies \cref{def:minext}.(3).

To prove \cref{def:minext}.(2) consider a non-split exact sequence 
\begin{equation}\label{eq:sec min}
0 \to B \xrightarrow{f} X \xrightarrow{g} T \to 0
\end{equation}
with $T\in \tc_\ell$. 
Suppose that $X\in \tc'_\ell \setminus \tc_\ell$. 
Then \cite[Proposition~2.6]{BCZ} implies the existence of an epimorphism $p: X \to B$.
If $p f =0$ then $p$ factors through $g$, which would imply that $B \in \tc_\ell$ since it is a quotient of $T$.
Then $pf$ is a non-zero endomorphism, which is necessarily an isomorphism because $B$ is a brick.  
Therefore the sequence \cref{eq:sec min} splits. 
This contradicts our assumption that $X\in \tc'_\ell \setminus \tc_\ell$. 

Finally, note that $\tc_\ell \subsetneq T(\{B\} \cup \tc_\ell) \subset \tc'_\ell$. 
This implies $T(\{B\} \cup \tc_\ell) = \tc'_\ell$ due to the fact that $\tc_\ell$ is maximally contained in $\tc'_\ell$.
\end{proof}

A semibrick \cite{AsaiSemibricks} is a set of bricks $\{B_i \mid i \in I\}$ such that $\Hom_A(B_i, B_j)=0$ if $i \neq j$.
We can assign to every torsion class $\tc_\ell$ in $\tors_\ell A$ two sets of bricks $\sc^-_{\tc_\ell}$, $\sc^+_{\tc_\ell}$ corresponding to all the bricks labelling the arrows ending and starting in $\tc_\ell$, respectively.

\begin{corollary}\label{cor:semibricks}
Let $\tc_\ell \in \tors_\ell A$ and let $\sc^-_{\tc_\ell}$, $\sc^+_{\tc_\ell}$ as above. 
Then $\sc^-_{\tc_\ell}$, $\sc^+_{\tc_\ell}$ are semibricks. 
\end{corollary}

\begin{proof}
We show this for $\sc^-_{\tc_\ell}$, since the proof for $\sc^+_{\tc_\ell}$ follows from the duality between $\tors_\ell A$ and $\ftors_\ell A$.
Let $B \in \sc^-_{\tc_\ell}$. 
Then there is $\tc_\ell^B \in \tors_\ell A$ maximally contained in $\tc_\ell$ such that  $B$ is a minimal extending module for $\tc^B_\ell$.

If $\sc^-_{\tc_\ell} = \{B\}$ we are done. 
Otherwise, let $B' \in \sc^-_{\tc_\ell}$ different from $B$. 
If $B' \in \tc_\ell \setminus \tc^B_\ell$ we have that $B \cong B'$ by \cite[Lemma~2.5]{BCZ}, which is a contradiction with our hypothesis. 
Then $B'\in \tc^B_\ell$ and $\Hom_A(B', B)=0$ by \cref{def:minext}.
So $\sc^-_{\tc_\ell}$ is a semibrick in $(\modA)_\ell$.
\end{proof}

An algebra $A$ is said to be $\tau$-tilting finite if there are finitely many torsion classes in $\modA$. 
It has been shown in \cite{ST} that an algebra $A$ is $\tau$-tilting finite if and only if there is a $t\in \mathbb{N}$ such that $\lg(B) \leq t$ for every brick $B\in \modA$. 
The following is then a direct consequence of  \cite[Theorem~1.3]{ST} and \cref{cor:bricklabelling_l}.

\begin{corollary}\label{cor:tors_l tau-tilting finite}
An Artin algebra $A$ is $\tau$-tilting finite if and only if there is a $t\in \mathbb{N}$ such that $\tors A = \tors_\ell A$. 
Moreover, in this case, $\tors_\ell A$ is finite for all $\ell \in \mathbb{N}$.
\end{corollary}

The lattice $\tors A$ for a $\tau$-tilting finite algebra $A$ enjoys interesting properties. 
Besides being finite, the Hasse quiver $\hc(A)$ is connected \cite{DIJ} and every torsion class $\tc$ in $\tors A$ is determined by the semibricks  $\sc^-_{\tc}$ and $\sc^+_{\tc}$ \cite{AsaiSemibricks}. 
These properties need not hold for algebras that are not $\tau$-tilting finite.

\begin{conjecture}
The lattice $\tors_\ell A$ is connected for every Artin algebra $A$ and every $\ell \in \mathbb{N}$. 
\end{conjecture}

\begin{conjecture}
Every $\tc_\ell \in \tors_\ell A$ is completely determined by the semibricks $\sc^-_{\tc_\ell}$ and $\sc^+_{\tc_\ell}$.
\end{conjecture}

The proofs of these two facts for $\tau$-tilting finite algebras rely heavily on the tools of homological algebra that are no longer available in $(\modA)_\ell$ because $(\modA)_\ell$ is not an additive category. 
We finish the section with an example.

\begin{example}\label{ex: Kronecker 1}
Let $A=\mathbb{C}Q$ be the path algebra of the Kronecker quiver \(Q=\begin{tikzcd}
	1 & 2
	\arrow[shift right, from=1-1, to=1-2]
	\arrow[shift left, from=1-1, to=1-2]
\end{tikzcd}\) 
over $\mathbb{C}$ and consider the category $(\modA)_2$. 
The objects of this category are the simple $A$-modules $\rep{1}$ and $\rep{2}$, together with all the quasi-simple modules $\rep{1\\2}_\lambda$, $\lambda \in \mathbb{P}^1(\mathbb{C})$. 
These are the modules at the mouth of the tubes of the Auslander-Reiten quiver of $A$.

Then the set $\tors_2 A$ consists of $(\modA)_2$, $\{\rep{2}\}$, $\{0\}$ and $\tc_2^P:= \left\{\rep{1}, \rep{1\\2}_\lambda \mid \lambda \in P\right\}$ for every $P$ in the power set of  $\mathbb{P}^1(\mathbb{C})$. 
The Hasse quiver of $\tors_2A$ can be seen in \cref{fig:Hasse 2-Kronecker-2}, where $\hc(\mathcal{P}(\mathbb{P}^1(\mathbb{C})))$ is the Hasse quiver of $\mathcal{P}(\mathbb{P}^1(\mathbb{C}))$ ordered by inclusion. 
Note that the torsion class $\tc_2^{\varnothing} = \{\rep{1}\}$.
One can easily verify that there is no element $\tc_2$ in $\tors_2 A $ such that $\{0\} \subsetneq \tc_2 \subsetneq \{\rep{1}\}$ nor $\tc_2^{\mathbb{P}^1(\mathbb{C})} \subsetneq \tc_2 \subsetneq (\modA)_2$. 
As a consequence, $\tors_2 A$ is connected, unlike the classical $\tors A$ which has two connected components. 

\begin{figure}
% https://q.uiver.app/#q=WzAsNCxbMSwwLCIoXFxtb2RBKV8yIl0sWzAsMSwiXFx7XFxyZXB7Mn1cXH0iXSxbMiwxLCJcXGhjKFxcbWF0aGNhbHtQKFxcbWF0aGJie1B9XjEoaykpfSkiXSxbMSwyLCJcXHswXFx9Il0sWzAsMSwiMSIsMl0sWzAsMiwiMiJdLFsxLDMsIjIiLDJdLFsyLDMsIjEiXV0=
\[\begin{tikzcd}
	& {(\modA)_2} \\
	{\{\rep{2}\}} && {\hc(\mathcal{P}(\mathbb{P}^1(k)))} \\
	& {\{0\}}
	\arrow["1"', from=1-2, to=2-1]
	\arrow["2", from=1-2, to=2-3]
	\arrow["2"', from=2-1, to=3-2]
	\arrow["1", from=2-3, to=3-2]
\end{tikzcd}\]
\caption{The Hasse quiver of $\tors_2A$}
\label{fig:Hasse 2-Kronecker-2}
\end{figure}
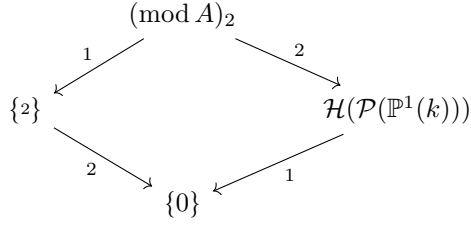

\end{example}

%%%%%%%%%%%
\section{Wall-and-chamber structures for algebras}\label{sec:conecomplex}

In this section, we begin by recalling in \cref{subsec: W&C modA} the basics of the wall-and-chamber structure $\Df(A)$ for $\modA$, introduced in \cite{BridgelandScat} and further developed in \cite{BST, Asai}.
For a more detailed survey of wall-and-chamber structures, we refer the reader to \cite{KT}. 
Later, in \cref{subsec: W&C modA_l}, we fix a positive integer $\ell$ and define the wall-and-chamber structure $\Df_\ell(A)$ by restricting stability conditions to $(\modA)_\ell$. 
We show that $\Df_\ell(A)$ has the structure of a cone complex whose cones are determined by the elements of $\tors_\ell A$.
Moreover, we show that, for every element  $\tc_\ell \in \tors_\ell A$ arising from a stability condition, there exists a chamber realizing it.

%%%%%%%%%
\subsection{The wall-and-chamber structure for $\modA$}\label{subsec: W&C modA}

Let us recall the definition of stability conditions introduced by King in \cite{King} using the notation introduced in \cref{subsec: cones and fans}.

\begin{definition}\cite{King}
Let $M\in \modA$ and $v\in \mathbb{R}^n$.
We say that $M$ is \textit{\mbox{$v$-semistable}} if $\langle v, [M]\rangle =0$ and $\langle v, [L]\rangle \leq 0$ for every submodule $L$ of $M$. 
We say that the module $M$ is \textit{$v$-stable} if the inequalities are strict for every proper submodule $L$ of $M$.      
\end{definition}

For a stability condition $v\in \mathbb{R}^n$ we denote by $\mod_v^{ss}A$ the category of all the $v$-semistable modules and the zero object.
It has been shown in \cite{Rudakov} that $\mod_v^{ss}A$ is an abelian subcategory of $\modA$ whose simple objects are the $v$-stable objects.

\begin{definition}
Given a nonzero module $M$ we define the stability locus $\dc(M)$ of $M$ to be the set 
\[
\dc(M) = \{v \in \mathbb{R}^n \mid M \text{ is $v$-semistable}\}.
\]
We say that $\dc(M)$ is a \textit{wall} if it is of codimension $1$ in $\mathbb{R}^n$.
A \textit{chamber} is an open and connected component of 
\[
\mathbb{R}^n \setminus \overline{\bigcup_{0 \neq M \in \modA} \dc(M)},
\]
the complement of the closure of the union of all walls. 
The union of all walls and chambers is known as the \textit{wall-and-chamber structure} of $A$ and we denote it by $\Df(A)$.
\end{definition}

%From the definitions it follows that $\dc(M\oplus N) = \dc(M) \cap \dc(N)$. 
%Therefore, the wall-and-chamber structure of an algebra $A$ is determined by the walls of the form $\dc(M)$ where $M$ is an indecomposable $A$-module. 
%Moreover, it can be shown that for every module $M$ there is a brick $B$ in $\modA$ such that $\dc(M)\subset\dc(B)$.

For any algebra $A$, there are at least $n$ walls $\{\dc(S(1)), \dots, \dc(S(n))\}$ in its wall-and-chamber structure which are determined by the simple modules $\{S(1), \dots, S(n)\}$.
Since we identify the class $[S(i)]\in K_0(A)$ with $e_i \in \mathbb{Z}^n$, it follows from the definitions that $\dc(S(i))$ corresponds to the whole coordinate hyperplane orthogonal to $e_i$. 
Also there are always at least two different chambers known as the \textit{positive} and the \textit{negative} chamber that we denote by $\cc^+$ and $\cc^-$, respectively. 
As indicated by their names, 
$$\cc^+ = \{v= (v_1, \dots, v_n)\in \mathbb{R}^n \mid v_i>0 \text{ for every $1\leq i \leq n$} \}$$ 
and 
$$\cc^- = \{v= (v_1, \dots, v_n)\in \mathbb{R}^n \mid v_i<0 \text{ for every $1\leq i \leq n$} \}.$$
The fact that $\cc^+$ and $\cc^-$ are chambers follows from the fact that if $M\neq 0$ then $\langle v, [M]\rangle >0$ whenever $v\in\cc^+$ and $\langle v, [M]\rangle <0$ whenever $v\in\cc^-$.
%In general, there are many more chambers in addition to $\cc^+$ and $\cc^-$.
%See \cite{KaipelTreffinger} for a detailed description of the wall-and-chamber structure of an algebra.

A key observation relating the notions of torsion classes and stability conditions first appeared in \cite{BKT}. 
In this paper the authors associate two torsion pairs to every stability condition $v\in\mathbb{R}^n$ as follows. 

\begin{proposition}\cite[Proposition 3.1]{BKT}\label{prop:TorsionBKT}
    Let $v \in \mathbb{R}^n$. 
    There are torsion pairs $(\tc_v, \overline{\fc_v})$ and $(\overline{\tc_v}, \fc_v)$ of $\modA$ where
    \begin{itemize}
        \item $\tc_v := \{ 0 \}\cup \{ M \in \modA : \langle v, [N]\rangle > 0 \text{ for all quotients $N$ of $M$}\}$.
        \item $\overline{\tc_v} := \{ 0 \}\cup \{ M \in \modA : \langle v, [N] \rangle\geq 0 \text{ for all quotients $N$ of $M$}\}$.
        \item $\fc_v := \{ 0 \}\cup \{ M \in \modA : \langle v, [L]\rangle < 0 \text{ for all submodules $L$ of $M$}\}$.
        \item $\overline{\fc_v} := \{ 0 \}\cup \{ M \in \modA : \langle v, [L]\rangle \leq 0 \text{ for all submodules $L$ of $M$}\}$.
    \end{itemize}
    Moreover, $\overline{\tc_v}\cap \overline{\fc_v}=\mod^{ss}_vA$.
\end{proposition}

An easy but important corollary of the previous result is that given a stability condition $v\in\mathbb{R}^n$, the torsion pairs $(\tc_v, \overline{\fc_v})$ and $(\overline{\tc_v}, \fc_v)$ coincide if and only if there are no nonzero \mbox{$v$-semistable} modules. 

\begin{proposition}
    Let $v\in \mathbb{R}^n$. 
    Then $\overline{\tc_v}= \tc_v * \mod^{ss}_vA$ and $\overline{\fc_v}= \mod_A^{ss}A*\fc_v$.
\end{proposition}

A torsion class $\tc\in \tors A$ is said to be a numerical torsion class if there exists $v\in \mathbb{R}^n$ such that $\tc= \tc_v$ or $\tc=\overline{\tc_v}$. 
We denote $\tors^{ss}(A)$ the set of all numerical torsion classes.
Taking $v=0$, it is easy to see that both $\{0\}$ and $\modA$ belong to $\tors^{ss}(A)$.
In fact, it has been shown in \cite{BST, Asai} that a vector $v\in \mathbb{R}^n$ lies in a chamber if and only if $\tc_v = \overline{\tc_v}$ and $\tc_v$ is functorially finite.
It follows from the results in \cite{DIJ, Asai, BST} that $\tors A=\tors^{ss}(A)$ if and only if there are finitely many chambers in $\Df(A)$.%the wall-and-chamber structure $\Df(A)$ of $A$.

\begin{definition}\cite[Section~2]{Asai}
    Given two stability conditions $v, v'\in \mathbb{R}^n$ we say that $v$ and $v'$ are TF-equivalent if $\tc_v=\tc_{v'}$ and $\overline{\tc_v}=\overline{\tc_{v'}}$. 
\end{definition}

Note that if $v$ and $v'$ are two TF-equivalent stability conditions, then $\mod_{v}^{ss}A = \mod_{v'}^{ss}A$. 
However, the converse is not true.

Let $v$ be a stability condition. 
Then we denote by $\df_v$ the TF-equivalence class of $v$. 
\[
\df_v=\{v'\in \mathbb{R}^n \mid v' \text{ is TF-equivalent to $v$}\}. 
\] 
The following is a direct consequence of \cite[Theorem~2.17]{Asai}.

\begin{proposition}\label{prop: union of cones}
    Let $A$ be an Artin algebra. 
    Then, for every $v\in \mathbb{R}^n$ the TF-equivalence class $\df_v$ is an open convex cone in $\mathbb{R}^n$. 
    Moreover, \[
    \Df(A)= \bigcup_{v\in \mathbb{R}^n} \df_v.
    \]
\end{proposition}

%%%%%%%%%%%%%%
\subsection{Wall-and-chamber structures for subcategories of bounded length}\label{subsec: W&C modA_l}
The previous discussion recalled the classical wall-and-chamber structure.
In this subsection we introduce and study the wall-and-chamber structure $\Df_\ell(A)$ of $(\modA)_\ell$. 

\begin{definition}
A wall in the wall-and-chamber structure $\Df_\ell(A)$ associated with $(\modA)_\ell$ is the stability locus $\dc(M)$ of a nonzero object $M \in (\modA)_\ell$ whenever $\dc(M)$ is of codimension $1$. 
A chamber in $\Df_\ell(A)$ is a connected open subset of 
\[
\mathbb{R}^n \setminus \bigcup_{M \in (\modA)_\ell} \dc(M).
\] 
\end{definition}

\begin{remark}
Note that, unlike in the definition of $\Df(A)$, we do not define chambers as the connected components of the complement of the closure of the union of all walls. 
This is because, for every algebra $A$ and every $\ell \in\mathbb{N}$, the wall-and-chamber structure $\Df_\ell(A)$ always has only finitely many walls and hence the union of all walls is a closed set in $\mathbb{R}^n$.
    Indeed, for any positive integer $\ell$, there are finitely many classes $[M] \in K_0(A)$ such that $\sum_{i=1}^n [M]_i \leq \ell$.
    Since the stability locus $\dc(M)$ of $M$ is determined by its class $[M]$ and the class of its submodules in $K_0(A)$ there cannot be infinitely many different walls. 
\end{remark}

As we saw in \cref{sec:l-torsion classes}, there is an equivalence relation $\sim_\ell$ on $\tors (A)$ defined by $\tc \sim_\ell \tc'$ if $\tc\cap (\modA)_\ell = \tc'\cap (\modA)_\ell$. 
This equivalence relation restricts to $\tors^{ss}(A)$, the set of numerical torsion classes.
We denote the quotient set by $\tors_\ell^{ss}A$.
We now introduce the corresponding notion of $\ell$-TF-equivalence.

\begin{definition}
Given two stability conditions $v, v'\in \mathbb{R}^n$, we say that $v$ and $v'$ are $\ell$-TF-equivalent if $(\tc_v)_\ell=(\tc_{v'})_\ell$ and $(\overline{\tc_v})_\ell=(\overline{\tc_{v'}})_\ell$.
\end{definition}

\begin{remark}
    It follows from \cref{cor:sametorsion-free} that two stability conditions $v, v'$ are $\ell$-TF-equivalent if and only if $(\fc_v)_\ell=(\fc_{v'})_\ell$ and $(\overline{\fc_v})_\ell=(\overline{\fc_{v'}})_\ell$.
\end{remark}

For $v, v' \in \mathbb{R}^n$, denote by $[v,v']$ the line segment $[v,v']:=\{\lambda v' + (1-\lambda)v \mid\lambda\in\mathbb{R} \text{ and } 0\leq \lambda\leq 1\}$.
The following result is a generalization of \cite[Theorem~2.17]{Asai}.
The proof is essentially the same as the original with some minor technical adaptations. 
We include it for the sake of completeness.

\begin{theorem}\label{thm:bounded l-TF}
Let $v, v' \in \mathbb{R}^n$ be two distinct stability conditions. 
Then the following are equivalent. 
\begin{enumerate}
    \item The stability conditions $v$ and $v'$ are $\ell$-TF-equivalent.
    \item The stability condition $v''$ is $\ell$-TF-equivalent to $v$ and $v'$ for every $v''\in[v,v']$.
    \item For every $v'' \in [v,v']$ the category $\left(\mod^{ss}_{v''}A \right)_\ell$ is constant.
    \item For every $A$-module $M$ of length at most $\ell$ then either $[v,v'] \cap \dc(M) = \varnothing$ or \mbox{$[v,v'] \subset \dc(M)$} holds.
    \item There is no brick $B\in (\modA)_\ell$ such that $[v,v'] \cap \dc(B)$ has exactly one element.
\end{enumerate}
\end{theorem}

\begin{proof}
\textit{(1) implies (2).}
Let $v''\in[v, v']$ and let $M\in (\tc_v)_\ell$ and let $N$ be a quotient of $M$.
Then 
\[\langle v'', [N]\rangle = \langle \lambda v' + (1-\lambda)v, [N]\rangle = \lambda\langle  v', [N]\rangle+ (1-\lambda)\langle v, [N]\rangle >0,\]
since every quotient of $M$ has positive pairing with $v''$. 
This implies $(\tc_v)_\ell \subset (\tc_{v''})_\ell$ and, as a consequence, $(\overline{\fc_{v''}})_\ell \subset (\overline{\fc_v})_\ell$. 
Similarly, it can be shown that $(\overline{\tc_v})_\ell \subset (\overline{\tc_{v''}})_\ell$ and $({\fc_{v''}})_\ell \subset ({\fc_v})_\ell$.
Consider a subobject $L$ of an object $M' \in (\fc_v)_\ell=(\fc_{v'})_\ell$.
Then 
\[\langle v'', [L]\rangle = \langle \lambda v' + (1-\lambda)v, [L]\rangle = \lambda\langle  v', [L]\rangle+ (1-\lambda)\langle v, [L]\rangle <0.\]
From this we conclude that $(\fc_v)_\ell\subset({\fc_{v''}})_\ell$ and, as a consequence, $(\overline{\tc_{v''}})_\ell \subset (\overline{\tc_{v}})_\ell$. 
Combining this with the above we obtain the equalities $(\fc_v)_\ell=({\fc_{v''}})_\ell$ and $(\overline{\tc_{v''}})_\ell = (\overline{\tc_{v}})_\ell$. 
Similarly we can prove $({\tc_{v''}})_\ell = (\tc_{v})_\ell$. 
This shows that  every $v''\in [v,v']$ is $\ell$-TF-equivalent to $v$.
\\

\noindent
\textit{(2) implies (3).} Let $v'' \in [v,v']$. Then the following equalities hold.
\begin{align*}
    (\mod^{ss}_{v''}A)_\ell & = (\overline{\tc_{v''}})_\ell \cap (\overline{\fc_{v''}})_\ell\\
    & = (\overline{\tc_{v}})_\ell \cap (\overline{\fc_{v}})_\ell\\
    %&= \overline{\tc_{v}}\cap (\modA)_\ell \cap \overline{\fc_{v}} \cap (\modA)_\ell\\
    & = (\mod^{ss}_{v}A)_\ell 
\end{align*}

\noindent
\textit{(3) implies (4).} It is enough to recall that an object $M\in (\modA)_\ell$ is $v$-semistable if and only if $v\in \dc(M)$ by definition.  \\

\noindent
\textit{(4) implies (5).} This is obvious.\\

\noindent
\textit{(5) implies (1).}
Assume to the contrary that $v$ and $v'$ are not $\ell$-TF-equivalent. We need to show the existence of a brick $B\in (\modA)_\ell$ such that $\dc(B)\cap[v,v']$ consists of only one element.

Since $v$ and $v'$ are not $\ell$-TF-equivalent we can assume that $(\overline{\tc_v})_\ell\neq (\overline{\tc_{v'}})_\ell$. 
By \cref{prop:objectintheintersection} we have that there is a brick $B$ in either $(\overline{\tc_v})_\ell\cap(\fc_{v'})_\ell$ or $({\tc_v})_\ell\cap(\overline{\fc_{v'}})_\ell$.
Without loss of generality, we can assume that $(\overline{\tc_v})_\ell\cap(\fc_{v'})_\ell$ contains a brick $B$.
Furthermore, we can assume that $B$ is of minimal length in $(\overline{\tc_v})_\ell\cap(\fc_{v'})_\ell$.

Consider the function $f_B : [0,1] \to \mathbb{R}$ defined as \[f_B(\lambda) = \langle \lambda v' + (1-\lambda)v,[B]\rangle.\]
This is clearly a linear function on $[0,1]$. 
Moreover $f_B(0) = \langle v, [B]\rangle \geq 0$ since $B\in \overline{\tc_v}$. 
Similarly $f_B(1) = \langle v', [B]\rangle < 0$ since $B \in \fc_{v'}$. 
Then there is a unique $\alpha \in [0,1]$ such that $f_B(\alpha)=0$. 
Let us denote $v''= \alpha v' + (1-\alpha)v$.

We claim that $B$ is $v''$-semistable. By definition of $v''$ we have $\langle v'', [B]\rangle =0$. Let $N$ be a quotient of $B$.
We need to show that $\langle v'', [N]\rangle \geq 0$.
By definition of $v''$ and the linearity of $\langle -, -\rangle$ we have
\begin{align*}
    \langle v'', [N]\rangle &=\alpha\langle v', [N]\rangle +(1-\alpha)\langle v, [N]\rangle 
    % \\
    % &= \alpha\langle v', [N]\rangle +(1-\alpha)(\langle v, [M]\rangle-\langle v, [M/L]\rangle) \\
    % &=\alpha\langle v', [L]\rangle -(1-\alpha)\langle v, [M/L]\rangle.
\end{align*}
Since $\alpha \in [0,1]$, it is enough to show that $\langle v, [N]\rangle \geq 0$ and $\langle v', [N]\rangle \geq 0$.
This is equivalent to show that $N \in (\overline{\tc_{v}})_\ell\cap (\overline{\tc_{v'}})_\ell$ for every quotient $N$ of $B$.

By hypothesis $B \in (\overline{\tc_{v}})_\ell$, so $N \in (\overline{\tc_{v}})_\ell$ because $(\overline{\tc_{v}})_\ell$ is closed under quotients.
Suppose now that there exists a proper quotient $N$ of $B$ such that $N \not \in (\overline{\tc_{v'}})_\ell$ and take the short exact sequence \[
0 \to t_{v'} N \to N \to f_{v'}N \to 0
\] of $N$ with respect to $\left((\overline{\tc_{v'}})_\ell, (\fc_{v'})_\ell\right)$. 
This sequence exists by \cref{thm: torsion pairs of bounded length}. 
Note that $f_{v'}N \in (\fc_{v'})_\ell$ by definition and $f_{v'}N\in (\overline{\tc_{v}})_\ell$ because $f_{v'}N$ is a quotient of $B$.
Then $f_{v'}N\in (\overline{\tc_{v}})_\ell\cap (\fc_{v'})_\ell$.
Also, since $N$ is a proper quotient of $B$ we get $\lg(f_{v'}N) \leq \lg(N) < \lg(B)$. 
But this contradicts the minimality of the length of $B$ in the category $(\overline{\tc_{v}})_\ell\cap (\fc_{v'})_\ell$, so $N\in (\overline{\tc_{v'}})_\ell$. 
Therefore 
\begin{align*}
    \langle v'', [N]\rangle &=\alpha\langle v', [N]\rangle +(1-\alpha)\langle v, [N]\rangle \geq 0.
    % \\
    % &= \alpha\langle v', [N]\rangle +(1-\alpha)(\langle v, [M]\rangle-\langle v, [M/L]\rangle) \\
    % &=\alpha\langle v', [L]\rangle -(1-\alpha)\langle v, [M/L]\rangle.
\end{align*}
Hence $B$ is $v''$-semistable, and therefore
$v''\in\dc(B)$.

Since $f_B$ is linear and vanishes only at $\alpha$, we obtain
$\dc(B)\cap[v,v']=\{v''\}$, contradicting (5).
This completes the proof.
%Hence, whenever there are two stability conditions $v,v'\in \mathbb{R}^n$ that are not $\ell$-TF--equivalent there is a brick $B\in(\modA)_\ell$ such that $[v,v']\cap \dc(B)$. 
%this finishes the proof. 
\end{proof}

Given a stability condition $v \in \mathbb{R}^n$ we define the set
\[
\df_v^\ell := \left\{v' \in \mathbb{R}^n \mid v' \text{ is $\ell$-TF-equivalent to }v\right\}.
\]

\begin{corollary}\label{cor:finite cone complex for l}
    The set $\df_v^\ell$ is a convex cone in $\mathbb{R}^n$ for every $v\in \mathbb{R}^n$, 
    \[
    \Df_\ell(A) = \bigcup_{v \in \mathbb{R}^n} \df_v^\ell,
    \]
    and the number of $\ell$-TF-equivalence classes $\{\df_v^\ell \mid v\in \mathbb{R}^n\}$ is finite.
    In particular, the poset $\tors^{ss}_\ell(A)$ is finite.
\end{corollary}

\begin{proof}
    It follows directly from \cref{thm:bounded l-TF} that $\df_v^\ell$ is convex. 
    Moreover, it follows from the definitions that $\overline{\tc_v}=\overline{\tc_{\lambda v}}$ and $\tc_v=\tc_{\lambda v}$ for every $\lambda \geq 0$. Hence $\df_v^\ell$ is a cone.
    The fact that the number of $\ell$-TF-equivalence classes is finite follows from \cref{thm:bounded l-TF} and the fact that there are only finitely many classes $[B]\in K_0(A)$ represented by bricks of length at most $\ell$.
    By Theorem \ref{thm:bounded l-TF}, two stability conditions fail to be $\ell$-TF-equivalent precisely when the segment joining them crosses the wall of a brick in $(\modA)_\ell$.
    Hence the decomposition of $\Df_\ell(A)$ into $\ell$-TF-equivalence classes is completely determined by the stability loci of such bricks.
    Moreover given a $B\in (\modA)_\ell$ the stability locus $\dc(B)$ is the union of finitely many $\ell$-TF-equivalence classes, implying the equality in our statement. 
    The final statement follows directly from the definition of the poset $\tors^{ss}_\ell(A)$.
\end{proof}

\begin{corollary}
Let $\cc$ be a chamber in $\Df_\ell(A)$. Then $\cc = \df_v^\ell$ for any $v\in \cc$. 
\end{corollary}

\begin{proof}
Let $v'\in \cc$. 
Since $\cc$ is a chamber there is no $M\in (\modA)_\ell$ that is $v'$-semistable.
Then $v$ and $v'$ are $\ell$-TF-equivalent by \cref{thm:bounded l-TF} and $\df_v^\ell\subset \cc$.

By the definition of a chamber, for every stability condition $v'' \in \overline{\cc}\setminus \cc$ there is a nonzero object $M\in (\modA)_\ell$ which is $v''$-semistable.
In particular $v''$ is not $\ell$-TF-equivalent to $v$.
Now, for every $v' \in \mathbb{R}^n \setminus \overline{\cc}$ there is an element $v'' \in[v, v']$ such that $v'' \in \dc(M)$ for some $M \in (\modA)_\ell$.
Therefore $v'$ is not $\ell$-TF-equivalent to $v$ by \cref{thm:bounded l-TF}.
Hence $\cc = \df_v^\ell$ as we wanted to show. 
\end{proof}

Inspired by the terminology from cluster algebras, we say that a cone $\df$ is sign coherent if for any two vectors $v=(v_1, \dots, v_n)$ and  $w=(w_1, \dots, w_n)$ in $\df$, $v_i \geq 0$ if and only if $w_i\geq 0$ or $v_i \leq 0$ if and only if $w_i\leq 0$ holds for all $1 \leq i \leq n$.

\begin{corollary}
Let $\df$ be a cone in $\Df_\ell(A)$. 
Then $\df$ is sign coherent. 
\end{corollary}

\begin{proof}
Let $v, w \in \mathbb{R}^n$ such that $v_i <0$ while $w_i>0$ and consider the interval $[v, w]$ between them. 
Then there exists a vector $v' \in [v,w]$ such that $v'_i =0$. 
An easy verification shows that $v'_i$ belongs to the stability space $\dc(S_i)$ of the simple module $S_i$, which is in $(\modA)_\ell$ and the result follows from \cref{thm:bounded l-TF}.
\end{proof}

We now give a very explicit description of the poset $\tors_\ell^{ss}A$ in terms of the cone complex $\Df_\ell(A)$.

\begin{proposition}\label{prop: tors to chambers}
Suppose that $\tc_\ell\in \tors_\ell A$ satisfies $\tc_\ell = (\tc_v)_\ell$ or $\tc_\ell = (\overline{\tc}_v)_\ell$ for some $v\in \mathbb{R}^n$. 
Then there is a chamber $\df \in \Df_\ell(A)$ such that $$\tc_\ell = (\tc_\df)_\ell = (\overline{\tc}_\df)_\ell.$$
\end{proposition}

\begin{proof}
We only show the case where $\tc_\ell = (\tc_v)_\ell$ for some $v\in \mathbb{R}^n$, the other case is dual.
In order to prove our statement, it is enough to show that there is a vector $w \in \mathbb{R}^n$ with no nonzero $w$-semistable objects $M\in (\modA)_\ell$ such that $\tc_\ell = (\tc_w)_\ell$.
Then the $\ell$-TF-equivalence class $\df^\ell_w$ is the desired chamber.
 Indeed, because there are finitely many walls in the wall-and-chamber of $(\modA)_\ell$, every connected component in the complement of the walls is necessarily open in $\mathbb{R}^n$.

Consider the torsion pair $(\tc_\ell, \fc_\ell) = ((\tc_v)_\ell, (\overline{\fc}_v)_\ell)$.
By definition, we have that \mbox{$\langle v , [T] \rangle > 0$} for every nonzero object $T \in \tc_\ell$ and $\langle v, [F] \rangle \leq 0$ for every $F \in \fc_\ell$. 
Note that the set
\(
\{[T] \in K_0(A) \mid T\in \tc_\ell\} \)
is finite. 
Then $\epsilon =\min \{\langle v , [T] \rangle \mid 0 \neq T \in \tc_\ell\}$  is well-defined and greater than $0$.
Let $w = v - \frac{\epsilon}{2\ell} (1,1,\dots, 1) \in \mathbb{R}^n$.

Then for every $T \in \tc_\ell$ we have that 
\[
\langle w , [T] \rangle = \langle v , [T] \rangle - \frac{\epsilon}{2\ell} \langle (1,1,\dots, 1), [T] \rangle  = \langle v , [T] \rangle - \frac{\epsilon}{2\ell} \lg(T) \geq \langle v , [T] \rangle - \frac{\epsilon}{2} > 0
\]
where the inequality holds because $\lg(T) \leq \ell$ for all $T \in \tc_\ell$.

Similarly, let $F\in \fc_\ell$ be nonzero.
Then 
\[
\langle w , [F]\rangle = \langle v , [F] \rangle - \frac{\epsilon}{2\ell} \langle (1,1,\dots, 1), [F] \rangle  = \langle v , [F] \rangle - \frac{\epsilon}{2\ell} \lg(F) \leq 0 - \frac{\epsilon}{2\ell} \lg(F) < 0. 
\]

The above inequalities imply the following inclusions. 
\begin{align*}
\tc_\ell \subset (\tc_w)_\ell &\subset (\overline{\tc}_w)_\ell \\
\fc_\ell \subset (\fc_w)_\ell &\subset (\overline{\fc}_w)_\ell 
\end{align*}
Moreover, we know that $(\tc_\ell, \fc_\ell)$ and $((\tc_w)_\ell, (\overline{\fc}_w)_\ell)$ are torsion pairs in $(\modA)_\ell$, implying $(\overline{\fc}_w)_\ell \subset \fc_\ell$. 
Then
$\tc_\ell = (\tc_w)_\ell = (\overline{\tc}_w)_\ell$. 
\end{proof}

As a consequence of the previous proposition and the fact that $\Df_\ell(A)$ is complete, we can describe maximal inclusions in $\tors^{ss}_\ell A$.

\begin{corollary}\label{cor: maximal inclusions tors^ss}
Let $\tc_\ell, \tc'_\ell \subset \tors_\ell A$ be such that $\tc_\ell$ is maximally included in $\tc'_\ell$. 
Then there is a $\ell$-$TF$-equivalence class $\df'$ codimension $1$ and two chambers $\df_1, \df_2$ in $\Df_\ell(A)$ such that 
$\tc_\ell = (\tc_{\df'})_\ell = (\tc_{\df_1})_\ell$ and $\tc'_\ell = (\overline{\tc}_{\df'})_\ell = (\tc_{\df_2})_\ell$.
\end{corollary}

Given two cone complexes $\Df$ and $\widetilde{\Df}$, we say that $\Df$ is finer than $\widetilde{\Df}$ if every cone in $\widetilde{\Df}$ is the union of cones in $\Df$.
In this case we say that $\Df$ is a refinement of $\widetilde{\Df}$.
Suppose that $\ell ,m \in \mathbb{N}$ and $\ell \leq m$. 
Then \cref{thm:bounded l-TF} can be used to compare $\Df_\ell(A)$ and $\Df_m(A)$. 

\begin{corollary}\label{cor:refinement of W&C}
Let $A$ be an Artin algebra and let $\ell,m \in \mathbb{N}$ such that $\ell < m$. 
Then $\Df_m(A)$ is a refinement of $\Df_\ell(A)$.
Moreover, $\Df(A)$ is a refinement of $\Df_\ell(A)$ for every nonnegative integer $\ell$.
\end{corollary}

\begin{proof}
It follows from \cref{thm:bounded l-TF} that the structure of $\Df_\ell(A)$ is completely determined by the walls $\dc(B)$ for all bricks $B\in (\modA)_\ell$.
If $\ell<m$, then every brick $B\in (\modA)_\ell$ satisfies $B\in (\modA)_m$ and the result follows. 
The moreover part of the statement is obvious.
\end{proof}

If we draw an arrow $\phi:\widetilde{\Df} \to \Df$ whenever $\widetilde{\Df}$ is finer than $\Df$, then \cref{cor:refinement of W&C} is illustrated in \cref{fig:inverse limits W&C} following the style of \cref{fig:inverse limits}.
Note that this diagram is commutative and the following corollary is immediate.

\begin{corollary}\label{cor:limit2}
    \[
     \underleftarrow\lim \tors^{ss}_\ell A =\tors^{ss} A.
    \]
\end{corollary}

\begin{figure}
    \centering
% https://q.uiver.app/#q=WzAsNixbOCwwLCIoXFx0b3JzIEEpXzAiXSxbNiwwLCIoXFx0b3JzIEEpXzEiXSxbNCwwLCJcXGNkb3RzIl0sWzIsMCwiIChcXHRvcnMgQSlfbCJdLFsyLDIsIlxcdG9ycyBBIl0sWzAsMCwiXFxjZG90cyJdLFszLDIsIlxcdmFycGhpX3tsLGwtMX0iLDAseyJzdHlsZSI6eyJoZWFkIjp7Im5hbWUiOiJlcGkifX19XSxbMiwxLCJcXHZhcnBoaV97MiwxfSIsMCx7InN0eWxlIjp7ImhlYWQiOnsibmFtZSI6ImVwaSJ9fX1dLFsxLDAsIlxcdmFycGhpX3sxLDB9IiwwLHsic3R5bGUiOnsiaGVhZCI6eyJuYW1lIjoiZXBpIn19fV0sWzQsMywiXFx2YXJwaGlfe1xcaW5mdHksbH0iLDAseyJzdHlsZSI6eyJoZWFkIjp7Im5hbWUiOiJlcGkifX19XSxbNCwxLCJcXHZhcnBoaV97XFxpbmZ0eSwxfSIsMCx7InN0eWxlIjp7ImhlYWQiOnsibmFtZSI6ImVwaSJ9fX1dLFs0LDAsIlxcdmFycGhpX3tcXGluZnR5LDB9IiwyLHsic3R5bGUiOnsiaGVhZCI6eyJuYW1lIjoiZXBpIn19fV0sWzUsMywiXFx2YXJwaGlfe2wrMSxsfSJdXQ==
\[\begin{tikzcd}
	\cdots && { \Df_\ell(A)} && \cdots && {\Df_2(A)} && {\Df_1(A)} \\
	\\
	&& {\Df(A)}
	\arrow["{\phi_{l+1,l}}", from=1-1, to=1-3]
	\arrow["{\phi_{l,l-1}}", from=1-3, to=1-5]
	\arrow["{\phi_{3,2}}", from=1-5, to=1-7]
	\arrow["{\phi_{2,1}}", from=1-7, to=1-9]
	\arrow["{\phi_{\infty,l}}", from=3-3, to=1-3]
	\arrow["{\phi_{\infty,2}}", from=3-3, to=1-7]
	\arrow["{\phi_{\infty,1}}"', from=3-3, to=1-9]
\end{tikzcd}\]
    \caption{$\Df(A)$ as the inverse limit of the cone complexes $\Df_\ell(A)$ for $\ell\in\mathbb{N}$.}
    \label{fig:inverse limits W&C}
\end{figure}
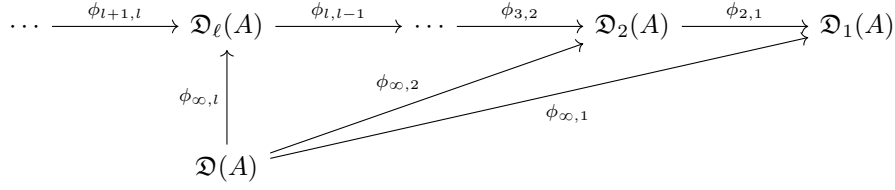

Finally, \cref{thm:bounded l-TF} also allow us to give a result analogous to \cref{cor:tors_l tau-tilting finite} in terms of wall-and-chamber structures.

\begin{corollary}\label{cor:W&C tau-tilting finite}
An Artin algebra $A$ is $\tau$-tilting finite if and only if there is a $\ell \in \mathbb{N}$ such that $\Df(A) = \Df_\ell(A)$.
\end{corollary}

\begin{remark}\label{rem: differences}
If $A$ is a Nakayama algebra then the indecomposable objects of $(\modA)_\ell$ are in bijection with the indecomposable modules of $\mod A_\ell$, where $A_\ell = A/\operatorname{rad}^\ell(A)$.
Therefore, if $A$ is a Nakayama algebra then every chamber $\cc \in \Df_\ell(A)$ is a simplicial cone by \cite{DIJ, BST, Asai}.
For more information on Nakayama algebras, see \cite[Chapter~5]{bluebook1}.
In general, the converse is not true, as it can be seen in \cref{ex: no li chamber}.

Suppose that $\cc_1, \cc_2$ are two different chambers in $\Df(A)$ such that $\overline{\cc_1} \cap \overline{\cc_2} = \df$ is a cone of codimension $1$. 
Then it follows from the definitions that there is a brick $B\in \modA$ such that $\df \subset [B]^\perp$. 
It has been shown in \cite{Treffinger_c-vectors} that such a brick $B$ is unique up to isomorphism for every algebra $A$, see also \cite{AsaiSemibricks}. 
This is no longer true for $\Df_\ell(A)$ in general, as it can be seen in \cref{ex: Kronecker 2}.
For instance in \cref{fig:Df_3} there are two chambers that are separated by the wall $\dc\left(\rep{1\\2}_\lambda\right)$, which is the stability locus of $\rep{1\\2}_\lambda$ for all $\lambda \in \mathbb{P}^1(k)$.
\end{remark}

We finish this section with a couple of examples that illustrate the previous remark.

\begin{example}\label{ex: no li chamber}
Let $A$ be the path algebra $\mathbb{C}Q$ of the quiver \(Q=\begin{tikzcd}
	1 & 2 & 3
	\arrow[from=1-1, to=1-2]
	\arrow[from=1-3, to=1-2]
\end{tikzcd}\).
There are six indecomposable $A$-modules. 
The Auslander-Reiten quiver of $A$ is as follows. 
% https://q.uiver.app/#q=WzAsNixbMCwxLCJcXGJ1bGxldCJdLFsxLDAsIlxcYnVsbGV0Il0sWzEsMiwiXFxidWxsZXQiXSxbMiwxLCJcXGJ1bGxldCJdLFszLDAsIlxcYnVsbGV0Il0sWzMsMiwiXFxidWxsZXQiXSxbMCwxXSxbMCwyXSxbMiwzXSxbMSwzXSxbMCwzLCIiLDEseyJzdHlsZSI6eyJib2R5Ijp7Im5hbWUiOiJkb3R0ZWQifSwiaGVhZCI6eyJuYW1lIjoibm9uZSJ9fX1dLFszLDRdLFszLDVdLFsxLDQsIiIsMSx7InN0eWxlIjp7ImJvZHkiOnsibmFtZSI6ImRvdHRlZCJ9LCJoZWFkIjp7Im5hbWUiOiJub25lIn19fV0sWzIsNSwiIiwxLHsic3R5bGUiOnsiYm9keSI6eyJuYW1lIjoiZG90dGVkIn0sImhlYWQiOnsibmFtZSI6Im5vbmUifX19XV0=
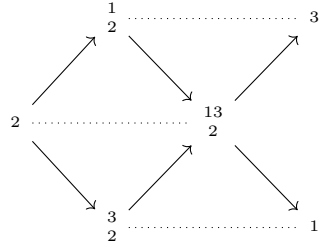
\[\begin{tikzcd}
	& \rep{1\\2} && \rep{3} \\
	\rep{2} && \rep{13\\2}\\
	& \rep{3\\2} && \rep{1}
	\arrow[dotted, no head, from=1-2, to=1-4]
	\arrow[from=1-2, to=2-3]
	\arrow[from=2-1, to=1-2]
	\arrow[dotted, no head, from=2-1, to=2-3]
	\arrow[from=2-1, to=3-2]
	\arrow[from=2-3, to=1-4]
	\arrow[from=2-3, to=3-4]
	\arrow[from=3-2, to=2-3]
	\arrow[dotted, no head, from=3-2, to=3-4]
\end{tikzcd}\]
Since every indecomposable $A$-module has length at most three, the algebra $A$ has exactly three wall-and-chamber structures associated to it, all of which are depicted in \cref{fig:no li chamber}.
Note that in $\Df_2(A)$ there is a chamber, shaded in green in \cref{fig:A3_2}, bounded by the walls $\dc(\rep{1})$, $\dc(\rep{3})$, $\dc(\rep{1\\2})$ and $\dc(\rep{3\\2})$. 
In particular, $\Df_2(A)$ is not a simplicial cone complex.

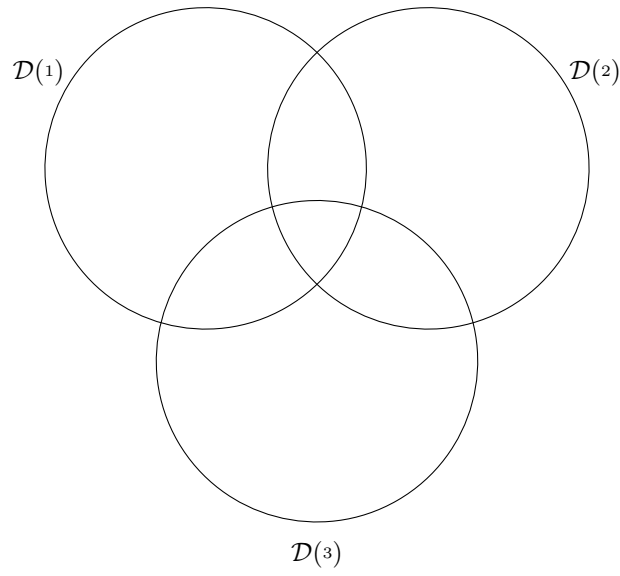
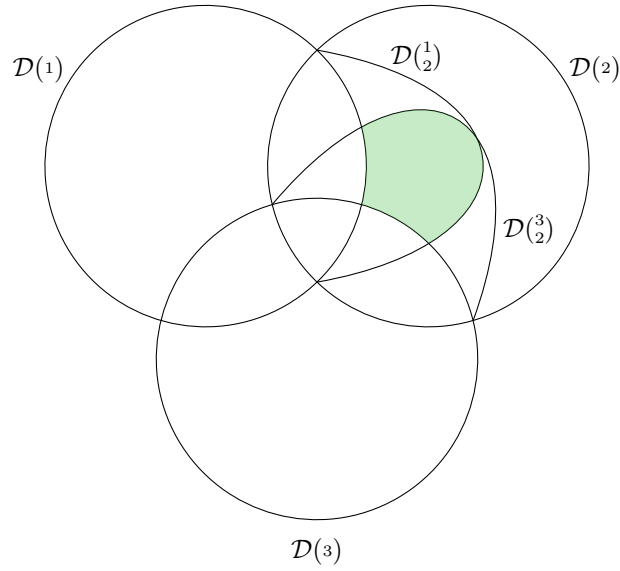
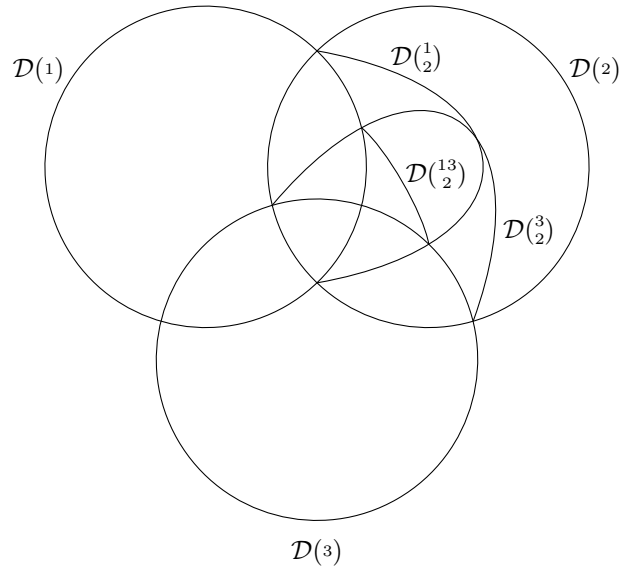
\begin{figure}

 \begin{subfigure}[t]{\textwidth}
 \centering
 \tikzset{invclip/.style={clip,insert path={{[reset cm]
      (-16383.99999pt,-16383.99999pt) rectangle (16383.99999pt,16383.99999pt)
    }}}}
\begin{tikzpicture}[rotate=30, scale=0.85]
%  \draw [help lines] (4,4) grid (-4,-4);
  \coordinate (c1) at (0:2cm);
  \coordinate (c2) at (120:2cm);
  \coordinate (c3) at (240:2cm);
  \draw node at (0:5) {$\dc(\rep{2})$};
  \draw node at (120:5) {$\dc(\rep{1})$};
  \draw node at (240:5) {$\dc(\rep{3})$};  
%  \draw node at (30:3.1) {$\dc(\rep{1\\2})$};
%  \draw node at (-30:3.1) {$\dc(\rep{3\\2})$};  
  \draw [name path = d1] (c1) circle [radius = 2.5cm];
  \draw [name path = d2] (c2) circle [radius = 2.5cm];  
  \draw [name path = d3, label=left:$3$] (c3) circle [radius = 2.5cm];  
%  \draw [red, thick, name intersections={of=d1 and d2}]
%   (intersection-1) to [bend left] (c1) to [bend left]  (intersection-2);
%  \draw [name path = d12, name intersections={of=d1 and d2}, label = $D12$] 
%  (intersection-1) .. controls +(320:3cm) and +(340:3cm) .. (intersection-2);
%   \draw [name path = d31, name intersections={of=d1 and d3}] 
%  (intersection-1) .. controls +(20:3cm) and +(40:3cm) .. (intersection-2);
%   \clip [name path = d12, name intersections={of=d1 and d2}] 
%  (intersection-1) .. controls +(320:3cm) and +(340:3cm) .. (intersection-2);
%  \clip [name path = d31, name intersections={of=d1 and d3}] 
%  (intersection-1) .. controls +(20:3cm) and +(40:3cm) .. (intersection-2);
%  \clip [invclip] (c2) circle [radius = 2.5cm];
%  \clip [invclip] (c3) circle [radius = 2.5cm];
%  \fill [opacity=0.2](c3) circle [radius = 2.5cm] (c2) circle [radius = 2.5cm];
%  \fill [opacity=0.5, color=green!50!blue, name intersections={of=d2 and d31, by={i1}}, name intersections={of=d31 and d12, by={i2}}, name intersections={of=d3 and d12, by={i3}}, name intersections={of=d3 and d2, by={i4, i5}}] 
%  (i1) .. controls +(20:3cm) and +(40:3cm) .. (i2).. controls +(320:3cm) and +(340:3cm) ..(i3)--(i4);
 \end{tikzpicture}
 \caption{$\Df_1(A)$}
\label{fig:A3_1}
\end{subfigure}

 \begin{subfigure}[t]{\textwidth}
 \centering
 \tikzset{invclip/.style={clip,insert path={{[reset cm]
      (-16383.99999pt,-16383.99999pt) rectangle (16383.99999pt,16383.99999pt)
    }}}}
\begin{tikzpicture}[rotate=30, scale=0.85]
%  \draw [help lines] (4,4) grid (-4,-4);
  \coordinate (c1) at (0:2cm);
  \coordinate (c2) at (120:2cm);
  \coordinate (c3) at (240:2cm);
  \draw node at (0:5) {$\dc(\rep{2})$};
  \draw node at (120:5) {$\dc(\rep{1})$};
  \draw node at (240:5) {$\dc(\rep{3})$};  
  \draw node at (30:3.1) {$\dc(\rep{1\\2})$};
  \draw node at (-30:3.3) {$\dc(\rep{3\\2})$};  
  \draw [name path = d1] (c1) circle [radius = 2.5cm];
  \draw [name path = d2] (c2) circle [radius = 2.5cm];  
  \draw [name path = d3, label=left:$3$] (c3) circle [radius = 2.5cm];  
%  \draw [red, thick, name intersections={of=d1 and d2}]
%   (intersection-1) to [bend left] (c1) to [bend left]  (intersection-2);
  \draw [name path = d12, name intersections={of=d1 and d2}, label = $D12$] 
  (intersection-1) .. controls +(320:3.5cm) and +(340:3.5cm) .. (intersection-2);
   \draw [name path = d31, name intersections={of=d1 and d3}] 
  (intersection-1) .. controls +(20:3.5cm) and +(40:3.5cm) .. (intersection-2);
   \clip [name path = d12, name intersections={of=d1 and d2}] 
  (intersection-1) .. controls +(320:3.5cm) and +(340:3.5cm) .. (intersection-2);
  \clip [name path = d31, name intersections={of=d1 and d3}] 
  (intersection-1) .. controls +(20:3.5cm) and +(40:3.5cm) .. (intersection-2);
  \clip [invclip] (c2) circle [radius = 2.125cm];
  \clip [invclip] (c3) circle [radius = 2.125cm];
  \fill [opacity=0.3, color=green!50!gray, name intersections={of=d2 and d31, by={i1}}, name intersections={of=d31 and d12, by={i2}}, name intersections={of=d3 and d12, by={i3}}, name intersections={of=d3 and d2, by={i4, i5}}] 
  (i1) .. controls +(20:3cm) and +(40:3cm) .. (i2).. controls +(320:3cm) and +(340:3cm) ..(i3)--(i4);
 \end{tikzpicture}
 \caption{$\Df_2(A)$}
\label{fig:A3_2}
\end{subfigure}

\begin{subfigure}[t]{\textwidth}
\centering
\begin{tikzpicture}[rotate=30, scale=0.85]
%  \draw [help lines] (4,4) grid (-4,-4);
  \coordinate (c1) at (0:2cm);
  \coordinate (c2) at (120:2cm);
  \coordinate (c3) at (240:2cm);
  \draw node at (0:5) {$\dc(\rep{2})$};
  \draw node at (120:5) {$\dc(\rep{1})$};
  \draw node at (240:5) {$\dc(\rep{3})$};  
  \draw node at (30:3.1) {$\dc(\rep{1\\2})$};
  \draw node at (-30:3.3) {$\dc(\rep{3\\2})$};  
  \draw node at (-5:2.05) {$\dc(\rep{13\\2})$};  
  \draw [name path = d1] (c1) circle [radius = 2.5cm];
  \draw [name path = d2] (c2) circle [radius = 2.5cm];  
  \draw [name path = d3, label=left:$3$] (c3) circle [radius = 2.5cm];  
%  \draw [red, thick, name intersections={of=d1 and d2}]
%   (intersection-1) to [bend left] (c1) to [bend left]  (intersection-2);
  \draw [name path = d12, name intersections={of=d1 and d2}, label = $D12$] 
  (intersection-1) .. controls +(320:3.5cm) and +(340:3.5cm) .. (intersection-2);
   \draw [name path = d31, name intersections={of=d1 and d3}] 
  (intersection-1) .. controls +(20:3.5cm) and +(40:3.5cm) .. (intersection-2);
%   \clip [name path = d12, name intersections={of=d1 and d2}] 
%  (intersection-1) .. controls +(320:3cm) and +(340:3cm) .. (intersection-2);
%  \clip [name path = d31, name intersections={of=d1 and d3}] 
%  (intersection-1) .. controls +(20:3cm) and +(40:3cm) .. (intersection-2);
%  \clip [invclip] (c2) circle [radius = 2.5cm];
%  \clip [invclip] (c3) circle [radius = 2.5cm];
%  \fill [opacity=0.2](c3) circle [radius = 2.5cm] (c2) circle [radius = 2.5cm];
%\fill [opacity=0.5, color=green!50!blue, name intersections={of=d2 and d31, by={i1}}, name intersections={of=d31 and d12, by={i2}}, name intersections={of=d3 and d12, by={i3}}, name intersections={of=d3 and d2, by={i4, i5}}] 
%(i1) .. controls +(20:3cm) and +(40:3cm) .. (i2).. controls +(320:3cm) and +(340:3cm) ..(i3)--(i4);
   \draw [name path = d312, name intersections={of=d2 and d31, by={i1}}, name intersections={of=d3 and d12, by={i3}},] 
  (i1) .. controls +(-70:0.5cm) and +(70:0.5cm) .. (i3);
\end{tikzpicture}
\caption{$\Df(A)$}
\label{fig:A3 completa}
\end{subfigure}
 \caption{Wall-and-chamber structures for $\mathbb{A}_3$ with a zig-zag orientation.}
 \label{fig:no li chamber}
\end{figure}
\end{example}

\begin{example}\label{ex: Kronecker 2}
Let $A=\mathbb{C}Q$ be the path algebra of the Kronecker quiver \(Q=\begin{tikzcd}
	1 & 2
	\arrow[shift right, from=1-1, to=1-2]
	\arrow[shift left, from=1-1, to=1-2]
\end{tikzcd}\) 
as in \cref{ex: Kronecker 1}.
We depict the wall-and-chamber structures $\Df_1(A)$, $\Df_2(A)$, $\Df_3(A)$, $\Df_5(A)$, and $\Df(A)$ in \cref{fig:Df Kronecker}.

\begin{figure}
    \centering
    
    % --- Row 1 (Subfigures a, b) ---
    \begin{subfigure}[t]{0.48\textwidth} % 48% width for the first image
        \centering
       \begin{tikzpicture}
  \draw[-] (-3, 0) -- (3, 0) node[right] {$\dc(\rep{2})$};
  \draw[-] (0, -3) -- (0, 3) node[right] {$\dc(\rep{1})$};
%  \draw[-] (0,0) -- (3,-1.5) node[right] {$\dc(\rep{1\\22})$};
%  \draw[-] (0,0) -- (3,-2) node[right] {$\dc(\rep{11\\222})$};
%  \draw[-] (0,0) -- (3,-9/4);
%  \draw[-] (0,0) -- (3,-12/5);
%  \draw[-] (0,0) -- (3,-2.5);
  \draw[white] (0,0) -- (3,-3) node[below right, white] {$\dc\left(\rep{1\\2}_\lambda\right)$};
%  \draw[-] (0,0) -- (2.5,-3);
%  \draw[-] (0,0) -- (12/5,-3);
%  \draw[-] (0,0) -- (9/4,-3) ;
%  \draw[-] (0,0) -- (2,-3) node[below ] {$\dc(\rep{111\\22})$};
%  \draw[-] (0,0) -- (1.5,-3) node[below left] {$\dc(\rep{11\\2})$};
\end{tikzpicture}
        \caption{$\Df_1(A)$}        
        \label{fig:Df_1}
    \end{subfigure}
    \hfill % Pushes the subfigures to the sides
    \begin{subfigure}[t]{0.48\textwidth} % 48% width for the second image
        \centering
       \begin{tikzpicture}
  \draw[-] (-3, 0) -- (3, 0) node[right] {$\dc(\rep{2})$};
  \draw[-] (0, -3) -- (0, 3) node[right] {$\dc(\rep{1})$};
%  \draw[-] (0,0) -- (3,-1.5) node[right] {$\dc(\rep{1\\22})$};
%  \draw[-] (0,0) -- (3,-2) node[right] {$\dc(\rep{11\\222})$};
%  \draw[-] (0,0) -- (3,-9/4);
%  \draw[-] (0,0) -- (3,-12/5);
%  \draw[-] (0,0) -- (3,-2.5);
  \draw (0,0) -- (3,-3) node[below right] {$\dc\left(\rep{1\\2}_\lambda\right)$};
%  \draw[-] (0,0) -- (2.5,-3);
%  \draw[-] (0,0) -- (12/5,-3);
%  \draw[-] (0,0) -- (9/4,-3) ;
%  \draw[-] (0,0) -- (2,-3) node[below ] {$\dc(\rep{111\\22})$};
%  \draw[-] (0,0) -- (1.5,-3) node[below left] {$\dc(\rep{11\\2})$};
\end{tikzpicture}
        \caption{$\Df_2(A)$}
        \label{fig:Df_2}
    \end{subfigure}
    
    \medskip % Vertical space between rows
    
    % --- Row 2 (Subfigures c, d) ---
    \begin{subfigure}[t]{0.48\textwidth}
        \centering
       \begin{tikzpicture}
  \draw[-] (-3, 0) -- (3, 0) node[right] {$\dc(\rep{2})$};
  \draw[-] (0, -3) -- (0, 3) node[right] {$\dc(\rep{1})$};
  \draw[-] (0,0) -- (3,-1.5) node[right] {$\dc(\rep{1\\22})$};
%  \draw[-] (0,0) -- (3,-2) node[right] {$\dc(\rep{11\\222})$};
%  \draw[-] (0,0) -- (3,-9/4);
%  \draw[-] (0,0) -- (3,-12/5);
%  \draw[-] (0,0) -- (3,-2.5);
  \draw (0,0) -- (3,-3) node[below right] {$\dc\left(\rep{1\\2}_\lambda\right)$};
%  \draw[-] (0,0) -- (2.5,-3);
%  \draw[-] (0,0) -- (12/5,-3);
%  \draw[-] (0,0) -- (9/4,-3) ;
%  \draw[-] (0,0) -- (2,-3) node[below ] {$\dc(\rep{111\\22})$};
  \draw[-] (0,0) -- (1.5,-3) node[below left] {$\dc(\rep{11\\2})$};
\end{tikzpicture}
        \caption{$\Df_3(A)$}
        \label{fig:Df_3}
    \end{subfigure}
    \hfill
    \begin{subfigure}[t]{0.48\textwidth}
        \centering
       \begin{tikzpicture}
  \draw[-] (-3, 0) -- (3, 0) node[right] {$\dc(\rep{2})$};
  \draw[-] (0, -3) -- (0, 3) node[right] {$\dc(\rep{1})$};
  \draw[-] (0,0) -- (3,-1.5) node[right] {$\dc(\rep{1\\22})$};
  \draw[-] (0,0) -- (3,-2) node[right] {$\dc(\rep{11\\222})$};
%  \draw[-] (0,0) -- (3,-9/4);
%  \draw[-] (0,0) -- (3,-12/5);
%  \draw[-] (0,0) -- (3,-2.5);
\draw (0,0) -- (3,-3) node[below right] {$\dc\left(\rep{1\\2}_\lambda\right)$};
%  \draw[-] (0,0) -- (2.5,-3);
%  \draw[-] (0,0) -- (12/5,-3);
%  \draw[-] (0,0) -- (9/4,-3) ;
  \draw[-] (0,0) -- (2,-3) node[below ] {$\dc(\rep{111\\22})$};
  \draw[-] (0,0) -- (1.5,-3) node[below left] {$\dc(\rep{11\\2})$};
\end{tikzpicture}
        \caption{$\Df_5(A)$}
        \label{fig:tikz-d}
    \end{subfigure}

    \medskip % Vertical space between rows
    
    % --- Row 3 (Single Subfigure e) ---
    % Use a wider width (e.g., 0.6 or 0.7) and include a \centering inside the subfigure
    \begin{subfigure}{\textwidth} % Use \textwidth to allow it to span the full figure width
        \centering % Centers the content (the TikZ picture) within the subfigure
        \begin{tikzpicture}
  \draw[-] (-3, 0) -- (3, 0) node[right] {$\dc(\rep{2})$};
  \draw[-] (0, -3) -- (0, 3) node[right] {$\dc(\rep{1})$};
  \draw[-] (0,0) -- (3,-1.5) node[right] {$\dc(\rep{1\\22})$};
  \draw[-] (0,0) -- (3,-2) node[right] {$\dc(\rep{11\\222})$};
  \draw[-] (0,0) -- (3,-9/4);
  \draw[-] (0,0) -- (3,-12/5);
  \draw[-] (0,0) -- (3,-2.5);
  \foreach \x in {7, 8, ..., 60} {\draw[-] (0,0) -- (3, -3+3/\x);};
\draw[thick] (0,0) -- (3,-3) node[below right, thick] {$\dc\left(\rep{1\\2}_\lambda\right)$};  
\foreach \y in {7, 8, ..., 60} {\draw[-] (0,0) -- (3-3/\y, -3);};
\draw[-] (0,0) -- (2.5,-3);
  \draw[-] (0,0) -- (12/5,-3);
  \draw[-] (0,0) -- (9/4,-3) ;
  \draw[-] (0,0) -- (2,-3) node[below ] {$\dc(\rep{111\\22})$};
  \draw[-] (0,0) -- (1.5,-3) node[below left] {$\dc(\rep{11\\2})$};
\end{tikzpicture}
        \caption{$\Df(A)$}
        \label{fig:tikz-e}
    \end{subfigure}
    
    % --- Main Caption and Label ---
    \caption{The wall-and-chamber structures $\Df_\ell(A)$ of the Kronecker algebra $A$ for different values of $\ell$.}
    \label{fig:Df Kronecker}
\end{figure}
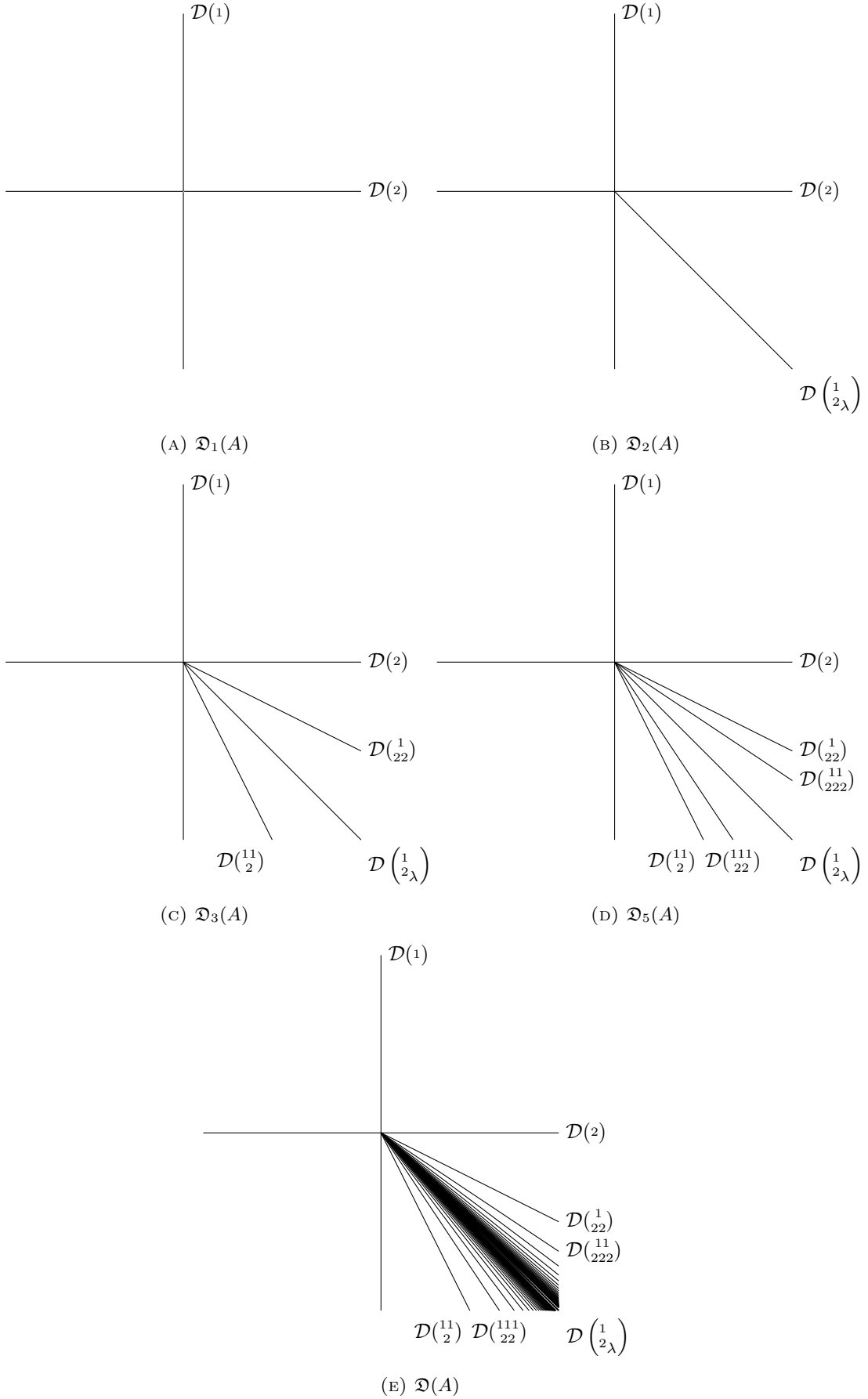
\end{example}

%%%%%%%%%%%%%%%%%
\section{Groupoids for lattices of torsion classes}\label{sec:groupoids}

In this section we begin in \cref{subsec: groupoid for posets} by associating a groupoid $\gc(\pc) = (\gc_0, \gc_1)$ to every poset $(\pc, \leq)$ where $\gc_0 = \pc$. 
We show that, under mild assumptions, there is a unique morphism $e_{(x,y)}\in \gc_1$ such that $s(e_{(x,y)}) = x$ and $t(e_{(x,y)}) = y$ for every pair of elements $x,y \in \pc$.
In \cref{subsec: Groupoids for torsion classes}, we apply this construction to the posets $\tors A$, $\tors_\ell A$, $\tors^{ss}A$, and $\tors^{ss}_\ell A$, and study the resulting groupoids. 
In particular, we interpret the morphisms in $\gc_1$ as subcategories and show that certain infinite products of morphisms are well defined.

\subsection{Groupoids for posets}\label{subsec: groupoid for posets}

In this subsection we associate a canonical groupoid to every poset. 
The objects of this groupoid are the elements of the poset, while its morphisms encode intervals. 
We show that, under mild assumptions, there is a unique morphism between every pair of objects, so that the groupoid admits an explicit description. 
This construction will later be applied to lattices of torsion classes.

\begin{definition}
A (small) groupoid $\mathcal{G}$ consists of a of a set of objects $G_0$ and a set of morphisms $G_1$, a pair of functions $s, t : G_1 \rightrightarrows G_0$, a product $\bullet: G_1\times G_1 \to G_1$, and maps $e: G_0 \to G_1$ and \mbox{$(-)^{-1}: G_1 \to G_1$} satisfying the following properties.
\begin{itemize}
    \item Given $f, g\in G_1$ the product $f\bullet g = fg$ is defined if and only if $t(f)=s(g)$. In this case $s(fg)= s(f)$ and $t(fg)= t(g)$.
    \item The product is associative, that is $f(gh) = (fg)h$ whenever $t(f)=s(g)$ and $t(g)=s(h)$.
    \item For every $x\in G_0$, there is an element $e(x)\in G_1$ is such that $x=s(e(x))=t(e(x))$. Moreover, if $f\in G_1$ is such that $s(f)=x$ and $t(f)=y$, then $ fe(y) =f= e(x)f$.
    \item For every $f \in G_1$, we have  $ff^{-1} = e(s(f))$ and $f^{-1}f = e(t(f))$.
\end{itemize}
\end{definition}

\begin{remark}
    The reader should be aware that in this paper we chose to multiply elements of $G_1$ from left to right because we are particularly interested in their relationship with actual groups. 
    This has the unfortunate consequence of composing morphisms from left to right whenever we consider a groupoid as a category.
\end{remark}

Given a poset $(\pc, \leq)$ and elements $x,y\in\pc$ where $x\leq y$, we denote by $[x,y]$ the interval defined by them, that is the set \([x,y] = \{z\in \pc \mid x\leq z \leq y\} \).
We denote by $\ic(\pc)$ the set
\(\ic(\pc):=\{[x,y] \mid x,y\in \pc, x\leq y\}\) of intervals of $\pc$ and by $\ic^{-1}(\pc)$ is the set of formal inverses of intervals. 
Let $W\left(\ic(\pc)\cup\ic^{-1}(\pc)\right)$ be the set of all the possible finite words written in the alphabet $\ic(\pc)\cup\ic^{-1}(\pc)$.

Before defining the groupoid associated with a poset, it is useful to view a groupoid as a category in which every morphism is invertible.
    Given any poset $(\pc, \leq)$, one can consider the free category generated by it as the category $Cat(\pc)$ whose objects correspond to the set $\pc$ and there is a unique morphism $f_{xy}\in \Hom_\pc(x,y)$ whenever $x\leq y$. 
    The groupoid that we define below can be regarded as the localization of $Cat(\pc)$ obtained by formally inverting every morphism.
    
    We now define a groupoid $\gc(\pc)$ for any given poset $(\pc, \leq)$. 

\begin{definition}\label{def:groupoids for posets}
    Let $(\pc, \leq)$ be a poset. 
We define $\gc(\pc)$ to be the quotient of the free groupoid generated by the intervals of $\pc$ modulo the following relations. 
    %    We define the groupoid $\gc(\pc)$ as the groupoid $\gc(\pc):= \left(\gc_0(\pc), \gc_1(\pc)\right)$ where $\gc_0(\pc) = \pc$ and $\gc_1(\pc)$ corresponds to $W\left(\ic(\pc)\cup\ic^{-1}(\pc)\right)$ modulo the following relations.
    \begin{itemize}
        \item $s([x,y])=x$ and  $t([x,y]) = y$.
        \item $s([x,y]^{-1})=y$ and  $t([x,y]^{-1}) = x$.
        \item $[x,y][y,z] = [x,z]$.
        \item $[y,z]^{-1}[x,y]^{-1} = [x,z]^{-1}$.
        \item $[x,y][x,y]^{-1} = [x,x]$ and $[x,y]^{-1}[x,y]= [y,y]$.
        \item $e(x) = [x,x] = [x,x]^{-1}$.
    \end{itemize}
\end{definition}

In the following proposition we show that under mild conditions there is always a unique element in $\gc_1(\pc)$ relating any two elements $x,y \in \pc$.

\begin{proposition}\label{prop: one map}
    Let $(\pc, \leq)$ be a poset and consider its groupoid $\gc(\pc)$ as defined above. 
    Suppose moreover that $(\pc,\leq)$ has a minimal element $0$, has a maximal element $1$, or that $(\pc, \leq)$ is a lattice.
    Then for every pair $x,y \in\pc$ there is a unique element $e_{xy}\in \gc_1(\pc)$ such that $s(e_{xy}) = x$ and $t(e_{xy})=y$.
    Moreover, the following holds.
    \begin{enumerate}
    \item If $(\pc, \leq)$ has a minimal element $0$, then $ e_{xy}=[0,x]^{-1}[0,y]$. 
    \item If $(\pc, \leq)$ has a maximal element $1$, then $e_{xy}=[x,1][y,1]^{-1}$. 
    \item If $(\pc, \leq)$ is a lattice, then $ e_{xy}=[x\wedge y,x]^{-1}[x\wedge y,y] = [x, x \vee y][y,x \vee y]^{-1}$. 
    \end{enumerate}
 
\end{proposition}

\begin{proof}
We start by showing the case where $(\pc, \leq)$ has a minimal element that we denote by $0$. 
Let $x,y\in\pc$ and let $f \in \gc_1(\pc)$ be such that $s(f) = x$ and $t(f)=y$.
By construction, $f$ is an equivalence class of elements in $W\left(\ic(\pc)\cup\ic^{-1}(\pc)\right)$. 
We will prove our claim by showing that any word $w\in f$ is equivalent in $\gc(\pc)$ to $[0,x]^{-1}[0,y]$.
Up to straight forward simplifications, $w$ is of one of two forms. 
Either  
    \begin{align}\label{eq:formf}
            w=[x_0,x_1][x_2,x_1]^{-1}[x_2,x_3][x_4,x_3]^{-1}\cdots [x_{s},x_{s+1}] 
    \end{align}
where $x_0 = x$, $x_{s+1}=y$, and both $x_{2i}\leq x_{2i+1}$ and  $x_{2i+2}\leq x_{2i+1}$ for every $1 \leq i \leq s/2$ or $w$ has the form 
        \begin{align}
            w=[x_1,x_0]^{-1}[x_1,x_2]\cdots [x_{s},x_{s+1}] 
    \end{align}
where $x_{2i}\geq x_{2i+1}$ and  $x_{2i+2}\geq x_{2i+1}$ for every $1 \leq i \leq s/2$. 

We show our claim by induction on $s$.
If $s=1$ then either $x \leq y$ or $y \leq x$ and the claim follows by the construction of $\gc(\pc)$.
    Indeed if $x \leq y$ then 
    \[
        [x,y] = [0,x]^{-1}[0,x][x,y] = [0,x]^{-1}[0,y].
    \]

Suppose that $s=2$ and assume that $w=[x,x_1][y,x_1]^{-1}$ where $x, y \leq x_1$. 
The following equalities are easily checked in $\gc(\pc)$.
    \begin{align}
        [x,x_1][y,x_1]^{-1} = & [x,x_1][x_1,x_1][y,x_1]^{-1}\\
        =& [x ,x_1][0,x_1]^{-1}[0,x_1][y,x_1]^{-1}\\
        =& [x ,x_1][x,x_1]^{-1}[0,x]^{-1}[0,y][y,x_1][y,x_1]^{-1}\\
        = & [0,x]^{-1}[0,y]
    \end{align}

    Suppose now that $[x,x_1][x_2,x_1]^{-1}[x_2,y] = w\in f$. 
    Then using the above we can rewrite $f$ as $f=[0,x]^{-1}[0,y]$.
    Indeed,  
    \begin{align}
        [x,x_1][x_2,x_1]^{-1}[x_2,y] = & [0,x]^{-1}[0,x_2][x_2,y]\\
        = & [0,x]^{-1}[0,y].
    \end{align}
    Combining this with the associativity of the composition in $\gc(\pc)$ implies that any word $w \in W\left(\ic(\pc)\cup\ic^{-1}(\pc)\right)$ that can be written as the product of $s$ intervals in $\pc$ and its formal inverses it also can be written using $s-1$ symbols.
    
    The proof for the other cases are very similar.
    If $\pc$ has a maximal element $1 \in \pc$ the proof is dual. 
    If $\pc$ is a lattice, one simply replaces the minimal element $0$ above by the meet of all the elements of $\pc$ appearing in the original expression for $w$ and then shows that $w = [x\wedge y, x]^{-1}[x\wedge y, y]$.
    Details are left to the reader.
    \end{proof}

The following corollary is straight forward.

\begin{corollary}\label{cor:quotientgroupoid}
Every epimorphism of posets $\varphi:(\pc, \leq) \to (\pc', \leq')$ induces an epimorphism $\widetilde{\varphi}: \gc(\pc) \to \gc(\pc')$ by sending $\widetilde{\varphi}([x,y]) = [\varphi(x), \varphi(y)]$ for every $[x,y]^{\pm 1}\in \ic(\pc)\cup\ic^{-1}(\pc)$ and extending $\widetilde{\varphi}$ to the whole of $W\left(\ic(\pc)\cup\ic^{-1}(\pc)\right)$ naturally.
\end{corollary}

Given two different elements $x, y \in \pc$, we say that $y$ covers $x$ if $x \leq y$ and there is no $z\in \pc$ such that $x \leq z \leq y$ with $z$ different from $x$ and $y$. 
If the poset $\pc$ is finite, every time that $x\leq y$ we can find a set 
\(
\{x_0, \dots, x_t \mid x_i \in \pc \text{ for every $1\leq i \leq t$}\}
\)
such that 
$x_0 = x$, $x_t =y$ and $x_i$ covers $x_{i-1}$ for every $1\leq i \leq t$.
Then \cref{prop: one map} implies the following.

\begin{corollary}\label{cor: finite generators}
Let $(\pc, \leq)$ be a finite poset satisfying the conditions of \cref{prop: one map}. 
Then the groupoid $\gc(\pc)$ is generated by the set 
\(
\{[x, y] \in \ic(\pc) \mid \text{ $y$ covers $x$}\}
\)
and their formal inverses, modulo the relations induced by $\gc(\pc)$.
\end{corollary}

%%%%%%%%%%
\subsection{A categorical interpretation for the groupoids for torsion classes.}\label{subsec: Groupoids for torsion classes}
In this subsection we specialize the construction of the previous subsection to lattices of torsion classes. 
Our main goal is to show that, in this setting, morphisms in the associated groupoids admit a natural categorical interpretation in terms of subcategories and the $*$-product. 
This interpretation allows us to define certain infinite products associated with chains of torsion classes and eventually leads to a concrete description of the groups $\widetilde G_\ell(A)$ that we associate to these groupoids.
To simplify notation, we set 
\begin{align*}
  \gc(A) &= \gc(\tors A), & \gc_\ell(A) &= \gc(\tors_\ell A), \\
  \gc^{\mathrm{ss}}(A) &= \gc(\tors^{\mathrm{ss}} A), & \gc^{\mathrm{ss}}_\ell(A) &= \gc(\tors^{\mathrm{ss}}_\ell A).
\end{align*}

Applying \cref{cor:quotientgroupoid} to the lattices of torsion classes and using \cref{cor:limit1,cor:limit2}, we immediately obtain the following.

\begin{corollary}\label{cor:limit4}
    \[
     \underleftarrow\lim\, \gc_\ell(A) = \gc(A) \quad \text{ and } \quad\underleftarrow\lim\, \gc_\ell^{ss}(A) = \gc^{ss}(A).
    \]
\end{corollary}

The posets considered in this subsection are posets of torsion pairs in an ambient category.
This allows a categorical interpretation of the product of certain elements in the corresponding groupoid. 

Recall from \cref{subsec: star product} that a chain of torsion classes $\eta$ is a set 
\[
\eta:= \{ \tc_\eta(i) \in \tors A \mid 
\tc_\eta(i) \subset \tc_\eta(j) \text{ if $j \leq i$ in $I$}\} \]
where $I$ is a totally ordered set with minimal element $0$ and maximal element $1$.
Given such a chain $\eta$, there is a set of morphisms in $\gc(A)$ indexed by $I$ defined as follows:
\[
\left\{e_\eta(t)=\left[\bigcup_{j>t} \tc_\eta(j),\bigcap_{i<t} \tc_\eta(i)\right] \mid t \in I
\right\}
\]
We aim to compose all of these morphisms. However, when $I$ is infinite, such a composition is not a priori defined in a general groupoid.
In the following proposition we use the categorical interpretation of morphisms together with the $*$-product to show that these infinite products are well defined.

\begin{proposition}\label{prop:infinite products}
    For every chain of torsion classes $\eta$ in $\tors A$, $\tors^{ss}A$, $\tors_\ell A$ and $\tors^{ss}_\ell A$ the formal compositions associated with an arbitrary chain of torsion classes always collapse to the unique morphism joining the endpoints of the chain.
    In other words, the products
    $$
    e_\eta = \prod_{t\in I}^{\leftarrow} e_\eta(t)\quad
    \text{ and } \quad
    e^{-1}_\eta = \prod_{t\in I}^{\rightarrow} e^{-1}_\eta(t),
    $$
    are well-defined and are equal to $e_\eta = [\tc_\eta(1), \tc_\eta(0)]$ or $e_{\eta}^{-1} = [\tc_\eta(1), \tc_\eta(0)]^{-1}$, respectively.
\end{proposition}

\begin{proof}
The statement for $\tors A$ follows directly from \cref{cor:infinite products modA}.
The proof for $\tors^{ss}A$ is a particular case of the above since every chain of torsion classes in  $\tors^{ss}A$ is a chain of torsion classes in $\tors A$. 
For $\tors_\ell A$, the result follows by noting that the filtration of \cref{cor:infinite products modA} has all its subfactors in $(\modA)_\ell$. 
Finally, every chain of torsion classes in $\tors_\ell^{ss}A$ is a chain of torsion classes in $\tors_\ell A$.
\end{proof}

As mentioned in \cref{cor:finite cone complex for l}, the lattice  $\tors^{ss}_\ell A$ is finite, so every chain in it is also finite. 
Thus, in this case, the previous proposition is a direct consequence of the construction of $\gc_\ell^{ss}(A)$. 
Nevertheless, this proposition gives a useful interpretation of the products of elements coming from its chains of torsion classes.

Given a groupoid $\gc$ one can always define a group $\widetilde{G}$, viewed as a groupoid with one element, which is the free group generated by all the elements of $\gc_1$ modulo the relations induced by $\gc$.
We now combine the previous proposition with the finite generation result of Corollary~\ref{cor: finite generators} to obtain an explicit presentation of the group $\widetilde{G}_\ell(A)$ associated with $\gc^{ss}_\ell(A)$.

\begin{proposition}\label{prop:deftildeG}
The group $\widetilde{G}_\ell(A)$ is isomorphic to the group whose generators are the subcategories  
\[
\{(\tc_{\df})_\ell \mid \df \in \Df^{0}_\ell(A)\} \text{ and } \{(\mod_{\df'}^{ss}A)_\ell \mid \df' \in \Df^1_\ell(A)\}  
\]
in $(\modA)_\ell$ and its formal inverses, with the $*$-product as an operation, modulo the relations 
\begin{equation}\label{eq:relations Gtilde}
(\tc_{\df_1})_\ell * (\mod_{\df'}^{ss}A)_\ell  = (\tc_{\df_2})_\ell
\end{equation}
whenever $(\tc_{\df_1})_\ell = (\tc_{\df'})_\ell$ and $(\overline{\tc}_{\df'})_\ell = (\tc_{\df_2})_\ell$.
\end{proposition}

\begin{proof}
Let $H_\ell(A)$ be the free group generated by all the subcategories of $(\modA)_\ell$ with the $*$-product and let $\phi: \gc_\ell^{ss}(A) \to H_\ell(A)$ be the map sending the element $[\{0\}, \tc]^{-1}[\{0\}, \tc']$ to the formal $*$-product of subcategories $\tc^{-1}*\tc'$.
We will show that all the products in $\gc^{ss}_\ell(A)$ can be realised via $\phi$ using the subcategories and relations in the statement. 

Suppose first that $[\tc_\ell, \tc'_\ell]$ is an interval in $\tors^{ss}_\ell A$ such that $\tc$ is maximally contained in $\tc'$. 
In this case, we have shown in \cref{cor: maximal inclusions tors^ss}, there are cones $\df_1, \df_2 \in \Df^0_\ell(A)$ such that $\tc_\ell = (\tc_{\df_1})_\ell$ and $\tc'_\ell = (\tc_{\df_2})_\ell$ and there is a cone $\df' \in \Df^1_\ell(A)$ such that $(\tc_1)_\ell = (\tc_{\df'})_\ell$ and $(\tc_2)_\ell = (\overline{\tc_{\df'}})_\ell$.
Moreover, in this case the category $\xc_{[(\tc_1)_\ell, (\tc_2)_\ell]} = (\mod^{ss}_{\df'}A)_\ell$. 
Then \cref{prop:extending torsion classes} implies that 
\[
\phi([\{0\},  \tc'_\ell])=\phi([\{0\}, \tc_\ell][\tc_\ell, \tc'_\ell])=\phi([\{0\}, \tc_\ell])*\phi([\tc_\ell, \tc'_\ell])
\]
giving rise to \cref{eq:relations Gtilde}.

Now, let $[\tc_\ell, \tc'_\ell]$ be any interval in $\tors^{ss}_\ell(A)$. 
Then this interval can be refined to a maximal chain 
$\tc_\ell = (\tc_0)_\ell \subset (\tc_1)_\ell \subset \dots \subset (\tc_r)_\ell = \tc'_\ell$.
\cref{cor: maximal inclusions tors^ss} provide two codimension $0$ cones $\df_a$ and $\df_b$ such that $\tc_\ell = (\tc_{\df_a})_\ell$ and $\tc'_\ell = (\tc_{\df_b})_\ell$ and a set of cones $\{\df'_i \mid 1\leq i \leq r\}$ realising the maximal inclusions. 
We obtain
\begin{equation}\label{eq:phi}
\phi([\tc_\ell, \tc'_\ell]) = \phi([(\tc_0)_\ell , (\tc_1)_\ell ]) * \phi([(\tc_1)_\ell , (\tc_2)_\ell ]) *\dots *\phi([(\tc_{r-1})_\ell , (\tc_r)_\ell ])
\end{equation}
as a consequence of \cref{cor:infinite products modA}. 
Note that \cref{eq:phi} is independent of the refinement.
%Then, once again 
%\[
%\phi([\{0\},  \tc'_\ell])=\phi([\{0\}, \tc_\ell][\tc_\ell, \tc'_\ell])=\phi([\{0\}, \tc_\ell])*\phi([\tc_\ell, \tc'_\ell]). 
%\]
Finally, the result follows from \cref{prop: one map} which states that every element in $\gc_\ell^{ss}(A)$ is of the form $[\{0\}, \tc]^{-1}[\{0\}, \tc']$.
\end{proof}

%\begin{remark}
%Although this description is useful from a theoretic point of view, it is very difficult to know if a given subcategory $\xc\subset(\modA)_\ell$ corresponds an element in $\widetilde{G}_\ell(A)$.
%In particular, it is difficult to give a general interpretation of the $*$-multiplication $\xc *\mathcal{Y}^{-1}$ of a subcategory $\xc\subset (\modA)_\ell$ and the formal inverse of another subcategory $\mathcal{Y}\subset(\modA)_\ell$. 
%\end{remark}

\begin{remark}
Proposition \ref{prop:deftildeG} provides an entirely categorical realization of $\widetilde G_\ell(A)$ in terms of subcategories and the $*$-product. 
On the other hand, deciding whether an arbitrary subcategory determines an element of the group remains difficult. 
In particular, no satisfactory intrinsic description of products of the form $\xc * \yc^{-1}$ is currently known.
\end{remark}

%%%%%%%%%%%%%%%%%%
\section{Picture groups and $\tau$-cluster morphism categories for Artin algebras}\label{sec:picture group}
%
%Our construction of scattering diagrams for Artin algebras depends on the notion of cluster morphism categories and picture groups, first introduced for hereditary algebras of Dinkyn type by Igusa, Todorov and Weymann in \cite{IT, ITW}. 
%This definition was later extended for $\tau$-tilting finite algebras in \cite{HansonIgusa}. 
%A combinatorial approach to picture groups was studied by Kaipel in \cite{Kaipel} for the so-called partitioned fans. 
%Given a partitioned fan $(\Sigma, \Pf)$, Kaipel shows the existence of a category $\Cf(\Sigma, \Pf)$ and then uses this category to define the picture group $G(\Sigma, \Pf)$ associated to the partitioned fan $(\Sigma, \Pf)$. 
%
%In this section, we study the fan 
%\(
%\Sigma_\ell(A):= \{ \overline{\df} \mid \df \in \Df_\ell (A)\}
%\).
%We show $\Sigma_\ell(A)$ is naturally partitioned in sense of Kaipel \cite{Kaipel}, which implies the existence of a category $\Cf_\ell(A)$ and a picture group $G_\ell(A)$ for every $A$ and every $\ell\geq1$. 
%By abuse of notation, we say that $\Cf_\ell(A)$ is the $\tau$-cluster morphism category of $(\modA)_\ell$.

The aim of this section is two-fold.
As a first stance, we show that there is a $\tau$-cluster morphism category $\Cf_\ell(A)$ and a picture group $G_\ell(A)$ for every Artin algebra $A$ and every positive integer $\ell$.
This allows us to propose a definition of a picture group for every algebra $A$ which recovers $G(A)$ for $\tau$-tilting finite algebras and which we will use to define scattering diagrams. 
It is worth mentioned that our proposed definition is not necessarily equivalent to the one proposed by Kaipel in \cite{Kaipel2}.
%The explicit relationship between these definitions of picture group for algebras that are not $\tau$-tilting finite are left for the future. 

%As a consequence of the above, we are able to extend the conjecture stated above as follows. 
%\begin{conjecture}
%Let $A$ be an Artin algebra and $\ell$ a positive integer. 
%Then the nerve $X_\ell(A)$ of the $\tau$-cluster morphism category $\Cf_\ell(A)$ is a $K(G_\ell(A), 1)$ category, where $G_\ell(A)$ is the picture group associated to $\Cf_\ell(A)$.
%\end{conjecture}

The second aim of this section is to compare the picture group $G_\ell(A)$ with the group $\widetilde{G}_\ell(A)$ defined in \cref{prop:deftildeG}. 
In fact, we show that they are isomorphic. 
As a consequence of this result and a result by Hanson \cite{BKH} we obtain the existence of a faithful functor $\Gamma_\ell : \Cf_\ell(A) \to G_\ell(A)$ for every $A$ and every positive integer $\ell$. 
In particular we show that such a functor exists for every $\tau$-tilting finite algebra $A$.

%%%%%%%%%%%%%
\subsection{The $\tau$-cluster morphism category of $(\modA)_\ell$}

The goal of this subsection is to construct a $\tau$-cluster morphism category for $(\mod A)_\ell$. 
By the construction of Kaipel \cite{Kaipel}, this amounts to showing that the partition of $\Sigma_\ell(A)$ induced by semistable subcategories is admissible. 
The remainder of the subsection is devoted to defining and proving this property.

Let $\Df$ be an open cone complex in $\mathbb{R}^n$ inducing a fan $\Sigma = \{\overline{\df} : \df \in \Df\}$ and let $\Pf$ be a partition of the set of cones of $\Sigma$. 
Then $\Pf$ induces an equivalence relation $\sim_\Pf$ on $\Sigma$. 
Recall that given a cone $\sigma \in \Sigma$ of codimension $d$, we denote by $\star(\sigma)$ the cone complex in $ \sigma^\perp \cong \mathbb{R}^d$ formed by the images of the cones $\tau \in \Sigma$ such that $\sigma \subset \tau$ under the map $\pi_\sigma : \mathbb{R}^n \to \sigma^\perp \cong \mathbb{R}^d$.

\begin{definition}\cite[Section~2]{Kaipel}\label{def:partitionfan}
An equivalence relation $\sim_\Pf$ in $\Sigma$ is said to be admissible if given two cones $\sigma_1 \sim_\Pf \sigma_2$ and a pair of cones  $\tau_1$ and $\tau_2$ satisfying $\pi_{\sigma_i}(\tau_i) \in \star(\sigma_i)$ and $\pi_{\sigma_1}(\tau_1) = \pi_{\sigma_2}(\tau_2)$ implies that $\tau_1 \sim_\Pf \tau_2$.
A partitioned fan is a pair $(\Sigma, \Pf)$ of a fan $\Sigma$ and a partition $\Pf$ on $\Sigma$ such that the equivalence relation $\sim_\Pf$ on $\Sigma$ is admissible. 
\end{definition}

Given a partitioned pair $(\Sigma, \Pf)$ where $\Pf$ is an admissible partition of $\Sigma$, then one can define the category of the partitioned fan $(\Sigma, \Pf)$ as follows. 

\begin{proposition}\cite[Proposition~3.8]{Kaipel}
Let $(\Sigma, \Pf)$ be an admissible partitioned fan.
Then there is a category $\Cf(\Sigma, \Pf)$ whose objects are the equivalence classes $\{[\sigma]_{\sim_\Pf}\mid \sigma\in \Sigma \}$ and the morphisms are given by 
\[
\Hom_{\Cf(\Sigma, \Pf)}([\sigma], [\tau]):= \bigcup_{\substack{\sigma \sim_\Pf \sigma' \\ \tau \sim_\Pf \tau''}} \Hom_{\leq} (\sigma', \tau'')
\]
where $\Hom_\leq(\sigma, \tau)$ is the set of morphisms of the category associate to $\Sigma$ viewed as a poset where $\sigma \leq \tau$ if and only if $\sigma \subset \tau$.
\end{proposition}

Let $A$ an Artin algebra, $\ell\in \mathbb{N}$ and consider the cone complex $\Df_\ell(A)$ of $\ell$-TF-equivalence classes of stability conditions over $(\modA)_\ell$.
Let $$\Sigma_\ell(A):= \{\overline{\df} \mid \df \in \Df_\ell(A)\}.$$
The semistable categories determine a natural partition of this fan: two cones
$\df,\df'\in \Df_\ell(A)$ are equivalent whenever
$(\mod_\df^{ss}A)_\ell=(\mod_{\df'}^{ss}A)_\ell$.
We denote the induced equivalence relation on $\Sigma_\ell(A)$ by $\sim_{ss}$ and the corresponding partition by $\Pf_{ss}$.

%
%
%There is a natutal equivalence relation $\sim_{ss}$ defined in $\Df_\ell(A)$ by $\df \sim_{ss} \df'$ if and only if \mbox{$(\mod_{\df}^{ss}A)_\ell = (\mod_{\df'}^{ss}A)_\ell$}. 
%This determines naturally an equivalence relation $\sim_{ss}$ and a partition $\Pf_{ss}$ in $\Sigma_\ell(A)$. 

\begin{remark}
By construction the fan $\Sigma_\ell(A)$ is always finite and complete.
\end{remark}

Our first aim is to show the following. 

\begin{theorem}\label{prop:l-tauclustermorphismcategory}
The equivalence relation $\sim_{ss}$ in $\Sigma_\ell(A)$ is admissible. 
In particular, there is a $\tau$-cluster morphism category $\Cf_\ell(A):= \Cf(\Sigma_\ell(A), \Pf_{ss})$ for every Artin algebra $A$ and every positive integer $\ell$. 
\end{theorem}

The proof of \cref{prop:l-tauclustermorphismcategory} reduces to showing that semistability behaves well under the local identifications appearing in the definition of admissibility. 
The following two lemmas establish precisely this fact.

\begin{lemma}\label{lem:l-taucluster1}
Let $\df_1, \df_2 \in \Df_\ell(A)$ such that $(\mod_{\df_1}^{ss}A)_\ell=(\mod_{\df_2}^{ss}A)_\ell$ and $\df_1^\perp = \df_2^\perp$.
Let $M \in (\mod_{\df_1}^{ss}A)_\ell$ and let $L$ be a subobject of $M$. 
Then one of the following holds:
\begin{enumerate}
\item  $\df_1 \subset [L]^\perp$ if and only if $\df_2 \subset [L]^\perp$.
\item There is a $v_1\in \df_1$ such that  $\langle v_1, [L]\rangle <0$ if and only if there is a $v_2 \in \df_2$ such that $\langle v_2, [L]\rangle <0$.
\end{enumerate}

\end{lemma}

\begin{proof}
First of all, note that the subspaces $\langle\df_1\rangle, \langle\df_2\rangle \subset \mathbb{R}^n$ spanned by $\df_1$ and $\df_2$ coincide since $\df_1^\perp = \df_2^\perp$.
Moreover, $\df_i \subset [L]^\perp$ if and only if $\langle\df_i\rangle \subset [L]^\perp$ because $[L]^\perp$ is a subspace of $\mathbb{R}^n$. 

Suppose that $\df_1$ is not contained in $[L]^\perp$.
Choose $v_1\in\df_1\setminus[L]^\perp$.
Since $M$ is $v$-semistable for every $v\in\df_1\cup\df_2$, we have
\[
\langle v_1,[L]\rangle<0.
\]
By part (1), $\df_2$ is also not contained in $[L]^\perp$.
Applying the same argument to $\df_2$ yields a vector
$v_2\in\df_2$ satisfying
\[
\langle v_2,[L]\rangle<0.
\]
\end{proof}

\begin{lemma}\label{lem:l-taucluster2}
Let $M \in (\modA)_\ell$ and let $\df_1, \df_2 \in \Df_\ell(A)$ be two cones such that $\df_i \subset \dc(M)$ and $\df_1^\perp = \df_2^\perp$.
If $\ef_1, \ef_2 \in \Df_\ell(A)$ are two cones such that $\df_i \subset \overline{\ef_i}$ and $\pi_{\df_1}(\ef_1)= \pi_{\df_2}(\ef_2)$ then $\ef_1 \subset \dc(M)$ if and only if $\ef_2 \subset \dc(M)$. 
\end{lemma}

\begin{proof}
First note that $\df_i \subset \dc(M)$ is equivalent to $M \in (\mod^{ss}_{\df_i})_\ell$. 
Since the category of $v$-semistable objects is constant for every $v\in \df_i$, the statement is equivalent to showing that, under the hypothesis of this lemma, there is a $v_1\in \ef_1$ such that $M$ is $v_1$-semistable if and only if there is a $v_2\in \ef_2$ such that $M$ is $v_2$-semistable.
In the following, starting from a vector $v_1\in\ef_1$, we construct a vector
$\overline v_2\in\ef_2$ which satisfies the semistability inequalities for every submodule of $M$.

Let $v_1 \in \ef_1$ and assume that $M$ is $v_1$-semistable. 
Write $v_1$ as $v_1=w_1+ w$ where $w_1 \in \df_1$ and $w \in \pi_{\df_1}(\ef_1) \subset \df_1^\perp$. 

Since $\pi_{\df_1}(\ef_1)= \pi_{\df_2}(\ef_2)$ there exists $v_2 \in \ef_2$ of the form $v_2 = w_2 + w$ where $w_2 \in \df_2$ and $w$ is the same as before.
Moreover $\langle v_2, [M] \rangle=0$. 
Indeed, 
\begin{multline*}
\langle v_2, [M] \rangle	 = \langle w_2 + w, [M] \rangle = \langle w_2, [M] \rangle +\langle w, [M] \rangle =\\= 0 +\langle w, [M] \rangle =
 \langle w_1, [M] \rangle +\langle w, [M] \rangle = \langle v_1, [M] \rangle =0.
\end{multline*}

Consider the set $\{[L]\in K_0(A) \mid L \text{ is a submodule of } M\}$ consisting in the elements of $K_0(A)$ associated to the subobjects of $M$. 
This set is finite and we can just write it as  $\{[L_1], \dots, [L_t]\}$.
By hypothesis 
$$\langle v_1, [L_i]\rangle = \langle w_1, [L_i] \rangle +\langle w, [L_i] \rangle  \leq 0$$
 for every $1\leq i\leq t$. 
 
Fix $1 \leq i \leq t$. 
If $\langle w, [L_i] \rangle >0$ then $\langle w_1 , [L_i]\rangle <0$. 
\cref{lem:l-taucluster1}  therefore provides a vector $v_2^i\in \df_2$ such that $\langle v_2^i, [L_i]\rangle < 0$.
After multiplying by a positive scalar if necessary, we can assume that $\langle v_2^i, [L_i]\rangle + \langle w, [L_i] \rangle \leq 0$.
Define \[\overline{v}_2 = \left(\sum_{i=1}^t v_2^i\right) + w_2 + w. \]
Since $\left(\sum_{i=1}^t v_2^i\right) \in \df_2$ and $ w_2 + w = v_2 \in \ef_2$ we have $\overline{v}_2 \in \ef_2$. 

We now show that $M$ is $\overline{v}_2$-semistable. 
Firstly, 
\(
\langle \overline{v}_2 , [M] \rangle = \langle v_2 , [M]\rangle + \langle \sum_{i=1}^t v_2^i, [M] \rangle = 0
\).
Secondly, let $L$ be a submodule of $M$. 
Then there is a $1\leq i \leq t$ such that $[L]=[L_i]$ and 
\begin{align*}
\langle \overline{v}_2, [L] \rangle &= \left\langle \left(\sum_{i=1}^t v_2^i\right) + w_2 + w, [L_i]\right\rangle\\
=&\underbrace{\langle v_2^i+ w, [L_i]\rangle}_{\leq 0} +
\underbrace{\left\langle \left(\sum_{j\neq i} v_2^i\right) + w_2, [L_i]\right\rangle}_{\leq 0}\leq 0.
\end{align*}
The first term of that sum is less or equal to zero by construction and the second is also less or equal to zero because $\left(\sum_{j\neq i} v_2^i\right) + w_2\in \df_2$. 
\end{proof}

\begin{proof}[Proof of \cref{prop:l-tauclustermorphismcategory}]
We need to show that the equivalence relation $\sim_{ss}$ of $\Sigma_\ell(A)$ is admissible in the sense of \cref{def:partitionfan}.
Let $\sigma_i = \overline{\df_i}$ with $\sigma_1 \sim_{ss}\sigma_2$.
Thus $(\mod_{\df_1}^{ss}A)_\ell=(\mod_{\df_2}^{ss}A)_\ell$.

So, let $\df_1, \df_2$ as above and take $\ef_i \in \star(\df_i)$ such that $\pi_{\df_1}(\ef_1) = \pi_{\df_2}(\ef_2)$. 
If \mbox{$M\in (\mod^{ss}_{\ef_i} A)_\ell$} then $M\in (\mod^{ss}_{\df_i} A)_\ell$ since $\df_i \subset \overline {\ef_i}$.
Hence $\ef_1\subset \dc(M)$ if and only if $\ef_2 \subset \dc(M)$ by \cref{lem:l-taucluster2}. 
As a consequence the following equalities hold.
\[
(\mod_{\ef_1}^{ss}A)_\ell = 
\{M \in (\modA)_\ell  \mid \ef_1 \subset \dc(M)\} =
\{M \in (\modA)_\ell  \mid \ef_2 \subset \dc(M)\} =
(\mod_{\ef_2}^{ss}A)_\ell
\]
In conclusion, if $\tau_1 = \overline{\ef_1}$ and $\tau_2 = \overline{\ef_2}$ we have that $\tau_1 \sim_{ss} \tau_2$ and $\sim_{ss}$ is admissible in $\Sigma_{\ell}(A)$ as we wanted to show..
\end{proof}

As an immediate consequence, we recover the following result of \cite{STTW}.

\begin{corollary}\cite[Theorem~1.4]{STTW}
%Let $A$ be a $\tau$-tilting finite algebra. 
%Then the $\tau$-cluster morphism category $\Cf(A)$ of $A$ is determined by its wall-and-chamber structure.
The $\tau$-cluster morphism category $\Cf(A)$ of every $\tau$-tilting finite algebra $A$ is determined by its wall-and-chamber structure.
\end{corollary}

We note, however, that the proofs in \cite{STTW} rely heavily on $\tau$-tilting theory.
The proof we present here, in contrast, uses only basic properties of stability conditions.

\begin{remark}
We can define for every Artin algebra $A$ its $\tau$-cluster morphism category $\Cf(A)$ as 
\[
\Cf(A) := \underleftarrow\lim\, \Cf_\ell(A).
\]
Another alternative to construct a $\tau$-cluster morphism category for any Artin algebra was proposed by Kaipel in \cite{Kaipel2}. 
These two constructions, Kaipel and ours, give raise to different categories for a $\tau$-tilting infinite algebra $A$. 
We choose to leave the comparison between these two categories for the future.
\end{remark}

%%%%%%%%%%%%%%%%%
\subsection{Picture groups for Artin algebras}\label{subsec: picture groups}

In this subsection we define the picture group $G_\ell(A)$ associated with the cone complex $\Df_\ell(A)$. 
The construction follows Kaipel's approach and proceeds in three steps. 
First, we show that the fan $\Sigma_\ell(A)$ carries the structure of a fan poset. 
Second, we use the admissible partition introduced in the previous subsection to obtain a partitioned fan poset. 
Finally, we recall Kaipel's construction of the associated picture group \cite{Kaipel} and specialize it to our setting.

\begin{definition}\cite[Definition~4.7]{Kaipel}\label{def:fanposet}
A fan poset is a pair $(\Sigma, \leq)$ where $\Sigma$ is a complete and finite fan and $\leq$ is a partial order on the codimension $0$ cones of $\Sigma$ satifying the following conditions:
\begin{enumerate}
\item For every cone $\sigma \in \Sigma$, the set of cones 
$$\star^0(\sigma):= \{\tau\in \Sigma \mid \sigma\subset \tau \text{ and } \codim(\tau)=0\}$$
forms an interval. 
The interval determined by $\sigma\in \Sigma$ is denoted by $[\sigma^-, \sigma^+]$ and it is said to be a facial interval.
\item If $I$ is an interval in $(\Sigma, \leq)$ then the union 
\(
\bigcup_{\tau \in I}{\tau}
\)
is a cone in $\mathbb{R}^n$. 
\end{enumerate}
\end{definition}

It has been proved in \cite[Proposition~6.15]{Kaipel} that for every $\tau$-tilting finite algebra $A$, the pair $(\Sigma(A), \leq)$ is a fan poset, where $\leq$ is the partial order induced by the torsion classes associated to the maximal cones in $\Sigma(A)$. 
We now prove that the same is true for $\Sigma_\ell(A)$ for every algebra $A$ and every positive integer $\ell$. 
The proof below follows the same ideas as \cite[Proposition~6.15]{Kaipel}. 

\begin{proposition}\label{prop:fanposet}
Let $A$ be an Artin algebra and $l\in \mathbb{N}$. 
Then $(\Sigma_\ell(A), \leq )$ is a fan poset, where the poset structure is given by $\sigma_1 \leq \sigma_2$ if and only if $(\tc_{\df_1})_\ell \subset (\tc_{\df_2})_\ell$ and $\sigma_i = \overline{\df_i}$. 
\end{proposition}

\begin{proof}
We first verify \cref{def:fanposet}.(1).
Let $\sigma\in \Sigma_\ell(A)$.
By definition $\sigma=\overline{\df}$ for some $\df \in \Df_\ell(A)$. 
Then it has been shown in \cite[Lemma~2.16]{Asai} that $\df \subset \overline{\df'}$ if and only if $(\overline{\tc}_{\df})_\ell \supset (\overline{\tc}_{\df'})_\ell$ and $(\tc_{\df})_\ell \subset (\tc_{\df'})_\ell$. 
This holds in particular for all cones $\df' \in \Df_\ell(A)$ of codimension $0$.
So $\star^0(\sigma)$ is an interval in $(\Sigma_\ell(A),\leq)$.

It remains to verify \cref{def:fanposet}.(2), let $I$ be an interval in $\tors^{ss}_\ell(A)$. 
Then, by definition $I = [\tc_\ell, \tc'_\ell]$ where $\tc_\ell$ and $\tc'_\ell$ are torsion classes associated to cones of codimension $0$ in $\Df_\ell(A)$.
Then there are codimension $0$ cones $\df, \df' \in \Df_\ell(A)$ such that $\tc_\ell = (\tc_\df)_\ell = (\overline{\tc}_{\df})_\ell$ and $\tc'_\ell = (\tc_{\df'})_\ell = (\overline{\tc}_{\df'})_\ell$.
Let us denote $\fc_\ell, \fc'_\ell$ the torsion-free classes such that $(\tc_\ell, \fc_\ell)$ and $(\tc'_\ell, \fc'_\ell)$ are torsion pairs in $(\modA)_\ell$.
Consider, moreover, the subsets $\mathcal{H}^+_{\tc'_\ell}$ and $\mathcal{H}^-_{\fc_\ell}$ defined as 
\[
\mathcal{H}^+_{\tc_\ell} := \bigcap_{M \in \tc_\ell} \{v \in \mathbb{R}^n \mid \langle v, [M]\rangle \geq 0\}, \quad
\mathcal{H}^-_{\fc'_\ell} := \bigcap_{M \in \fc'_\ell} \{v \in \mathbb{R}^n \mid \langle v, [M]\rangle \leq 0\}.
\]
The key observation is the identity
\[
\bigcup_{\sigma\in I}\sigma
=
\mathcal H^+_{\tc_\ell}
\cap
\mathcal H^-_{\fc'_\ell}.
\]
Indeed, if $\widetilde{\sigma} \in I$ we have that $\tc_\ell \subset (\tc_{\widetilde{\sigma}})_\ell$. 
By definition of $(\tc_{\widetilde{\sigma}})_\ell$, we have that $\langle v, [M]\rangle \geq 0$ for every $v\in \widetilde{\sigma}$ and every $M \in (\tc_{\widetilde{\sigma}})_\ell$. 
In particular, $\langle v, [M]\rangle \geq 0$ for every $v\in \widetilde{\sigma}$ and every $M \in \tc_\ell$. 
Therefore $\widetilde{\sigma} \subset \mathcal{H}^+_{\tc_\ell}$. 
Using dual arguments one shows that $\widetilde{\sigma} \subset \mathcal{H}^-_{\fc'_\ell}$.
Then $\widetilde{\sigma} \subset \mathcal{H}^+_{\tc_\ell}\cap\mathcal{H}^-_{\fc'_\ell}$.

Conversily, let $v \in \mathcal{H}^+_{\tc_\ell} \cap \mathcal{H}^-_{\fc'_\ell} \cap \widetilde{\sigma}$ for some $\widetilde{\sigma} \in \Sigma_\ell(A)$ of codimension $0$. 
Then if follows directly from the definitions that $\tc_\ell \subset (\tc_{\widetilde{\sigma}})_\ell$ and $\fc'_\ell \subset (\fc_{\widetilde{\sigma}})_\ell$. 
Hence $\widetilde{\sigma} \subset  \mathcal{H}^+_{\tc_\ell} \cap \mathcal{H}^-_{\fc'_\ell}$.
\end{proof}

\begin{remark}
The poset structure of $(\Sigma_\ell(A), \Pf_{ss}, \leq)$ is determined by the torsion classes induced by the stability conditions in the interior of each cone $\sigma \in \Sigma_\ell(A)$. 
Therefore the previous result recovers the realisation of $\tors^{ss}_\ell A$ obtained in \cref{subsec: W&C modA_l}.
\end{remark}

Given an algebra $A$ and a positive integer $\ell$ we get that $(\Sigma_\ell(A), \Pf_{ss}, \leq)$ is a partitioned fan poset. 
Therefore one can define the picture group $G_\ell(A)$ for any algebra $A$ and any $\ell$ using the structure of the $\tau$-cluster morphism category $\Cf_\ell(A)$.

Recall that the objects of $\Cf_\ell(A)$ are the $\sim_{ss}$-equivalence classes of cones in $\Df_\ell(A)$. 
A morphism from $[\df_1]$ to $[\df_2]$ is represented by an inclusion $\df'_1 \subset \df'_2$ for suitable representatives $\df'_i \sim_{ss} \df_i$
%Given two cones $\df_1, \df_2$ there is a map $f \in \Hom_{\Cf_\ell(A)} ([\df_1], [\df_2])$ if and only if there are cones $\df'_1,\df'_2 \in \Df_\ell(A)$ such that $\df_i \sim_{ss} \df'_i$ and $\df'_1 \subset \overline{\df'_2}$.
%Equivalently, for every pair of cones $\df_1, \df_2 \in \Df_\ell(A)$ such that $\df_1 \subset \overline{\df_2}$ we have a map $f_{[\df_1, \df_2]} \in \Hom_{\Cf_\ell(A)} ([\df_1], [\df_2])$.
%Moreover, if $\df_1, \df_2, \df'_1, \df'_2 \in \Df_\ell(A)$ are such that $\df_i \sim_{ss} \df'_i$, $\df_1 \subset \overline{\df_2}$ and $\df'_1 \subset \overline{\df'_2}$ we have that $f_{[\df_1, \df_2]} = f_{[\df'_1, \df'_2]}$ in $\Hom_{\Cf_\ell(A)} ([\df_1], [\df_2])$.
We first recall Kaipel's original definition.

\begin{definition}\cite[Definition~4.10]{Kaipel}\label{def:l-picture group1}
Let $(\Sigma_\ell(A), \Pf_{ss}, \leq)$ be the partitioned fan as above. 
Then the picture group $G_\ell(A)$ has as generators the set $\{X_{[\df]} \mid \df \in \Df^1_\ell(A)\}$
 and the following set of relations. 
\begin{enumerate}
	\item $X_{[\df_1]} X_{[\df_2]} \dots X_{[\df_s]} = X_{[\df'_1]} X_{[\df'_2]} \dots X_{[\df_t]}$ whenever $(\df_1, \df_2, \dots, \df_t)$ and $(\df'_1, \df'_2, \dots, \df'_s)$ are two distinct ordered sets of cones labelling maximal chains of an interval $[\tc_\ell, \tc'_\ell]$ in $(\Sigma_\ell(A), \leq)$.
	Denote this element $X_{[\tc_\ell, \tc'_\ell]}$. 
	\item $X_{[\sigma^-_1, \tau^-_1]} = X_{[\sigma^-_2, \tau^-_2]}$ whenever $f_{[\sigma_1, \tau_1]} = f_{[\sigma_2, \tau_2]}$ in $\Cf_\ell(A)$.
\end{enumerate}
\end{definition}

In our situation the second family of relations is redundant by  \cite[Lemmas~4.12, 4.13]{Kaipel}, yielding the following much simpler presentation.

\begin{definition}\label{def:l-picture group2}
The picture group $G_\ell(A)$ is defined has generators 
$$\{g_{\df} \in \Df^0_\ell(A)\} \cup \{g_{\df'} \in \Df^1_\ell(A)\}$$
modulo the relations \[
g_{\df_2} = g_{\df_1} g_{\df'}
\]
whenever $\df' = \overline{\df_1} \cap \overline{\df_2}$ and $(\tc_{\df_1})_\ell \subset (\tc_{\df_2})_\ell$.
\end{definition}

\begin{remark}\label{rmk:different notations}
Notice that the element $g_\df$ for a $\df \in \Df_\ell^0(A)$ in the notation of \cref{def:l-picture group2} corresponds to the element $X_{[\{0\}, (\tc_{\df})_\ell]}$ in the notation of \cref{def:l-picture group1}.
Moreover, it follows from \cref{def:l-picture group1} that the identity element $1_{G_\ell(A)}$ can be written as $X_{[(\tc_{\df})_\ell, (\overline{\tc_{\df}})_\ell]}$ for every $\df\in\Df^0_\ell(A)$. 
\end{remark}

%%%%%%%%%%%%%%%%%
\subsection{From groupoids to picture groups}

In the previous section we constructed the groupoid $\gc^{ss}_\ell(A)$ associated to $\tors^{ss}_\ell A$.
The aim of this subsection is to relate this groupoid to the picture group $G_\ell(A)$.
More precisely, we show that the natural embedding of codimension-one cones into $\gc^{ss}_\ell(A)$ extends to a morphism of groupoids whose image generates $G_\ell(A)$.

Consider the map 
\[
\iota: \Df_\ell(A) \to \gc_\ell^{ss}(A)
\]
sending any cone $\df \in \Df_\ell(A)$ to the interval $[(\tc_\df)_\ell, (\overline{\tc}_{\df})_\ell] \in \ic(\tors_\ell^{ss}A) \subset \gc_\ell^{ss}(A)$ which is well-defined and injective by the construction of $\Df_\ell(A)$.
In turn, this map induces naturally a map 
\[
\rho: \Image(\iota) \to G_\ell(A)
\]
sending $\iota(\df) = [(\tc_\df)_\ell, (\overline{\tc}_{\df})_\ell]$ to $\rho(\iota(\df)) = X_{[(\tc_\df)_\ell, (\overline{\tc}_{\df})_\ell]} \in G_\ell(A)$.
%We now prove that $\rho$ can be extended to a morphism of groupoids.

The next proposition provides the connection between the groupoid constructed in the previous section and the picture group introduced above.

\begin{proposition}\label{prop:from groupoid to group}
The map $\rho: \Image(\iota) \subset \gc_\ell^{ss}(A) \to G_\ell(A)$ defined as 
\[
\rho\left([(\tc_\sigma)_\ell, (\overline{\tc_\sigma})_\ell]\right) = X_{[(\tc_\sigma)_\ell, (\overline{\tc_\sigma})_\ell]}
\]
for every $\sigma\in \Df_\ell(A)$ extends to an morphism $\rho: \gc^{ss}_\ell(A) \to G_\ell(A)$ of groupoids, where $G_\ell(A)$ is thought of as a groupoid with one object.
Moreover $G_\ell(A)$ is generated by the image of $\rho$.
\end{proposition}

\begin{proof}
We first define $\rho$ on the generators of $\gc^{ss}_\ell(A)$. 
We then show that it extends naturally to intervals, and finally to all of $\gc^{ss}_\ell(A)$. 

As a first step, notice that the set $\{\iota(\df')\mid \df' \in \Df_\ell^1(A)\}$ generates the groupoid $\gc_\ell^{ss}(A)$ since this set is in bijection with the maximal containments of torsion classes in $\tors^{ss}_\ell(A)$. 
Moreover, if $\df'_1, \df'_2 \in \Df^1_\ell(A)$ are two cones of codimension $1$ such that $\overline{\tc_{\df'_1}} = \tc_{\df'_2}$ we have in $\gc_\ell^{ss}(A)$ that
$$[\tc_{\df'_1}, \overline{\tc_{\df'_1}} ][\tc_{\df'_2}, \overline{\tc_{\df'_2}} ] =[\tc_{\df'_1}, \overline{\tc_{\df'_2}}].$$
On the other hand,
$$
\rho(\iota(\df'_1))\rho(\iota(\df'_2))=X_{[\tc_{\df'_1}, \overline{\tc_{\df'_1}} ]}X_{[\tc_{\df'_2}, \overline{\tc_{\df'_2}} ] }=X_{[\tc_{\df'_1}, \overline{\tc_{\df'_2}}]}
$$
in $G_\ell(A)$. 

More generally, let $\df \in \Df^0_\ell(A)$ be a codimension $0$ cone and let $(\tc_\df)_\ell\in \tors^{ss}_\ell A$ be the torsion class associated to it.  
Consider a set $\{\df_i\in \Df_\ell^0(A) \mid 0\leq i \leq t \}$ of codimension $0$ cones such that $\df_0 = \cc^-$ is the negative chamber, $\df_t = \df$ and $(\tc_{\df_{i-1}})_\ell$ is maximally contained in $(\tc_{\df_{i}})_\ell$ for every $1\leq i \leq t$.
This set induces a unique set $\{\df'_i \in \Df_\ell^1(A) \mid 1 \leq i \leq t\}$ of codimension $1$ cones such that $\df'_i = \overline{\df_{i-1}} \cap \overline{\df_{i}}$.
This is the set of all the walls crossed in the chosen path between $\df_0$ and $\df_t$.
%STThe chosen chain of maximal cones determines a path from the negative chamber to $\df$. 
The corresponding product in the groupoid is
\[
\iota(\df_0)\iota(\df'_1)\iota(\df_1)\iota(\df'_2)\dots\iota(\df_{t-1})\iota(\df'_t)\iota(\df_t)
\]
is well-defined in $\gc_\ell^{ss}(A)$ and it is 
\begin{multline}\label{eq:map of gruopoids 1}
\iota(\df_0)\iota(\df'_1)\iota(\df_1)\iota(\df'_2)\dots\iota(\df_{t-1})\iota(\df'_t)\iota(\df_t) =\\ [\{0\}, \{0\}][\{0\}, (\tc_{\df_1})_\ell]\dots[(\tc_{\df_{t-1}})_\ell, (\tc_{\df_t})_\ell][(\tc_{\df_t})_\ell, (\tc_{\df_t})_\ell]=\\
[\{0\}, (\tc_{\df_t})_\ell] = [\{0\}, (\tc_{\df})_\ell] 
\end{multline}
by definition.
Applying the map $\rho$ to the set $\{\iota(\df_i) \mid 0\leq i \leq t\} \cup \{\iota(\df'_i) \mid 1\leq i \leq t\}$ and multiplying as in \cref{eq:map of gruopoids 1} we obtain
\begin{multline}\label{eq:map of gruopoids 2}
\rho(\iota(\df_0))\rho(\iota(\df'_1))\rho(\iota(\df_1))\dots\rho(\iota(\df'_t))\rho(\iota(\df_t)) =\\ 
X_{[\{0\}, \{0\}]}X_{[\{0\}, (\tc_{\df_1})_\ell]}\dots X_{[(\tc_{\df_{t-1}})_\ell, (\tc_{\df_t})_\ell]}X_{[(\tc_{\df_t})_\ell, (\tc_{\df_t})_\ell]}=X_{[\{0\}, (\tc_{\df})_\ell]} 
\end{multline}
by \cref{def:l-picture group1,rmk:different notations}.
This shows that $\rho$ can be extended naturally to a map 
\[
\rho: \ic(\tors_\ell^{ss}A) \to G_\ell(A)
\]
sending an interval $[(\tc)_\ell, (\tc'_\ell)]$ to $\rho([(\tc)_\ell, (\tc'_\ell)])=X_{[(\tc)_\ell, (\tc'_\ell)]}$. 
Finally define
\[
\rho\left([(\tc_\sigma)_\ell, (\overline{\tc_\sigma})_\ell]^{-1}\right) = X^{-1}_{[(\tc_\sigma)_\ell, (\overline{\tc_\sigma})_\ell]}
\]
 for every $[(\tc_\sigma)_\ell, (\overline{\tc_\sigma})_\ell]^{-1} \in \ic^{-1}(\tors_\ell^{ss}A)$.
Since  every element in $\gc^{ss}_{\ell}(A)$ is of the form $[\{0\},\tc_\df]^{-1}[\{0\},\tc_{\df'}]$ by \cref{prop: one map} the map extends uniquely to a morphism of groupoids.

The statement about generators follows immediately from the presentation of $G_\ell(A)$ in \cref{def:l-picture group2}, since every generator corresponds to the image of a cone.
\end{proof}

\begin{remark}
It is worth noting that $G_\ell(A)$ is is generally much larger than $\Image(\rho)$ since in $G_\ell(A)$ one can always make the multiplication 
$X_{[(\tc_1)_\ell, (\tc_2)_\ell]}X_{[(\tc_3)_\ell, (\tc_4)_\ell]}$ regardless if $(\tc_2)_\ell$ and $(\tc_3)_\ell$ coincide or not.
\end{remark}

The following result follows from an easy verification as in \cref{prop:morphismstodiagram,prop:composition}.

\begin{proposition}
Let $A$ be an Artin algebra and let $\ell,m$ be two positive integers such that $\ell< m$. 
Then there is an surjective group homomorphism $\pi_{m,\ell} : G_m(A) \to G_\ell(A)$ defined as 
$$\pi_{m,\ell} \left(X_{[\tc_m, \tc'_m]}\right) = X_{[\tc_\ell, \tc'_\ell]}$$
Moreover, if $\ell_1 < \ell_2 < \ell_3$ are three positive integers then $\pi_{\ell_3, \ell_2} \circ\pi_{\ell_2, \ell_1} = \pi_{\ell_3, \ell_1}$.
\end{proposition}

The family of groups \mbox{$\{G_\ell(A) \mid \ell \in \mathbb{N}\}$}, together with the compatible surjective homomorphism $\pi_{m,\ell} : G_m(A) \to G_\ell(A)$ forms an inverse system. 
This allows us to define a picture group for arbitrary Artin algebras.

\begin{definition}\label{def:picturegroup}
Let $A$ be an Artin algebra. 
We define the picture group $G(A)$ of $A$ as 
\[
G(A) := \underleftarrow\lim\, G_\ell(A).
\]
\end{definition}

\begin{remark}\label{rmk:tuples}
It is useful to keep in mind the following concrete description of inverse limits.
Every group that is defined as an inverse limit $G:= \underleftarrow\lim\, G_\ell$ has an explicit description as the subgroup of $\prod_{\ell \in \mathbb{N}}G_\ell$ consisting on all the infinite tuples $(X_{[\tc_\ell, \tc'_\ell]})_{\ell\in \mathbb{N}}$ such that $\pi_{m,\ell}\left(X_{[\tc_m, \tc'_m]}\right)=X_{[\tc_\ell, \tc'_\ell]}$ for every $\ell,m\in \mathbb{N}$. 
For the particular case of $G(A)$, we get that every generator $X_{[\tc, \tc']}$ of $G(A)$ corresponds to the tuple 
\[
X_{[\tc, \tc']} = \left(X_{[\tc_\ell, \tc'_\ell]}\right)_{\ell \in \mathbb{N}}\in G(A).
\]

The picture group $G(A)$ therefore encodes simultaneously the picture groups of all finite-length approximations $(\modA)_\ell$, and will be the group considered in \cref{sec:torsion scattering diagram}.
\end{remark}

\begin{remark}
As pointed out before, the definition of a picture group $G(A)$ for a $\tau$-tilting finite algebra was first given in \cite{HansonIgusa} and it is, a priori, different from \cref{def:picturegroup}. 
However, it was shown in \cite[Proposition~6.19]{Kaipel} that the picture group $G(A)$ defined in \cite{HansonIgusa} for a $\tau$-tilting finite algebra $A$ is completely determined by the fan poset structure of $\Df(A)$. 
It then follows from \cref{cor:W&C tau-tilting finite} that the groups of \cref{def:picturegroup} and \cite[Definition-Theorem~4.3]{HansonIgusa} are isomorphic. 

There is a difference to be noted, though. 
In \cite[Definition-Theorem~4.3]{HansonIgusa} the generators in of $G(A)$ are labeled by the bricks in $\modA$ while in \cref{def:picturegroup} the natural labelling of the generators correspond to the subcategories $\mod^{ss}_\sigma A$ for every $\sigma \in \Df^1(A)$. 

There is a natural connection between these two labelings since for a $\tau$-tilting finite algebra $A$ the categories $\mod^{ss}_\sigma A$ are wide subcategories having exactly one relative simple module $S$, which is a brick. 
This gives rise to a map that happens to be a bijection for $\tau$-tilting finite algebras, see \cite{AsaiSemibricks,Treffinger_c-vectors}. 
This analogy breaks in the more general case of algebras with infinitely many torsion classes in its module category, since there are $\sigma 
\in \Df^1(A)$ having infinitely many relative simple objects in $\mod^{ss}_\sigma A$, see \cref{fig:Df Kronecker}.
\end{remark}

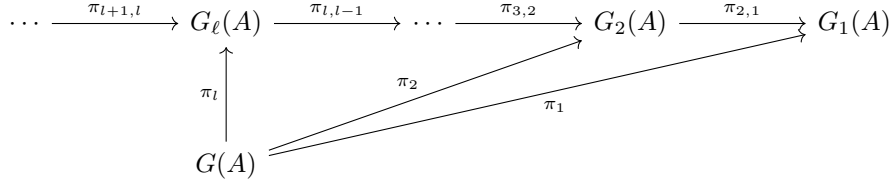
\begin{figure}
    \centering
% https://q.uiver.app/#q=WzAsNixbOCwwLCIoXFx0b3JzIEEpXzAiXSxbNiwwLCIoXFx0b3JzIEEpXzEiXSxbNCwwLCJcXGNkb3RzIl0sWzIsMCwiIChcXHRvcnMgQSlfbCJdLFsyLDIsIlxcdG9ycyBBIl0sWzAsMCwiXFxjZG90cyJdLFszLDIsIlxcdmFycGhpX3tsLGwtMX0iLDAseyJzdHlsZSI6eyJoZWFkIjp7Im5hbWUiOiJlcGkifX19XSxbMiwxLCJcXHZhcnBoaV97MiwxfSIsMCx7InN0eWxlIjp7ImhlYWQiOnsibmFtZSI6ImVwaSJ9fX1dLFsxLDAsIlxcdmFycGhpX3sxLDB9IiwwLHsic3R5bGUiOnsiaGVhZCI6eyJuYW1lIjoiZXBpIn19fV0sWzQsMywiXFx2YXJwaGlfe1xcaW5mdHksbH0iLDAseyJzdHlsZSI6eyJoZWFkIjp7Im5hbWUiOiJlcGkifX19XSxbNCwxLCJcXHZhcnBoaV97XFxpbmZ0eSwxfSIsMCx7InN0eWxlIjp7ImhlYWQiOnsibmFtZSI6ImVwaSJ9fX1dLFs0LDAsIlxcdmFycGhpX3tcXGluZnR5LDB9IiwyLHsic3R5bGUiOnsiaGVhZCI6eyJuYW1lIjoiZXBpIn19fV0sWzUsMywiXFx2YXJwaGlfe2wrMSxsfSJdXQ==
\[\begin{tikzcd}
	\cdots && { G_\ell(A)} && \cdots && {G_2(A)} && {G_1(A)} \\
	\\
	&& {G(A)}
	\arrow["{\pi_{l+1,l}}", from=1-1, to=1-3]
	\arrow["{\pi_{l,l-1}}", from=1-3, to=1-5]
	\arrow["{\pi_{3,2}}", from=1-5, to=1-7]
	\arrow["{\pi_{2,1}}", from=1-7, to=1-9]
	\arrow["{\pi_{l}}", from=3-3, to=1-3]
	\arrow["{\pi_{2}}", from=3-3, to=1-7]
	\arrow["{\pi_{1}}"', from=3-3, to=1-9]
\end{tikzcd}\]
    \caption{The picture group of an Artin algebra.}
    \label{fig:picture group}
\end{figure}

%%%%%%%%%%%%%%%
\subsection{A categorical approach to picture groups}
In this subsection we give a categorical realization of the picture group by identifying it with a group generated by suitable subcategories of $(\mod A)_\ell$. 
More precisely, we prove that the picture group $G_\ell(A)$ is isomorphic to the group $\widetilde{G}_\ell(A)$ defined in \cref{prop:deftildeG}. 
Compare with \cite[Proposition~A.5]{BKH}.

As an application, we obtain a faithful functor from the $\tau$-cluster morphism category to the picture group.

\begin{theorem}\label{prop:categorical interpretation}
The picture group $G_\ell(A)$  is naturally isomorphic to the group   $\widetilde{G}_\ell(A)$.
% generated by the set 
%\[
%\{(\tc_{\df})_\ell \mid \df \in \Df^{0}_\ell(A)\} \text{ and } \{(\mod_{\df'}^{ss}A)_\ell \mid \df' \in \Df^1_\ell(A)\}  
%\]
%with the $*$-product defined above and identity element $\{0\}\subset (\modA)_\ell$.
\end{theorem}

To prove this theorem, we first establish two auxiliary lemmas.
The first constructs a natural map $\varphi : {G}_\ell(A) \to \widetilde{G}_\ell(A)$.
The second lemma is the key step. 
It shows that the relations in the picture group identify only generators corresponding to the same semistable category.
As a consequence, it shows that $\varphi$ is well-defined in the generators and therefore a group homomorphism.

\begin{lemma}\label{lem: map between groups}
Let $G_\ell(A)$ and $\widetilde{G}_\ell(A)$ as above.
There is a map $\varphi : {G}_\ell(A) \to \widetilde{G}_\ell(A)$ compatible with the product which is defined on generators as $\varphi(g_\df) = (\tc_\df)_\ell$ and \mbox{$\varphi(g_{\df'}) = (\mod_{\df'}^{ss}A)_\ell$} for every $\df \in \Df_\ell^0(A)$ and every $\df'\in \Df^1_\ell(A)$, respectively.
\end{lemma}

\begin{proof}
Since $\varphi: {G}_\ell(A) \to \widetilde{G}_\ell(A)$ is defined on generators, we only need to show that it respects the defining relations of the respective groups.

Let $\df' \in \Df^1_\ell(A)$ be a cone separating two chambers $\df_1, \df_2 \in \Df^0_\ell(A)$.
Then, without loss of generality we can assume $(\tc_{\df'})_\ell = (\tc_{\df_1})_\ell$ and $(\overline{\tc_{\df'}})_\ell= (\tc_{\df_2})_\ell$.
Moreover $(\overline{\tc_{\df'}})_\ell \cap (\fc_{\df'})_\ell = (\mod_{\df'}^{ss}A)_\ell$, see \cref{cor: maximal inclusions tors^ss}. 
Also, \cref{prop:infinite products} implies that 
\[
\varphi(g_{\df_1} g_{\df'})=\varphi(g_{\df_2}) = (\tc_{\df_2})_\ell = (\tc_{\df_1})_\ell * (\mod_{\df'}^{ss}A)_\ell = \varphi(g_{\df_1})*\varphi(g_{\df'}).
\]
So we can conclude that $\varphi: {G}_\ell(A) \to \widetilde{G}_\ell(A)$ respects the multiplication.
\end{proof}

\begin{lemma}\label{lem:generators picture group}
Let $G_\ell(A)$ be the picture group as defined above. 
If two codimension-one cones $\df', \ef'$ determine the same element of the picture group, then 
\[
(\mod_{\df'}^{ss}A)_\ell = (\mod_{\ef'}^{ss}A)_\ell.
\]
\end{lemma}

\begin{proof}
As $\varphi$ respects the $*$-product, we know that $ \xc *\varphi(g_{\df'}) = \xc *\varphi(g_{\ef'}) $ for every subcategory $\xc \subset (\modA)_\ell$. 
In particular, we have that 
$ (\tc_{\ef_1})_\ell *\varphi(g_{\df'}) = (\tc_{\ef_1})_\ell *\varphi(g_{\ef'}) $
where $\ef_1$ is the codimension $0$ cone such that $(\tc_{\ef_1})_\ell = (\tc_{\ef'})_\ell$.
Now, $(\tc_{\ef_1})_\ell = \varphi(g_{\ef_1})$. 
The key observation is that the semistable category can be recovered from its extension product with the adjacent torsion class.
So 
\begin{equation}\label{eq:iso1}
(\tc_{\ef_1})_\ell *\varphi(g_{\df'}) = (\tc_{\ef_1})_\ell *\varphi(g_{\ef'}) = \varphi(g_{\ef_1}) *\varphi(g_{\ef'}) = \varphi(g_{\ef_1}g_{\ef'})= \varphi(g_{\ef_2}) = (\tc_{\ef_2})_\ell
\end{equation}
where $\ef_2$ is the codimension $0$ cone such that $(\tc_{\ef_2})_\ell = (\overline{\tc_{\ef'}})_\ell$.
\cref{eq:iso1} implies that for every object $M \in (\tc_{\ef_2})_\ell$ there is a short exact sequence 
\begin{equation}\label{eq:iso2}
0 \to M_{\ef_1} \to M \to M_{\df'}\to 0
\end{equation}
with $M_{\ef_1} \in (\tc_{\ef_1})_\ell$ and $M_{\df'} \in (\mod_{\df'}^{ss}A)_\ell$. 
Now, $(\mod_{\ef'}^{ss}A)_\ell = (\tc_{\ef_2})_\ell \cap (\fc_{\ef_1})_\ell$, where $(\fc_{\ef_1})_\ell$ is the torsion free class in $(\modA)_\ell$ corresponding to $(\tc_{\ef_1})_\ell$.
So $(\mod_{\ef'}^{ss}A)_\ell \subset (\tc_{\ef_2})_\ell$ and $\Hom_A(M_{\ef_1}, M_{\ef'})=0$ for every $M_{\ef_1} \in(\tc_{\ef_1})_\ell $ and $M_{\ef'} \in (\mod_{\ef'}^{ss}A)_\ell$.
As a consequence, \cref{eq:iso2} specifies to 
\begin{equation}\label{eq:iso3}
0 \to 0 \to M_{\ef'} \to M_{\df'}\to 0
\end{equation}
if $M_{\ef'} \in (\mod_{\ef'}^{ss}A)_\ell$. 
This allows us to conclude that $(\mod_{\ef'}^{ss}A)_\ell \subset (\mod_{\df'}^{ss}A)_\ell$.
Using the same argument we can show the reverse inclusion, implying that  
\[
\varphi(g_{\ef'}) = (\mod_{\ef'}^{ss}A)_\ell = (\mod_{\df'}^{ss}A)_\ell = \varphi(g_{\df'}).
\]
\end{proof}

We now prove \cref{prop:categorical interpretation}.

\begin{proof}[Proof of \cref{prop:categorical interpretation}]

By \cref{lem: map between groups,lem:generators picture group}, the assignment on generators extends to a well-defined surjective homomorphism $\varphi : {G}_\ell(A) \to \widetilde{G}_\ell(A)$. 
Since every defining generator of $\widetilde{G}_\ell(A)$ corresponds to a unique generator of ${G}_\ell(A)$, the inverse assignment defines a homomorphism
\(
\widetilde{\varphi}: \widetilde{G}_\ell(A) \to G_\ell(A)
\)
where $\widetilde{\varphi}(\tc_{\df}) = g_{\df}$ for every $\df \in \Df^0_(A)$ and $\widetilde{\varphi}((\mod_{\df'}^{ss}A)_\ell) = g_{\df'}$ for every $\df' \in \Df^1_\ell(A)$.
By construction, the presentations of ${G}_\ell(A)$ and  $\widetilde{G}_\ell(A)$ have corresponding generators and defining relations.
Therefore the inverse assignment on generators extends to a homomorphism
\(
\widetilde{\varphi}: \widetilde{G}_\ell(A) \to G_\ell(A).
\)
The two homomorphisms are inverse on generators, hence mutually inverse.

\end{proof}

We now combine the categorical realization above with Hanson's construction of a faithful functor.

\begin{corollary}\label{thm:faithfulfunctor}
Let $A$ be an Artin algebra and let $\ell$ be a positive integer. 
Then there is a faithful functor 
$\Gamma_\ell : \Cf_\ell(A) \to G_\ell(A)$.
In particular, for every $\tau$-tilting finite algebra $A$  there is a faithful functor 
$\Gamma : \Cf(A) \to G(A)$.
\end{corollary}

\begin{proof}
This follows directly from \cref{prop:categorical interpretation} and \cite[Theorem~A.8]{BKH}
\end{proof}

\begin{remark}
In \cite[Appendix~A]{BKH}, Hanson also considered a group that he named the interval-heart torsion group, denoted by $\operatorname{int-heart}(\tors A)$, to give an interpretation of the picture group as a group of subcategories. 
Using this group, he shows the existence of a group functor $\Gamma: \Cf(A) \to \operatorname{int-heart}(\tors A)$. 
Using this approach he has shown that for every $\tau$-tilting finite $k$-algebra $A$ over a finite field $k$, the functor $\Gamma$ is faithful. 
A similar result was shown independently for finite-dimensional $\mathbb{C}$-algebras in \cite{Borvefaithful}, see \cref{rmk: faithful C-alg}.
\end{remark}

%%%%%%%%%%%%%%%%%%
\section{Torsion scattering diagrams}\label{sec:torsion scattering diagram}

In this section we introduce torsion scattering diagrams for Artin algebras and prove that they are minimal and consistent. 
The strategy is simple. 
We first construct compatible scattering diagrams on the finite cone complexes $\Df_\ell(A)$.
Passing to the inverse limit then yields the desired scattering diagram on $\Df(A)$.

The section is organized as follows. 
In \cref{subsec: preliminary} we recall the basic definitions and fundamental results on scattering diagrams.
In \cref{subsec: Scattering diagram for (modA)_l} we construct, for every Artin algebra $A$ and every $\ell\in\mathbb N$, a minimal consistent scattering diagram
$$(\Df_\ell(A), \Phi:\Df^1_\ell(A)\to G_\ell(A))$$
where $\Df_\ell(A)$ is the cone complex introduced in \cref{sec:conecomplex} and $G_\ell(A)$ is the picture group defined in \cref{sec:picture group}.
Finally, in \cref{subsec: scattering diagrams for modA} we construct a minimal consistent scattering diagram 
$$(\Df(A), \Phi:\Df^1(A)\to G(A))$$
 for every Artin algebra $A$.

%%%%%%%%%%%%%%%%
\subsection{Background and preliminary results}\label{subsec: preliminary}
In this section, $\Df$ is a cone complex in $\mathbb{R}^n$.
Recall that the codimension of a cone $\df\in \Df$ is the dimension of the subspace 
$$\df^\perp=\{v \in \mathbb{R}^n \mid \langle v, d\rangle= 0 \text{ for all $d\in \df$}\}.$$ 
We assume that $\Df$ has at least two distinguished cones $\cc^+$ and $\cc^-$ of codimension $0$ that we call the positive and the negative chambers, respectively.
In the examples considered in this paper, $\Df$ may have infinitely many additional chambers.
We also recall that 
\[
\Df = \bigcup_{i=0}^n \Df^i
\]
where $\Df^i:=\{\df\in \Df \mid \codim \df = i\}$.
Note that for every $\df \in \Df^1$ there exists a vector $v\in \mathbb{R}^n_{\geq 0}$ such that \(\df \subseteq v^{\perp}\), unique up to scalar multiplication.

We begin with the definition of a scattering diagram.

\begin{definition}\label{def:scatteringdiagramI}
Let $G$ be a group. 
A scattering diagram is a pair $(\Df, \Phi: \Df^1 \to G)$ where $\Df$ is a cone complex in $\mathbb{R}^n$ and $\Phi: \Df^1 \to G$ is a map.
A scattering diagram is said to be minimal if $\Phi(\df) \neq e$ for every $\df \in \Df^1$, where $e$ is the identity element in $G$.
\end{definition}

To define consistency, we first introduce the notion of a generic path in a cone complex.

\begin{definition}\label{def:genericpath}
    Let $\gamma:[0,1] \to \Df$ be a smooth map. 
    We say that $\gamma$ is a $\Df$-generic path if the following conditions are satisfied. 
    \begin{enumerate}
        \item There are cones $\cc_1$ and $\cc_2$ of codimension $0$ such that $\gamma(0)\in\cc_1$ and $\gamma(1)\in \cc_2$.
        \item Whenever $\gamma(t)\in\df$, the cone $\df$ has codimension at most $1$.
        \item Let $t\in [0,1]$ and $\df\in \Df^1$. 
        If $\gamma(t)\in \df$, where $\df \subseteq v^\perp$ then $\langle\gamma'(t), v \rangle \neq 0$.
\end{enumerate}
\end{definition}

The following lemma is a straightforward consequence of elementary real analysis.

\begin{lemma}
    Let $\gamma:[0,1]\to \Df$ be a $\Df$-generic path. 
    Then $\gamma$ intersects every cone $\df\in\Df^1$ only finitely many times.
\end{lemma}

\begin{proof}
    Consider the set $S_\df:=\{t \in [0,1]\mid \gamma(t) \in \df\} \subset [0,1]$ and suppose, by contradiction, that $S_\df$ is an infinite set. 
    Since $[0,1]$ is compact, there exists an accumulation point $t_0\in [0,1]$  for $S_\df$. 
   By continuity of $\gamma$, we have $\gamma(t_0)\in \overline{\df}$.
   Since $\gamma$ is $\Df$-generic, condition (2) implies that $\gamma(t_0)\notin\overline{\df}\setminus\df$.
   %Moreover, $\gamma(t_0) \not \in \overline{\df}\setminus\df$ due to \cref{def:genericpath}.2.
    In particular, $t_0 \in S_\df$ and $S_\df = \overline{S_\df}$.

    Consider the function \(f_\gamma(t):=\langle\gamma(t), v\rangle: [0,1] \to \mathbb{R}.\)
    Since $\gamma$ is smooth, so are $f_\gamma$ and its derivatives.
    Observe that $f_{\gamma}(t)=0$ for every $t\in S_{\df}$ since $\gamma(t)\in \df\subseteq v^\perp$ for all $t\in S_\df$.
    
    Since $S_\df$ is compact and $f_\gamma'$ does not vanish on $S_\df$, continuity of $f_\gamma'$ implies that there exists $\varepsilon>0$ such that $|f'_\gamma(t)|\geq \varepsilon$.
     Now consider a totally ordered subset 
    $$S':= \{t_i \in S_\df \mid i \in \mathbb{N}\}$$
    of $S_\df$ from which we build a new set $\widetilde{S}$ as follows. If $sg(f'_\gamma(t_i)) \neq sg(f'_\gamma (t_{i+1}))$ we do nothing. 
    Otherwise, we add to $\widetilde{S}$ an element $t_{i+\frac{1}{2}}$ such that $f_\gamma (t_{i+\frac{1}{2}})=0$ and $sg(f'_\gamma (t_{i+\frac{1}{2}})) \neq sg(f'_\gamma(t_i))$. 
    We note that such a $t_{i+\frac{1}{2}}$ exists due to the continuity of $f_\gamma$ and $f'_\gamma$, although $\gamma(t_{i+\frac{1}{2}})$ might not belong to $\df$. 
    Finally we reindex the elements of $\widetilde{S}$ preserving the order so $\widetilde{S} = \{s_i \in [0,1] \mid i \in \mathbb{N}\}$. 
    Then by Lagrange's mean value theorem there is a set $C= \{c_i \in [0,1] \mid s_i < c_i < s_{i+1}\}$ such that $f''(c_i) = \frac{f'(s_{i+1})-f'(s_i)}{s_{i+1} - s_i}$. Taking limits we get the following.
    \[
   \lim_{i \to \infty} |f''(c_i)| = \lim_{i \to \infty} \frac{\left|f'(s_{i+1})-f'(s_i)\right|}{|s_{i+1} - s_i|}\geq \lim_{i \to \infty} \frac{\varepsilon}{|s_{i+1} - s_i|} \to \infty
    \]    
    This implies that there is no $R\in \mathbb{R}$ such that $|f''_\gamma(x)|\leq R$ for every $x\in[0,1]$. Hence $f''_\gamma$ is not smooth.    This is contradiction shows that $S_\df$ must be finite.
\end{proof}

As an immediate consequence we obtain the following.

\begin{corollary}\label{prop:finitecrossing}
    Let $\Df$ be a cone complex and let $\gamma: [0,1]\to \Df$ be a generic path.
    If $\Df$ has finitely many cones, then the set  
    \[
    S_\gamma := \{t\in [0,1] \mid \gamma(t) \in \df \text{ where $\df\in \Df^1$}\}
    \]
is finite.
\end{corollary}

Let $\gamma$ be a $\Df$-generic path and consider $S_\gamma$ as above. 
Then $S_\gamma$ inherits a natural total order from $[0,1]$. 
For every $i \in S_\gamma$ we set $g_i := \Phi(\df_i)$ and $\epsilon_i := sg(\langle v_i, \gamma'(t_i) \rangle)$, where $v_i \in \mathbb{R}^n$ and $\df_i \subset v_i^\perp$.
Since $S_\gamma$ is totally ordered, the ordered product.
\[
g_\gamma = \prod_{i\in S_\gamma} g_i^{\epsilon_i}
\]
is well defined whenever it exists.

\begin{definition}\label{def:scatteringdiagramII}
Let $(\Df, \Phi: \Df^1 \to G)$ be a scattering diagram. 
We say that the scattering diagram is well defined if $g_\gamma$ exists for every $\Df$-generic path. 
It is consistent whenever $g_\gamma$ is completely determined by $\gamma(0)$ and $\gamma(1)$. 
\end{definition}

Let $(\Df, \Phi: \Df^1 \to G)$ be a scattering diagram where the cone complex $\Df$ has infinitely many cones. 
When $\Df$ has infinitely many walls, it is no longer automatic that $g_\gamma$ is well defined, since a generic path may cross infinitely many walls and therefore determine an infinite product in $G$.

\begin{remark}
If $(\Df, \Phi: \Df^1 \to G)$ is a consistent scattering diagram, then there is an element $g\in G$ such that $g=g_\gamma$ for every $\Df$-generic path $\gamma$ such that $\gamma(0)$ belongs to the negative chamber $\cc^-$ and $\gamma(1)$ belongs to the positive chamber $\cc^+$. 
In the literature it is often said that in this case $(\Df, \Phi: \Df^1 \to G)$ is a scattering diagram for $g$. 
It was shown in \cite{KontsevichSoibelman} that for a $g\in G$ there is at most one minimal consistent scattering diagram in $\mathbb{R}^n$, up to isomorphism.
In this paper we construct scattering diagrams in which the group $G$, the element $g$, and the ambient dimension $n$ are determined by the Artin algebra under consideration.
\end{remark}

We finish this subsection with a technical lemma.

\begin{lemma}\label{lem:restrictionofpaths}
Let $\Df$ and $\Df'$ be two cone complexes of dimension $n$ such that $\Df$ is a refinement of $\Df'$. 
Then every $\Df$-generic path $\gamma:[0,1] \to \Df$ restricts to a $\Df'$-generic path $\gamma: [0,1]\to \Df'$.
\end{lemma}

\begin{proof}
Let $\gamma:[0,1]\to \Df$ be a $\Df$-generic path and let $\df \in \Df^1$ such that $\gamma(t) \in \df$. Then there is a cone $\df' \in \Df'^1$ such that $\df \subset \df'$  because $\Df$ is a refinement of $\Df'$. 
Let $t\in[0,1]$. 
If $\gamma(t)$ belongs to a codimension-one cone of $\Df$, then it belongs to a codimension-one cone of $\Df'$ containing it.
The transversality condition is preserved because the supporting hyperplane is the same. 
The remaining conditions are immediate.
\end{proof}

%%%%%%%%%%%%%%%%
\subsection{Scattering diagrams for the category of modules of bounded length}\label{subsec: Scattering diagram for (modA)_l}

In this subsection we construct the scattering diagram for $(\modA)_\ell$ for every Artin algebra $A$ and every positive integer $\ell$.
Recall from \cref{subsec: W&C modA_l} that the $\ell$-TF-equivalence classes of stability conditions over $(\modA)_\ell$ form a finite cone complex
\[
\Df_\ell(A) = \bigcup_{v \in \mathbb{R}^n} \df_v^\ell.
\]
By \cref{cor:finite cone complex for l}, $\Df_\ell(A)$ is finite and every codimension-one cone is contained in the hyperplane $[B]^\perp$ for some brick $B\in(\modA)_\ell$.

Following \cref{sec:groupoids,sec:picture group}, let $\gc_\ell^{ss}(A)$ denote the groupoid associated with the lattice $\tors_\ell^{ss}(A)$, and let $G_\ell(A)$ denote the corresponding picture group.
Before proving the main result of this section, we give an easy but key lemma. 

\begin{lemma}\label{lem:liftingtogroupoid}
Let $\gamma: [0,1] \to \Df_\ell(A)$ be a $\Df$-generic path. 
Then $\gamma$ determines a well-defined element 
\[
g_\gamma = \prod_{i\in S_\gamma} [\tc_{\sigma_i}, \overline{\tc}_{\sigma_i}]^{\epsilon_i} \in \gc_\ell^{ss}(A).
\]
Moreover, $g_\gamma = [\{0\}, \tc_{\gamma(0)}]^{-1}[\{0\}, \tc_{\gamma(1)}]$.
\end{lemma}

\begin{proof}
Let $\sigma\in\Df_\ell^1(A)$. 
Then there are precisely two adjacent chambers 
$\sigma^+, \sigma^- \in \Df^0_\ell(A)$
 satisfying  $\sigma \subset \overline{\sigma^+}\cap  \overline{\sigma^-}$.
Moreover, it follows from \cref{prop:fanposet} that $$(\tc_\sigma )_\ell =(\tc_{\sigma^-})_\ell= (\overline{\tc}_{\sigma^-})_\ell \quad \text{ and } \quad (\overline{\tc}_\sigma)_\ell =(\tc_{\sigma^+})_\ell= (\overline{\tc}_{\sigma^+})_\ell.$$
Let $\gamma: [0,1] \to \Df_\ell(A)$ be a $\Df_\ell(A)$-generic path and let $t_i<t_{i+1}$ be two consecutive elements of $S_\gamma$.
By the definition of $S_\gamma$, there is a chamber $\cc_i \in \Df^0_\ell(A)$  such that $\gamma(t)\in\cc_i$ for every $t\in (t_i, t_{i+1})$.
Therefore the target of the $i$-th factor coincides with the source of the $(i+1)$-st factor, so the product
\[
[\tc_{\sigma_i}, \overline{\tc}_{\sigma_i}]^{\epsilon_i}[\tc_{\sigma_{i+1}}, \overline{\tc}_{\sigma_{i+1}}]^{\epsilon_{i+1}}
\]
s defined in $\gc_\ell^{ss}(A)$.
Consequently, the product 
\[
\prod_{i\in S_\gamma} [\tc_{\sigma_i}, \overline{\tc}_{\sigma_i}]^{\epsilon_i}.
\]
is well-defined in $\gc_\ell^{ss}(A)$.
The final statement follows immediately from \cref{prop: one map}.
\end{proof}

We are now ready to define the torsion scattering diagram for $(\modA)_\ell$ as follows. 

\begin{theorem}\label{thm:boundedscatteringdiagrams}
Let $A$ be an Artin algebra and $\ell$ be a positive integer. 
Then the pair $(\Df_\ell(A),\Phi^\ell_A: \Df_\ell^1(A) \to G_\ell(A))$ is a minimal consistent scattering diagram.
Here $\Df_\ell(A)$ is the cone complex determined by the $\ell$-TF-equivalence classes of stability conditions, $G_\ell(A)$ is the picture group of $\Df_\ell(A)$ and $\Phi^\ell_A: \Df_\ell^1(A) \to G_\ell(A)$ is defined as $\Phi^l_A(\df_v) = (\mod^{ss}_{\df_v} A)_\ell$.
\end{theorem}

\begin{proof}
Let $\gamma: [0,1] \to \Df_\ell(A)$ be a $\Df_\ell(A)$-generic path. 
Since $\Df_\ell(A)$ is finite, the product defining $g_\gamma$ is finite and therefore well defined.

Minimality follows because $(\mod_\df^{ss}A)_\ell$ is nonzero for every codimension-one cone $\df$.

Finally, the consistency of $(\Df_\ell(A),\Phi^l_A: \Df_\ell^1(A) \to G_\ell(A))$ follows from \cref{lem:liftingtogroupoid,prop:from groupoid to group}. 
By \cref{lem:liftingtogroupoid}, every generic path $\gamma$ determines the unique morphism
$$[\{0\}, \tc_{\gamma(0)}]^{-1}[\{0\}, \tc_{\gamma(1)}]$$
 in $\gc_\ell^{ss}(A)$.
Applying the canonical morphism $\rho: \gc_\ell(A) \to G_\ell(A)$
yields
$$g_\gamma = \rho([\{0\}, \tc_{\gamma(0)}]^{-1}[\{0\}, \tc_{\gamma(1)}]) \in G_\ell(A)$$
which depends only on the endpoints of $\gamma$.
Hence the scattering diagram is consistent.%Let $\gamma:[0,1] \to \Df_\ell(A)$ be a $\Df_\ell(A)$-generic path.
\end{proof}

\begin{remark}
The key idea behind the proof of the previous theorem is that any $\Df_\ell$-generic path $\gamma$ can be interpreted as a continuous walk in the lattice $\tors^{ss}_\ell(A)$. 
\end{remark}

%%%%%%%%%%%%%%%%%
\subsection{Scattering diagrams for module categories}\label{subsec: scattering diagrams for modA}

The main result of this subsection, and of the paper, is the following theorem.
The main difficulty is to show that every $\Df(A)$-generic path
$\gamma: [0,1] \to \Df(A)$ determines a well-defined element $g_\gamma\in G(A)$, since $\gamma$ may cross infinitely many walls.

\begin{theorem}\label{thm:scatteringdiagrams}
Let $A$ be an Artin algebra.
Then the pair $(\Df(A),\Phi_A: \Df^1(A) \to G(A))$ is a minimal consistent scattering diagram.
Here $\Df(A)$ is the cone complex determined by the TF-equivalence classes of stability conditions, $G(A)$ is the associated picture group and $\Phi_A(\df_v) = \mod^{ss}_{\df_v} A$.
\end{theorem}

\begin{proof}
It follows directly from the definition of $(\Df(A),\Phi_A: \Df^1(A) \to G(A))$ that it is a minimal scattering diagram. 
We need to show that it is well defined and consistent.

Let $\gamma: [0,1]\to \Df(A)$ be a $\Df(A)$-generic path. 
If $\ell \geq 1$.
 \cref{lem:restrictionofpaths} implies that $\gamma$ is also a $\Df_\ell(A)$-generic path. 
Then, $\gamma$ induces an element 
$$g^\ell_\gamma = \rho([\{0\}, (\tc_{\gamma(0)})_\ell]^{-1}[\{0\}, (\tc_{\gamma(1)})_\ell]) \in G_\ell(A)$$ for every $\ell \in \mathbb{N}$.
Note that for every pair of integers $\ell ,m \in \mathbb{N}$ with $\ell<m$ we have $\pi_{m,\ell}(g^m_\gamma) = g^\ell_\gamma$. 
Therefore the family $(g^\ell_\gamma)_{\ell\in\mathbb{N}} $
is compatible with the transition maps, and hence defines an element $g_\gamma\in G(A)$ by \cref{rmk:tuples}.
Since each $g_\gamma^\ell$ depends only on the endpoints of $\gamma$, the same is true for $g_\gamma$.
Hence the scattering diagram is well defined and consistent.
\end{proof}

Following \cite{BST}, a $\Df$-generic path is called green if $\gamma(0) \in \cc^-$, $\gamma(1) \in \cc^+$ and $\langle \gamma'(t), v \rangle >0$ whenever $\gamma(t)\in \df \subseteq v^\perp$. 
Dually, a $\Df$-generic path $\gamma$ is said to be red if $\gamma(0) \in \cc^+$, $\gamma(1) \in \cc^-$ and $\langle \gamma'(t), v \rangle <0$ whenever $\gamma(t)\in \df \subseteq v^\perp$. 

\begin{remark}
Recall that every consistent scattering diagram is determined by an element $g\in G$, which corresponds to $g_\gamma$ for some $\Df$-generic path $\gamma$ such that $\gamma(0) \in \cc^-$ and $\gamma(1) \in \cc^+$.
This characterization may be restricted to green paths.
For every green path $\gamma$, it follows from the results in \cref{sec:groupoids,sec:picture group} that the scattering diagram $(\Df(A),\Phi_A: \Df^1(A) \to G(A))$ we introduce in this section corresponds to the scattering diagram associated with the element
$\modA\in G(A)$, which is unique up to isomorphism.
Furthermore, consistency along green paths also follows directly from \cref{prop:infinite products}. 
However, this argument does not apply to arbitrary $\Df$-generic paths.
\end{remark}

\begin{remark}
Let $A$ and $A'$ be two algebras that are not Morita equivalent but that satisfy $\Df_\ell(A) = \Df_\ell(A')$ as cone complexes for every positive integer $\ell$.
Examples of such algebras are the GLS-algebras studied in \cite{Pfeifer}. 
Then it follows from our construction that $\Df(A) = \Df(A')$ as a cone complex. 
Moreover, their respective picture groups $G(A)$ and $G(A')$ are isomorphic but not equal. 
As a consequence, the scattering diagrams $(\Df(A),\Phi_A: \Df^1(A) \to G(A))$ and $(\Df(A'),\Phi_{A'}: \Df^1(A') \to G(A'))$ are isomorphic but distinct. 
\end{remark}

%%%%%%%%%%%%%%%%
\section{An isomorphism to the stability scattering diagram}\label{sec:isomorphism}

In this section, $Q$ is a finite quiver and $A=\mathbb{C}Q/I$ is the path algebra of $Q$ over $\mathbb{C}$ modulo an admissible ideal $I$. 
In particular, $A$ is an Artin algebra.

The main goal of this section is to show that the torsion scattering diagram for $A$, constructed in \cref{sec:torsion scattering diagram}, is isomorphic to the stability scattering diagram introduced by Bridgeland in \cite{BridgelandScat}.

We begin by recalling the necessary background on motivic Hall algebras, following \cite{BridgelandScat}; see also \cite{BridgelandHall}. 
We then review the definition of the stability scattering diagram and conclude by proving that it is isomorphic to the torsion scattering diagram constructed in this paper.

%%%%%%%%%%%%
\subsection{Motivic Hall algebras}
In this subsection we restrict to the case where $A$ is connected a finite-dimensional algebra over $\mathbb{C}$. 
By a result of Gabriel, we can assume without loss of generality that $A$ is isomorphic to $\mathbb{C}Q/I$, the path algebra of a finite quiver $Q$ bound by an admissible ideal $I\subset \mathbb{C}Q$. 
It is well-known that the category $\modA$ in this case is equivalent to the category $\operatorname{rep}(Q,I)$ of representations of the quiver $Q$ satisfying the relations imposed by $I$. 

With our conventions, the class $[M]\in K_0(A)$ is the dimension vector of $M$. 
This terminology comes from the identification of $\modA$ with $\operatorname{rep}(Q,I)$.
%fact that $M$ is 
%$$M = (\{M_i\in \operatorname{vect}({\mathbb{C}}) \mid {i \in Q_0} \}, \{f_\alpha \in \Hom_\mathbb{C}(M_{s(\alpha)}, M_{t(\alpha)})\mid {\alpha \in Q_1})\}.$$
%when thought as an object in $\operatorname{rep}(Q,I)$.
Hence $[M]_i = \dim_{\mathbb{C}} M_i$.
In particular, the composition length of $M$ coincides with the dimension $\dim_\mathbb{C} M$ of $M$ as a $\mathbb{C}$-vector space.

Associated with $A$ is an Artin stack of finite type $\mc$ parametrising the objects of $\modA$. 
There is also an Artin stack $\mc^{(2)}$ parametrising all short exact sequences in $\modA$. 
An object $\xi$ in $\mc^{(2)}$ is a short exact sequence of the form 
\[
\xi :=0 \to M_1 \to E \to M_2 \to 0.
\]
This gives rise to three natural morphisms from $\mc^{(2)}$ to $\mc$ as follows. 
\[
\begin{aligned}
\begin{aligned}
a_1 : \mc^{(2)} &\to \mc \\
\xi &\mapsto M_1
\end{aligned}
\qquad
\begin{aligned}
b: \mc^{(2)} &\to \mc \\
\xi &\mapsto E
\end{aligned}
\qquad
\begin{aligned}
a_2 : \mc^{(2)} &\to \mc \\
\xi &\mapsto M_2
\end{aligned}
\end{aligned}
\]
This yields the following diagram of morphisms of stacks.
% https://q.uiver.app/#q=WzAsMyxbMCwwLCJcXG1jXnsoMil9Il0sWzAsMiwiXFxtY1xcdGltZXNcXG1jIl0sWzIsMCwiXFxtYyJdLFswLDIsImIiLDAseyJzaG9ydGVuIjp7InNvdXJjZSI6MjAsInRhcmdldCI6MjB9fV0sWzAsMSwiKGFfMSwgYV8yKSIsMix7InNob3J0ZW4iOnsic291cmNlIjoyMCwidGFyZ2V0IjoyMH19XV0=
\[\begin{tikzcd}
	{\mc^{(2)}} && \mc \\
	\\
	{\mc\times\mc,}
	\arrow["b",  from=1-1, to=1-3]
	\arrow["{(a_1, a_2)}"',  from=1-1, to=3-1]
\end{tikzcd}\]
where $(a_1, a_2)$ is of finite type and $b$ is representable and proper. 

As explained in \cite{BridgelandHall}, there is a well-defined Grothendieck ring of stacks $K(St/\mathbb{C})$. 
Given an algebraic stack $S$, one defines the relative Grothendieck group $K(St/S)$, whose elements are classes $[X \to S]$, where $X$ is of finite type over $\mathbb{C}$ and has affine stabilisers.
These generators satisfy the following relations. 
\begin{itemize}
	\item For every pair of stacks $X_1, X_2$ there is a relation 
	\[
	\left[X_1 \coprod X_2 \xrightarrow{f_1 \coprod f_2} S \right] = [X_1 \xrightarrow{f_1} S ] + [X_2 \xrightarrow{f_2} S ]. 
	\]
	\item If $g: X_1 \to X_2$ is a geometric bijection, then for every diagram % https://q.uiver.app/#q=WzAsMyxbMCwwLCJYXzEiXSxbMiwwLCJYXzIiXSxbMSwxLCJTIl0sWzAsMSwiZyJdLFsxLDIsImZfMiJdLFswLDIsImZfMSIsMl1d
\[\begin{tikzcd}
	{X_1} && {X_2} \\
	& S
	\arrow["g", from=1-1, to=1-3]
	\arrow["{f_1}"', from=1-1, to=2-2]
	\arrow["{f_2}", from=1-3, to=2-2]
\end{tikzcd}\] there is a relation 
\[
 [X_1 \xrightarrow{f_1} S ] = [X_2 \xrightarrow{f_2} S ]. 
\]
\item If $[X \xrightarrow{f} S]$ is an object in $K(St/S)$ and $h_1: Y_1 \to X$ and $h_2 : Y_2 \to X$ are two locally trivial fibrations in the Zariski topology with isomorphic fibers, then the following equality holds. 
\[
[Y_1 \xrightarrow{h_1 \circ f} S] = [Y_2 \xrightarrow{h_2 \circ f} S]
\]
\end{itemize}
The motivic Hall algebra $H(A)$ is defined by considering the Grothendieck group $K(St/\mc)$ relative to the universal stack $\mc$ of $\modA$ equipped with the product 
\begin{equation}\label{eq:convolutionHall}
[X_1 \xrightarrow{f_1} \mc]*[X_2 \xrightarrow{f_2} \mc] =  [Z \xrightarrow{b\circ h} \mc]
\end{equation}
where the map $h : Z \to \mc^{(2)}$ is defined by the pullback square.
% https://q.uiver.app/#q=WzAsNSxbMiwwLCJcXG1jXnsoMil9Il0sWzIsMiwiXFxtY1xcdGltZXNcXG1jIl0sWzQsMCwiXFxtYyJdLFswLDAsIloiXSxbMCwyLCJYXzFcXHRpbWVzIFhfMiJdLFswLDIsImIiXSxbMCwxLCIoYV8xLCBhXzIpIiwyXSxbNCwxLCJmXzEgXFx0aW1lcyBmXzIiXSxbMywwLCJoIl0sWzMsNF1d
\[\begin{tikzcd}
	Z && {\mc^{(2)}} && \mc \\
	\\
	{X_1\times X_2} && {\mc\times\mc}
	\arrow["h", from=1-1, to=1-3]
	\arrow[from=1-1, to=3-1]
	\arrow["b", from=1-3, to=1-5]
	\arrow["{(a_1, a_2)}"', from=1-3, to=3-3]
	\arrow["{f_1 \times f_2}", from=3-1, to=3-3]
\end{tikzcd}\]
It has been shown in \cite{Joyce} that there exists a map $\Upsilon : K(St/\mathbb{C}) \to \mathbb{C}(t)$ which endows  $H(A)$ with the structure of a $\mathbb{C}(t)$-algebra.
Moreover, the motivic Hall algebra $H(A)$ is unital and associative.

An element $[X \xrightarrow{f} \mc]\in H(A)$ may be viewed as representing a full subcategory $\xc\subset\modA$. 
Under this interpretation, the Hall product corresponds to taking extensions: 
the class $[Z \xrightarrow{f} \mc]$ appearing above represents the full subcategory consisting of all objects $E$ admitting a short exact sequence
\[
0 \to X_1 \to E \to X_2\to 0
\]
with $X_i \in \xc_i$.
In particular, the unit element $1_{H(A)}$ is represented by the inclusion of the zero object.

%
%In this setting, for every object $[X \xrightarrow{f} \mc]$ we can think of $X$ as a subcategory $\xc$ of $\modA$. 
%Under this light, the subcategory $\mathcal{Z}$ of $\modA$ associated to the element $[Z \xrightarrow{b \circ h} \mc]$ obtained from the multiplication of $[X_1 \xrightarrow{f_1} \mc]$ and $[X_2 \xrightarrow{f_2} \mc]$ is the full category consisting on all the objects in $\modA$ that can be obtained as an extension of an object in $\xc_1$ by an object in $\xc_2$.
%In particular, this implies that the unit $1_{H(A)}$ corresponds to the class $[\{0\} \to \mc]$ of the natural inclusion of $\{0\}$ in $\modA$. 

The universal stack $\mc$ has a stratification as 
\[
\mc = \coprod_{\underline{d} \in \mathbb{Z}^n_{\geq 0}} \mc_{\underline{d}},
\]
where $\mc_{\underline{d}}$ is the stack parametrising all the representations with dimension vector $\underline{d}\in \mathbb{Z}^n_{\geq 0}$.
In particular, we can coarsen this stratification to 
\[
\mc = \coprod_{\ell \in \mathbb{N}} \mc_{\ell},
\]
where $\mc_\ell$ is the stack parametrising all the representations of length $\ell \in \mathbb{N}$.
This stratification of $\mc$ induces a $\mathbb{Z}$-grading 
\[
H(A)=\bigoplus_{i \in \mathbb{N}}H(A)_i
\]
of $H(A)$ satisfying $H(A)_i * H(A)_j \subset H(A)_{i+j}$. 
In particular, we get that 
\[
H(A)_{> \ell }:= \bigoplus_{j\geq \ell} H(A)_j
\]
is a two-sided ideal of $H(A)$ for every positive integer $\ell$. 
We define the truncated Hall algebra $H_\ell(A)$ as the quotient $H_\ell(A) := H(A)/H(A)_{>\ell}$ and the completed motivic Hall algebra $\hat{H}(A)$ as 
\[
\hat{H}(A):= \underleftarrow\lim\, H_\ell(A).
\]
For every object of the form $[\xc \to \mc]\in H(A)$ with $\xc$ a subcategory of $\modA$, we denote by $[\xc \to \mc]_\ell$ the corresponding object in $H_\ell(A)$. 
Moreover, by abuse of notation, we also denote by $[\xc\to \mc]$ the corresponding object in $\hat{H}(A)$.

\smallskip
Every associative algebra carries a natural Lie algebra structure via the commutator bracket.
The associated simply connected Lie group is obtained via the exponential map, whose product is defined using the Baker-Campbell-Hausdorff formula. 

Let $\hf(A)$ denote the Lie algebra associated to $H(A)$ and $\hf(A)_{> l}$ the ideal of $\hf(A)$ corresponding to $H(A)_{>l}$. 
Then the quotient $\hf_{l}(A) := \hf(A) / \hf(A)_{>l}$ is a nilpotent Lie algebra. 
We define the pro-nilpotent Lie algebra $\hat{\hf}(A)$ as 
\[
\hat{\hf}(A):= \underleftarrow\lim\, \hat{\hf}_\ell(A).
\]

Since the Lie algebra $\hf_\ell(A)$ is nilpotent, its associated Lie group $\dot{G}_\ell(A)$ is unipotent. 
The corresponding pro-unipotent Lie group $\dot{G}(A)$ is defined as
\[
\dot{G}(A) := \underleftarrow\lim\, \dot{G}_\ell(A).
\]
As explained in \cite[§5.10]{BridgelandScat}, $\dot{G}(A)$ can be identified with $1_{\hat{H}(A)}+\hat{H}_{>0}(A) \subset \hat{H}(A)$.
Consequently, an object $[\xc \to \mc]$ in $\hat{H}(A)$ is invertible if and only if the zero object belongs to $\xc$. 
In particular, Joyce \cite{Joyce} proved that the object $1_{v}:=[\mod_v^{ss}A \to \mc] \in \hat{H}(A)$ is invertible for every stability condition $v \in \mathbb{R}^n$.

\begin{remark}\label{rmk: faithful C-alg}
Note that the objects $1_v, 1_{v'}$ associated to two stability conditions $v, v' \in \mathbb{R}^n$ coincide if and only if $\mod_v^{ss}A = \mod_{v'}^{ss}A$. 
This implies the existence of a faithful functor $\Gamma: \Cf(A) \to G(A)$ from the $\tau$-cluster morphism category $\Cf(A)$ to the picture group $G(A)$ for every finite-dimensional $\mathbb{C}$-algebra $A$. 
This was proven independently by B\o rve in \cite{Borvefaithful}.
\end{remark}

%%%%%%%%%%%%%%%%%%
\subsection{Stability scattering diagrams and main result}
We now recall the stability scattering diagram introduced in \cite{BridgelandScat}.

\begin{theorem}\cite[Section~6]{BridgelandScat}
Let $Q$ be a quiver and let $A = \mathbb{C}Q/I$ be a finite-dimensional $\mathbb{C}$-algebra. 
There is a minimal consistent scattering diagram $(\Df(A),\hat{\Phi}_A: \Df^1(A) \to \hat{H}(A))$ associated to $A$, where $\Df(A)$ is the cone complex introduced above and $\hat{\Phi}_A: \Df^1(A) \to \hat{H}(A)$ is defined by $\hat{\Phi}(\df) = [\mod_\df^{ss}A \to \mc]$.
\end{theorem}

We can now compare the two scattering diagrams.

\begin{theorem}\label{thm:isomorphism}
Let $A=\mathbb{C}Q/I$ be a finite-dimensional algebra. 
Then the stability scattering diagram $(\Df(A),\hat{\Phi}_A: \Df^1(A) \to \hat{H}(A))$ is isomorphic to the torsion scattering diagram $(\Df(A),\Phi_A: \Df^1(A) \to G(A))$ under the map $I : \Image(\Phi_A) \to \Image(\hat{\Phi}_A)$ defined by $I(\mod_{\df}^{ss}A) = [\mod_\df^{ss}A \to \mc]$.
\end{theorem}

\begin{proof}
By construction, both scattering diagrams are defined on the same cone complex $\Df(A)$.
Moreover, the map $I : \Image(\Phi(A)) \to \Image(\hat{\Phi}(A))$ is bijective by definition.  
It therefore remains to verify that $I$ is compatible with the group operations.
This follows from \cite[Lemma~6.6]{BridgelandScat}, \cref{lem:liftingtogroupoid} and \cref{prop: one map}.
\end{proof}

\begin{remark}
Although the isomorphism is formally straightforward, the stability scattering diagram and the torsion scattering diagram are not a priori isomorphic since they take value in different groups.
\end{remark}

We conclude the paper with the following corollary.

\begin{corollary}\label{cor: finitary isomorphism}
Let $A=\mathbb{C}Q/I$ be a finite-dimensional algebra and let $\ell$ be a positive integer. 
Then the stability scattering diagram $(\Df_\ell(A),\hat{\Phi}^{\ell}_A: \Df_\ell^1(A) \to H_\ell(A))$ is isomorphic to the torsion scattering diagram $(\Df_\ell(A),\Phi^{\ell}_A: \Df^1_\ell(A) \to G_\ell(A))$ via the map \mbox{$I_\ell : \Image(\Phi^\ell(A)) \to \Image(\hat{\Phi}^\ell_A)$} defined by $I_\ell((\mod_{\df}^{ss}A)_\ell) = [\mod_\df^{ss}A \to \mc]_\ell$.
\end{corollary}

\begin{proof}
Recall that there are natural homomorphisms $\pi_\ell : G(A) \to G_\ell(A)$ and $\hat{\pi} : \hat{H}(A) \to H_\ell(A)$ induced by the inverse-limit constructions of $G(A)$ and $\hat{H}(A)$, respectively. 
It is straightforward to verify that the diagram 
\[\begin{tikzcd}
	G(A) & \hat{H}(A) \\
	G_\ell(A) & H_\ell(A)
	\arrow["I", from=1-1, to=1-2]
	\arrow["\pi_\ell", from=1-1, to=2-1]
	%\arrow["b", from=1-3, to=1-5]
	\arrow["\hat{\pi}_\ell"', from=1-2, to=2-2]
	\arrow["I_\ell", from=2-1, to=2-2]
\end{tikzcd}\]
commutes and the result follows.
\end{proof}

\end{document}